\documentclass[]{article}
\usepackage[utf8]{inputenc}

\usepackage{standalone}
\usepackage{amsmath}
\usepackage{amssymb}
\usepackage{amsfonts}
\usepackage{amsthm}
\usepackage{tikz}
\usetikzlibrary{positioning}
\usetikzlibrary{decorations.pathreplacing}
\usepackage{multicol}
\usepackage{theoremref}
\usepackage{amsrefs}
\usepackage{enumitem}
\usepackage{marginnote}
\usepackage{comment}
\usepackage{graphicx}
\usepackage{theoremref}
\usepackage{amsrefs}

\newcommand{\M}{M}
\newcommand{\p}{P}

\newcommand{\f}{f}
\DeclareMathOperator{\wt}{wt}
\newcommand{\xequiv}[1]{\ \rlap{$\equiv$}{\raisebox{.55em}{\tiny \ $#1$}}\ }
\newcommand{\aequiv}{\xequiv{a}}

\newcommand\numberthis{
\addtocounter{equation}{1}\tag{\theequation}
}

\newcounter{fact}
\newtheorem{theorem}[fact]{Theorem}
\newtheorem{proposition}[fact]{Proposition}
\newtheorem{lemma}[fact]{Lemma}
\newtheorem{corollary}[fact]{Corollary}

\newcommand{\alf}{{\ensuremath{\alpha}}}
\newcommand{\bet}{{\ensuremath{\beta}}}
\newcommand{\gam}{{\ensuremath{\gamma}}}
\newcommand{\del}{{\ensuremath{\delta}}}

\newcommand{\hex}[4]{
	\pgfmathtruncatemacro{\a}{#1}
	\pgfmathtruncatemacro{\b}{#2}
	\pgfmathtruncatemacro{\c}{#3}
	\pgfmathtruncatemacro{\t}{#4}
    
	\pgfmathtruncatemacro{\longdiag}{\a+\b+\c+\t+\t}
	
	\draw[clip] (\a+\t,0,\t) -- ++(0,\b+\t,0) -- ++(-\a,0,0) -- ++(0,0,\c+\t) -- ++(0,-\b,0) -- ++(\a+\t,0,0) -- cycle;

	\foreach \i in {0,...,\longdiag}{
        \draw[-] (\i,0,\c+\t) -- ++(0,\longdiag,0);
        \draw[-] (\a+\t,\i,0) -- ++(0,0,\longdiag);
        \draw[-] (0,\b+\t,\i) -- ++(\longdiag,0,0);
    }
}

\newcommand{\labelnote}[1]{\label{#1}}%\marginnote{#1}}

\newcommand{\vdent}[3]{\draw[fill=blue!5](#1,#2,#3) --++(0,0,1) --++(0,-1,0) --++(0,0,-1)-- cycle; }
\newcommand{\ldent}[3]{\draw[fill=green!10](#1,#2,#3) --++(1,0,0) --++(0,-1,0) --++(-1,0,0)-- cycle; }
\newcommand{\rdent}[3]{\draw[fill=red!10](#1,#2,#3) --++(1,0,0) --++(0,0,1) --++(-1,0,0)-- cycle; }

\usepackage{titling}

\title{Simple Relationships Between Lozenge Tiling Functions of Related Regions }
\author{Daniel Condon }
\date{\today}

\begin{document}

\maketitle
\thispagestyle{empty}

\begin{abstract}
    We give a formula for the number of symmetric tilings of hexagons on the triangular lattice with unit triangles removed from arbitrary positions along two non-adjacent non-opposite sides. We show that for certain families of such regions, the ratios of their numbers of symmetric tilings are given by simple product formulas. We also prove that for certain weighted regions which arise when applying Ciucu’s Factorization Theorem, the formulas for the weighted and unweighted counts of tilings have a simple explicit relationship.
    
\end{abstract}

\section{Introduction}

The triangular lattice induces a tiling of the plane with unit equilateral triangles so that some of the lattice lines are horizontal. A \textbf{region} on the triangular lattice is a connected bounded union of these triangles.\footnote{Henceforth we will say ``triangle'' to mean a unit equilateral triangle on the triangular lattice.} We say triangles which share an edge are \textbf{adjacent}. A \textbf{lozenge} is the union of two adjacent triangles, and a \textbf{lozenge tiling}\footnote{Henceforth we will simply say ``tiling'' to mean a {lozenge tiling}.} of a region is a covering of that region with lozenges within that region so that pairwise the lozenges do not overlap except on their boundaries.

We may associate a \textbf{weight function} to a region which assigns a weight to each lozenge within that region. This induces a weight on each tiling of the region, which is the product of the weights of the lozenges in that tiling. We say a region is \textbf{unweighted} if each lozenge (and therefore each tiling) has weight 1. A region with no specified weight function is assumed to be unweighted. We consider the weight function a feature of the region.

Given a region $R$ we let $\M(R)$ denote the sum of weights of tilings of $R$; thus if $R$ is unweighted, $\M(R)$ simply counts the tilings of $R$. We use the letter $\M$ because a natural bijection exists between tilings of a region on the triangular lattice and perfect \textbf{m}atchings on a corresponding graph, namely the region's planar dual. Kasteleyn's paper \cite{Ka}, and Temperley and Fisher's paper \cite{TeFi}, are an excellent introduction to matching theory.

Given a family of regions $R_{\bar x}$ indexed by parameters $\bar x$, the function $f(\bar x):=\M(R_{\bar x})$ is called the \textbf{tiling function} of that family.
For example, let $H_{a,b,c}$ denote a region which is a semiregular hexagon, with horizontal sides of length $a$, sides immediately clockwise from these of length $b$, and sides immediately clockwise from these of length $c$. Then
\begin{equation}
  \p(a,b,c)  := \M(H_{a,b,c}) = \prod_{i=1}^a \prod_{j=1}^b \prod_{k=1}^c \frac{i+j+k-1}{i+j+k-2}
\end{equation}
and we say that $\p(a,b,c)$ is the tiling function for semiregular hexagons. This result is by MacMahon \cite{Ma} who was studying lozenge tilings from the perspective of \textbf{p}lane \textbf{p}artitions, hence our use of the letter $\p$. Figure \ref{fig:MacMahon} depicts an example of a tiling of one such hexagon.
The simple elegance of MacMahon's formula has inspired a search for other families of regions with tiling functions given by simple product formulas.

\begin{figure}
    \begin{minipage}[c]{\textwidth}
    \centering
    \begin{multicols}{2}
    \begin{tikzpicture}[x={(0:1cm)},y={(120:1cm)},z={(240:1cm)},scale=.45]
	\hex{6}{4}{5}{0}
    \end{tikzpicture}
    
    \begin{tikzpicture}[x={(0:1cm)},y={(120:1cm)},z={(240:1cm)},scale=.45]
    \draw[thick,fill=red!10] (2,0,-4) --++ (-6,0,0) --++(0,0,5)--++(0,-4,0)--++(6,0,0)--++(0,0,-5)--++(0,4,0);

	\clip (2,0,-4) --++ (-6,0,0) --++(0,0,5)--++(0,-4,0)--++(6,0,0)--++(0,0,-5)--++(0,4,0);
	
	%%Red
	\begin{scope}
	    \clip(2,0,-4)--++ (-6,0,0) --++(0,0,5)--++(2,0,0)--++(0,0,-1)--++(1,0,0)--++(0,0,-1)--++(3,0,0)--cycle;
	    
	    \foreach \x in {-3,-2,-1,0,1,2}{
	        \draw[-] (\x,0,1)--++(0,0,-5);
	    }
	    
	    \foreach \z in {-3,-2,-1,0,1}{
	        \draw[-](-4,0,\z) --++(6,0,0);
	    }
	\end{scope}
	
	%%Green
	\fill[green!10] (-4,0,1) --++(0,-4,0)--++(6,0,0) --++ (0,2,0) --++(-2,0,0)--++(0,1,0)--++(-2,0,0)--++(0,1,0)-- cycle;
	
	\begin{scope}
	 \clip (-4,0,1) --++(0,-4,0)--++(6,0,0) --++ (0,2,0) --++(-2,0,0)--++(0,1,0)--++(-2,0,0)--++(0,1,0)-- cycle;
	 
	\foreach \z in {1,2,3,4}{
	\draw[-] (-4,0,\z)--++(10,0,0);
	}
	
    \foreach \x in {1,2,3,4,5,6}{
    \draw[-] (\x,0,5)--++(0,4,0);
    }
	\end{scope}

	%%Blue
	\fill[blue!5] (2,0,-4)--++(0,0,3)--++(0,-1,0)--++(0,0,1)--++(0,-1,0)--++(0,0,1)--++(0,-2,0)--++(0,0,-5)--cycle;
	\begin{scope}
	\clip (2,0,-4)--++(0,0,3)--++(0,-1,0)--++(0,0,1)--++(0,-1,0)--++(0,0,1)--++(0,-2,0)--++(0,0,-5)--cycle;
	
	\foreach \z in {1,2,3,4,5} {
	\draw[-] (6,0,\z)--++(0,4,0);
	}
	
	\foreach \x in {1,2,3,4,5}{
	\draw[-](\x,0,5)--++(0,0,-10);
	}
	\end{scope}
	
	\vdent010\vdent{-2}00\vdent0{-1}0
    \ldent{-2}00\ldent010\ldent110\ldent210\ldent0{-1}0\ldent1{-1}0
	\rdent000\rdent001\rdent{-1}01\rdent100\rdent200\rdent0{-2}0\rdent1{-2}0
	
	\draw[thick] (2,0,-4) --++ (-6,0,0) --++(0,0,5)--++(0,-4,0)--++(6,0,0)--++(0,0,-5)--++(0,4,0);
	
    \end{tikzpicture}
    \end{multicols}
    \caption{The region $H_{6,4,5}$ and a tiling of that region. Tilings of semiregular hexagons have a natural bijection with plane partitions, namely the tilings of $H_{a,b,c}$ depict the plane partitions within a box of dimension $a \times b \times c$. An appropriate weight function for $H_{a,b,c}$ yields a tiling function that enumerates plane partitions in the box by volume.}
    \label{fig:MacMahon}
    \end{minipage}
\end{figure}
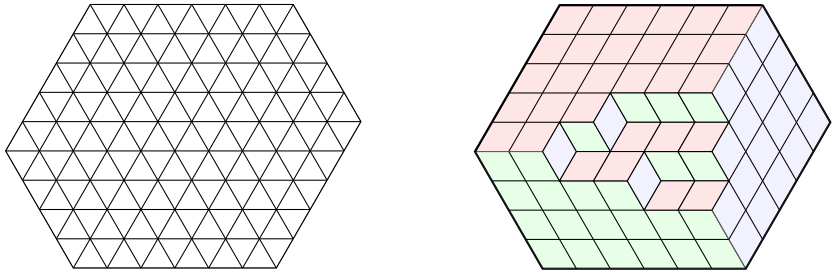

Note that since $H_{a,b,c}$ is congruent to the hexagon with the parameters $a,b,c$ given in any order, $\p(a,b,c)$ is symmetric on its parameters; nonetheless, being strict about the orientation of our figures will be helpful as they become more complex. For example, we will restrict our study of figures which are symmetric across a lattice line to just those which are symmetric across a horizontal lattice line, and we will say these are \textbf{horizontally symmetric}. Similarly, we will study figures that are symmetric across a vertical line, calling these \textbf{vertically symmetric}. When $R$ is a region with at least one of these symmetries, we will let $\M_{-}(R)$ give the weighted count of horizontally symmetric tilings of $R$, and $\M_{|}(R)$ denote the weighted count of vertically symmetric tilings of $R$. Since $H_{a,b,b}$ is both horizontally and vertically symmetric, we will denote:
\begin{align}
    \p_{-}(a,b):= \M_{-}(H_{2a,b,b}) =& \binom{2a+b}{b}\prod_{1 \leq i <j \leq b} \dfrac{2a+i+j}{i+j}\\
    \p_{|}(a,b):= \M_{|}(H_{2a,b,b}) = & \dfrac{(a+1)_{b-1}}{(2a+1)_{b-1}}\prod_{1 \leq i \leq j \leq b-1} \dfrac{2a+i+j-1}{i+j-1}.
\end{align}
The formula for $\p_{-}(a,b)$ was proven independently in various forms by Andrews \cite{An78}, Gordon \cite{Go83}, Macdonald \cite{MD}, and Proctor \cite{Pr84}. The formula for $\p_{|}(a,b)$ is due to Proctor \cite{Pr88}. Notably, $$P_{-}(a,b) \cdot P_{|}(a,b) = P(a,b,b),$$ an observation by Ciucu, who noticed the relationship after these three formulas were known. So far there is no known bijective proof of this fact.

Recently a new type of result has gained prominence; identifying simple relationships between tiling functions which themselves might not have a known simple expression. This paper will explore relationships such as these. Descriptions of the regions we will discuss follow.

We use $H_{a,b,c,t}$ to denote the equiangular hexagon(al region) with side lengths $a, b+t, c, a+t, b, c+t$ clockwise starting from the northern side. For $(u_i)_1^m, (v_j)_1^n$  integer vectors with $1 \leq u_i < u_{i+1} \leq b+t$ and $1 \leq v_i < v_{i+1} \leq c+t$, we use $H_{a,b,c,t,(u_i)_1^m, (v_j)_1^n}$ to denote the region obtained by modifying $H_{a,b,c,t}$ as follows: remove each $u_i$th unit triangle along the northeast border, counting those triangles which share an edge with the border and starting from the north; remove also each $v_j$th unit triangle along the northwest border, counting those triangles which share an edge with the border and starting from the north. We call regions of this kind \textbf{dented hexagons}, and an example of such a region is given in the top left of Figure \ref{fig:DentedHex}. We discuss these regions at length in \cite{Co}. Regions constructed by removing dents from hexagons have been studied before. In \cite{Ei} Eisenk{\"o}lbl gave a tiling function for a region constructed from $H_{a,b,c,1}$ by removing a dent from each of three alternating sides, and in \cite{Ci96} Ciucu and Fischer give a tiling function for hexagons with dents removed in arbitrary positions along the border, presented as the Pfaffian of a matrix.  In \cite{Gi} Gilmore gives a tiling function for regions with dents removed from arbitrary positions along the border and within the interior.
\begin{figure}
    \begin{minipage}[c]{\textwidth}
    \centering
    \begin{multicols}{2}

    \begin{tikzpicture}[x={(0:.45cm)},y={(120:.45cm)},z={(240:.45cm)}]
	\pgfmathtruncatemacro{\a}{4}
	\pgfmathtruncatemacro{\b}{2}
	\pgfmathtruncatemacro{\c}{3}
	\pgfmathtruncatemacro{\t}{4}
    
	\pgfmathtruncatemacro{\longdiag}{\a+\b+\c+\t+\t}

	\draw[dashed] (0,.5,-.5) --++(0,-6,6);	

	\draw[clip] (0,0,0)  --++(-2,0,0) -- ++(0,0,2)
	-- ++ (0,-1,0) -- ++(-1,0,0) -- ++ (0,0,1)
	-- ++ (0,-1,0) -- ++(-1,0,0) -- ++ (0,0,2) -- ++ (0,-3,0) --++ (4,0,0) 
	--++(4,0,0) --++(0,0,-3) --++(0,2,0) --++ (-1,0,0) --++(0,0,-1) --++ (0,1,0) 
	--++ (-1,0,0) --++(0,0,-1) --++ (0,2,0) -- cycle;

	\foreach \i in {0,...,\longdiag}{
        \draw[-] (2,-\i,0) -- ++(-\longdiag,0,0);
        \draw[-] (2-\i,0,0) -- ++(0,-\longdiag,0);
        \draw[-] (-2+\i,0,0) -- ++(0,0,\longdiag);
    }

    \end{tikzpicture}

	$H_{4,3,3,4,(3,5),(3,5)}$

    \columnbreak
        
    \begin{tikzpicture}[x={(0:.45cm)},y={(120:.45cm)},z={(240:.45cm)}]
	\pgfmathtruncatemacro{\a}{4}
	\pgfmathtruncatemacro{\b}{2}
	\pgfmathtruncatemacro{\c}{3}
	\pgfmathtruncatemacro{\t}{4}
	\pgfmathtruncatemacro{\longdiag}{\a+\b+\c+\t+\t}

	\draw[dashed] (0,.5,-.5) --++(0,-6,6);	

	\draw[dotted] (0,-5,5) --++(4,0,0) --++(0,0,-3) --++(0,2,0) --++ (-1,0,0) --++(0,0,-1) --++ (0,1,0) 
	--++ (-1,0,0) --++(0,0,-1) --++ (0,2,0) -- ++(-2,0,0);
	
	\draw[clip] (0,0,0)  --++(-2,0,0) -- ++(0,0,2)
	-- ++ (0,-1,0) -- ++(-1,0,0) -- ++ (0,0,1)
	-- ++ (0,-1,0) -- ++(-1,0,0) -- ++ (0,0,2) -- ++ (0,-3,0) --++ (4,0,0)
	--++(0,1,0)
	--++(0,0,-1) --++(0,1,0)
	--++(0,0,-1) --++(0,1,0)
	--++(0,0,-1) --++(0,1,0) --++(0,0,-1)  --++(0,1,0)-- cycle;

	\foreach \i in {0,...,\longdiag}{
        \draw[-] (2,-\i,0) -- ++(-\longdiag,0,0);
        \draw[-] (2-\i,0,0) -- ++(0,-\longdiag,0);
        \draw[-] (-2+\i,0,0) -- ++(0,0,\longdiag);
    }
    \end{tikzpicture}

	$V_{2,3,2,(3,5)}$
    
    \end{multicols}

    \begin{multicols}{2}

    \begin{tikzpicture}[x={(0:.45cm)},y={(120:.45cm)},z={(240:.45cm)}]
	\pgfmathtruncatemacro{\a}{4}
	\pgfmathtruncatemacro{\b}{2}
	\pgfmathtruncatemacro{\c}{3}
	\pgfmathtruncatemacro{\t}{4}
	\pgfmathtruncatemacro{\longdiag}{\a+\b+\c+\t+\t}

	\draw[dashed] (0,.5,-.5) --++(0,-6,6);	

	\draw[dotted] (0,-5,5) --++(4,0,0) --++(0,0,-3) --++(0,2,0) --++ (-1,0,0) --++(0,0,-1) --++ (0,1,0) 
	--++ (-1,0,0) --++(0,0,-1) --++ (0,2,0) -- ++(-2,0,0);
	
	\draw[clip] (0,0,0)  --++(-2,0,0) -- ++(0,0,2)
	-- ++ (0,-1,0) -- ++(-1,0,0) -- ++ (0,0,1)
	-- ++ (0,-1,0) -- ++(-1,0,0) -- ++ (0,0,2) -- ++ (0,-3,0) --++ (4,0,0)
	--++(0,0,-1) --++(0,1,0)
	--++(0,0,-1) --++(0,1,0)
	--++(0,0,-1) --++(0,1,0)
	--++(0,0,-1) --++(0,1,0) --++(0,0,-1) -- cycle;

	\foreach \i in {0,...,\longdiag}{
        \draw[-] (2,-\i,0) -- ++(-\longdiag,0,0);
        \draw[-] (2-\i,0,0) -- ++(0,-\longdiag,0);
        \draw[-] (-2+\i,0,0) -- ++(0,0,\longdiag);
    }
    \end{tikzpicture}

	$V^+_{2,3,2,(3,5)}$

    \columnbreak
        
    \begin{tikzpicture}[x={(0:.45cm)},y={(120:.45cm)},z={(240:.45cm)}]
	\pgfmathtruncatemacro{\a}{4}
	\pgfmathtruncatemacro{\b}{2}
	\pgfmathtruncatemacro{\c}{3}
	\pgfmathtruncatemacro{\t}{4}
	\pgfmathtruncatemacro{\longdiag}{\a+\b+\c+\t+\t}

	\draw[dashed] (0,.5,-.5) --++(0,-6,6);	

	\draw[dotted] (0,-5,5) --++(4,0,0) --++(0,0,-3) --++(0,2,0) --++ (-1,0,0) --++(0,0,-1) --++ (0,1,0) 
	--++ (-1,0,0) --++(0,0,-1) --++ (0,2,0) -- ++(-2,0,0);
	
	\draw[clip] (0,0,0)  --++(-2,0,0) -- ++(0,0,2)
	-- ++ (0,-1,0) -- ++(-1,0,0) -- ++ (0,0,1)
	-- ++ (0,-1,0) -- ++(-1,0,0) -- ++ (0,0,2) -- ++ (0,-3,0) --++ (4,0,0)
	--++(0,0,-1) --++(0,1,0)
	--++(0,0,-1) --++(0,1,0)
	--++(0,0,-1) --++(0,1,0)
	--++(0,0,-1) --++(0,1,0) --++(0,0,-1) -- cycle;

	\foreach \i in {0,...,\longdiag}{
        \draw[-] (2,-\i,0) -- ++(-\longdiag,0,0);
        \draw[-] (2-\i,0,0) -- ++(0,-\longdiag,0);
        \draw[-] (-2+\i,0,0) -- ++(0,0,\longdiag);
	}

	\foreach \i in {1,3,5,7,9}{
	\fill[black, fill opacity=.2] (.5*\i,0,\i) ellipse (.1cm and .2cm);
	}
    \end{tikzpicture}

	$\overline V^+_{2,3,2,(3,5)}$
    
    \end{multicols}
    \vspace{-.5cm}
    \caption{The vertically symmetric dented hexagon $H_{4,3,3,4,(3,5),(3,5)}$, with its axis of symmetry depicted, and its various subregions: $V$, $V^+$, and $\overline V^+$. The ovals in the depiction of $\overline V^+$ represent that the lozenges they cover have weight 1/2. In each of these regions, $\underline u_1 = 3$ and $\underline u_2=2$. 
    }
    \label{fig:DentedHex}
    \end{minipage}
\end{figure}
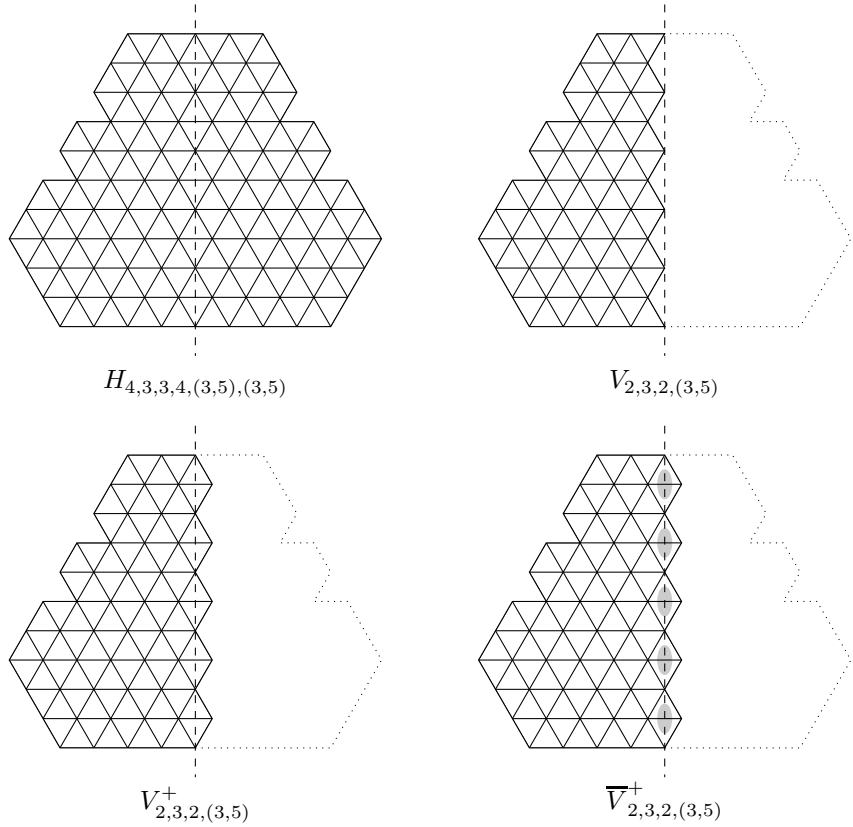

Consider a vertically symmetric dented hexagon $H:=H_{2a,b,b,2t,(u_i)_1^n,(u_i)_1^n}$. We use $V_{a,b,t,(u_i)_1^n}$ to denote the region consisting of triangles within $H$ with interiors strictly west of its vertical axis of symmetry. We use the letter $V$ because tilings of this region exactly correspond to the vertically symmetric tilings of $H$. We call these \textbf{dented half-hexagons}.
Similarly, we will use $V^+_{a,b,t,(u_i)_1^n}$ to denote the region consisting of triangles within $H$ with interiors weakly west of its vertical axis of symmetry; note $V^+$ is congruent to the complementary subregion of $V$ within $H$. We will also use $\overline V^+_{a,b,t,(u_i)_1^n}$ to denote this region with a special weighting of lozenges: those vertical\footnote{Meaning vertically symmetric.} lozenges on the symmetry axis of $H$ are assigned weight 1/2. We will also use an overline to denote an analogous weighting scheme for other regions. An example of these three regions appears in Figure \ref{fig:DentedHex}. 

A useful statistic of these regions is the distance from a dent to the corner of the hexagon beneath it minus the number of other dents between that dent and that corner: for the dent indexed by $u_i$, this statistic is $\underline{u_i}:= b+n+i-u_i$.

\begin{figure}
    \centering
	\begin{multicols}{2}
    \begin{tikzpicture}[x={(0:1cm)},y={(120:1cm)},z={(240:1cm)},scale=.45]
	\pgfmathtruncatemacro{\a}{5}
	\pgfmathtruncatemacro{\b}{4}
	\pgfmathtruncatemacro{\c}{4}
	\pgfmathtruncatemacro{\t}{1}

	\draw[dashed] (.5,0,-5) --++ (0,-5,5);

	\pgfmathtruncatemacro{\longdiag}{\a+\b+\c+\t+\t}
	
	\begin{scope}[x={(180:1cm)},y={(60:1cm)},z={(300:1cm)},xshift=6cm]
	\draw[dashed] (1,0,-4) --++
	(-4,0,0) --++ (0,0,1) --++(-1,0,0) --++(0,0,1) --++(0,-1,0) --++(1,0,0) --++ (0,-1,0) --++ (0,0,2) --++(0,-1,0) --++(0,0,1) --++(0,-1,0) --++(1,0,0) --++(0,0,-1) --++(3,0,0);

	\end{scope}
	
	%\draw[very thick] (-3,0,-4) --++ (0,0,1) --++(-1,0,0) --++(0,0,1) --++(0,-1,0) --++(1,0,0) --++ (0,-1,0) --++ (0,0,2) --++(0,-1,0) --++(0,0,1) --++(0,-1,0) --++(1,0,0) --++(0,0,-1);
	
	\draw[very thick] (1,0,-4) --++ (0,0,1) --++(0,-1,0)  --++ (0,0,1) --++(0,-1,0)  --++ (0,0,1) --++(0,-1,0)  --++ (0,0,1) --++(0,-1,0);
	
	\draw[clip] (1,0,-4) --++
	(-4,0,0) --++ (0,0,1) --++(-1,0,0) --++(0,0,1) --++(0,-1,0) --++(1,0,0) --++ (0,-1,0) --++ (0,0,2) --++(0,-1,0) --++(0,0,1) --++(0,-1,0) --++(1,0,0) --++(0,0,-1) --++(3,0,0)
	--++ (0,1,0) --++(0,0,-1)
	--++ (0,1,0) --++(0,0,-1)
	--++ (0,1,0) --++(0,0,-1)
	--++ (0,1,0) -- cycle;

	\foreach \i in {0,...,\longdiag}{
        \draw[-] (\i,0,\c+\t) -- ++(0,\longdiag,0);
        \draw[-] (\a+\t,\i,0) -- ++(0,0,\longdiag);
        \draw[-] (0,\b+\t,\i) -- ++(\longdiag,0,0);
    }

    \end{tikzpicture}

	 $R$

	\columnbreak

    \begin{tikzpicture}[x={(0:1cm)},y={(120:1cm)},z={(240:1cm)},scale=.45]
	\pgfmathtruncatemacro{\a}{3}
    
	\pgfmathtruncatemacro{\longdiag}{8}
	\draw[dashed,white] (2,1,0) --++(0,-5,5);
	
	\begin{scope}[x={(0:1cm)},y={(120:1cm)},z={(240:1cm)},xshift=0cm]
	\draw[clip] (0,0,0) --++(-\a,0,0)
	 --++(0,0,1)--++(0,-1,0)
	 --++(0,0,1)--++(0,-1,0)
	 --++(0,0,1)--++(0,-1,0)
	 --++(0,0,1)--++(0,-1,0)
	--++(\a,0,0) 
	--++(0,1,0)--++(0,0,-1)
	--++(0,1,0)--++(0,0,-1)
	--++(0,1,0)--++(0,0,-1)
	--++(0,1,0)--++(0,0,-1) -- cycle;

	\foreach \i in {0,...,\longdiag}{
        \draw[-] (-\i*.5,-\i,0) -- ++(-\longdiag,0,0);
	\draw[-] (-5+\i,0,0) --++(0,0,\longdiag);
	\draw[-] (-\i,0,0)--++(0,-\longdiag,0);
	}
	\end{scope}

	\begin{scope}[x={(0:1cm)},y={(120:1cm)},z={(240:1cm)},xshift=6cm]
	\draw[clip] (0,0,0) --++(-\a,0,0)
	 --++(0,0,1)--++(0,-1,0)
	 --++(0,0,1)--++(0,-1,0)
	 --++(0,0,1)--++(0,-1,0)
	 --++(0,0,1)--++(0,-1,0)
	--++(\a,0,0) --++(0,0,-1)
	--++(0,1,0)--++(0,0,-1)
	--++(0,1,0)--++(0,0,-1)
	--++(0,1,0)--++(0,0,-1)
	--++(0,1,0) -- cycle;

	\foreach \i in {0,...,\longdiag}{
        \draw[-] (.5-\i*.5,-\i,0) -- ++(-\longdiag,0,0);
	\draw[-] (-5+\i,0,0) --++(0,0,\longdiag);
	\draw[-] (-\i,0,0)--++(0,-\longdiag,0);
	}
	\end{scope}

    \end{tikzpicture}

	$T(3,4)$ and $T(3.5,4)$
	\end{multicols}

	\begin{multicols}{2}
    \begin{tikzpicture}[x={(0:1cm)},y={(120:1cm)},z={(240:1cm)},scale=.45]
	\pgfmathtruncatemacro{\a}{5}
	\pgfmathtruncatemacro{\b}{4}
	\pgfmathtruncatemacro{\c}{4}
	\pgfmathtruncatemacro{\t}{1}

	\draw[white] (.5,0,-5) --++ (0,-5,5);

	\pgfmathtruncatemacro{\longdiag}{\a+\b+\c+\t+\t+2}
	
	\begin{scope}[x={(180:1cm)},y={(60:1cm)},z={(300:1cm)},xshift=6cm]
	\draw[white] (1,0,-4) --++
	(-4,0,0) --++ (0,0,1) --++(-1,0,0) --++(0,0,1) --++(0,-1,0) --++(1,0,0) --++ (0,-1,0) --++ (0,0,2) --++(0,-1,0) --++(0,0,1) --++(0,-1,0) --++(1,0,0) --++(0,0,-1) --++(3,0,0);

	\end{scope}

	\draw[very thick] (1,0,-4) --++ (0,0,1) --++(0,-1,0)  --++ (0,0,1) --++(0,-1,0)  --++ (0,0,1) --++(0,-1,0)  --++ (0,0,1) --++(0,-1,0);
	
	\draw[clip] (4,0,-4) --++
	(-7,0,0) --++ (0,0,1) --++(-1,0,0) --++(0,0,1) --++(0,-1,0) --++(1,0,0) --++ (0,-1,0) --++ (0,0,2) --++(0,-1,0) --++(0,0,1) --++(0,-1,0) --++(1,0,0) --++(0,0,-1) --++(6,0,0)
	--++ (0,1,0) --++(0,0,-1)
	--++ (0,1,0) --++(0,0,-1)
	--++ (0,1,0) --++(0,0,-1)
	--++ (0,1,0) -- cycle;

	\foreach \i in {0,...,\longdiag}{
        \draw[-] (\i,0,\c+\t) -- ++(0,\longdiag,0);
        \draw[-] (\a+\t+1,\i-1,0) -- ++(0,0,\longdiag);
        \draw[-] (0,\b+\t,\i) -- ++(\longdiag,0,0);
    }

    \end{tikzpicture}

	$R_{z}(3)$

\begin{tikzpicture}[x={(0:1cm)},y={(120:1cm)},z={(240:1cm)},scale=.45]
	\pgfmathtruncatemacro{\a}{5}
	\pgfmathtruncatemacro{\b}{4}
	\pgfmathtruncatemacro{\c}{4}
	\pgfmathtruncatemacro{\t}{1}

	\draw[dashed,white] (.5,0,-5) --++ (0,-5,5);

	\pgfmathtruncatemacro{\longdiag}{\a+\b+\c+\t+\t+2}
	
	\begin{scope}[x={(180:1cm)},y={(60:1cm)},z={(300:1cm)},xshift=6cm]
	\draw[white] (1,0,-4) --++
	(-4,0,0) --++ (0,0,1) --++(-1,0,0) --++(0,0,1) --++(0,-1,0) --++(1,0,0) --++ (0,-1,0) --++ (0,0,2) --++(0,-1,0) --++(0,0,1) --++(0,-1,0) --++(1,0,0) --++(0,0,-1) --++(3,0,0);

	\end{scope}

	\draw[very thick] (1,0,-4) --++ (0,0,1) --++(0,-1,0)  --++ (0,0,1) --++(0,-1,0)  --++ (0,0,1) --++(0,-1,0)  --++ (0,0,1) --++(0,-1,0);
	
	\draw[clip] (4,0,-4) --++
	(-7,0,0) --++ (0,0,1) --++(-1,0,0) --++(0,0,1) --++(0,-1,0) --++(1,0,0) --++ (0,-1,0) --++ (0,0,2) --++(0,-1,0) --++(0,0,1) --++(0,-1,0) --++(1,0,0) --++(0,0,-1) --++(6,0,0) --++(0,0,-1)
	--++ (0,1,0) --++(0,0,-1)
	--++ (0,1,0) --++(0,0,-1)
	--++ (0,1,0) --++(0,0,-1)
	--++ (0,1,0) -- cycle;

	\foreach \i in {0,...,\longdiag}{
        \draw[-] (\i,0,\c+\t) -- ++(0,\longdiag,0);
        \draw[-] (\a+\t+2,\i-1,0) -- ++(0,0,\longdiag);
        \draw[-] (0,\b+\t,\i) -- ++(\longdiag,0,0);
    }

	\foreach \i in {-3,-1,1,3}{
	\fill[black, fill opacity=.2] (6+.5*\i,0,\i) ellipse (.2cm and .4cm);
	}
    \end{tikzpicture}

	$\overline R_{z}(3.5)$
	\end{multicols}
    \caption{Top Left: a region $R$ depicted within its symmetrization. Its eastern boundary is the same shape as the western boundary of a tube of height 4. Top Right: two tubes of height 4. Bottom: tubey regions constructed from $R$, where $z$ is the eastern boundary of $R$. Note that these constructions are not unique.}
    \label{fig:pathRegion}
\end{figure}
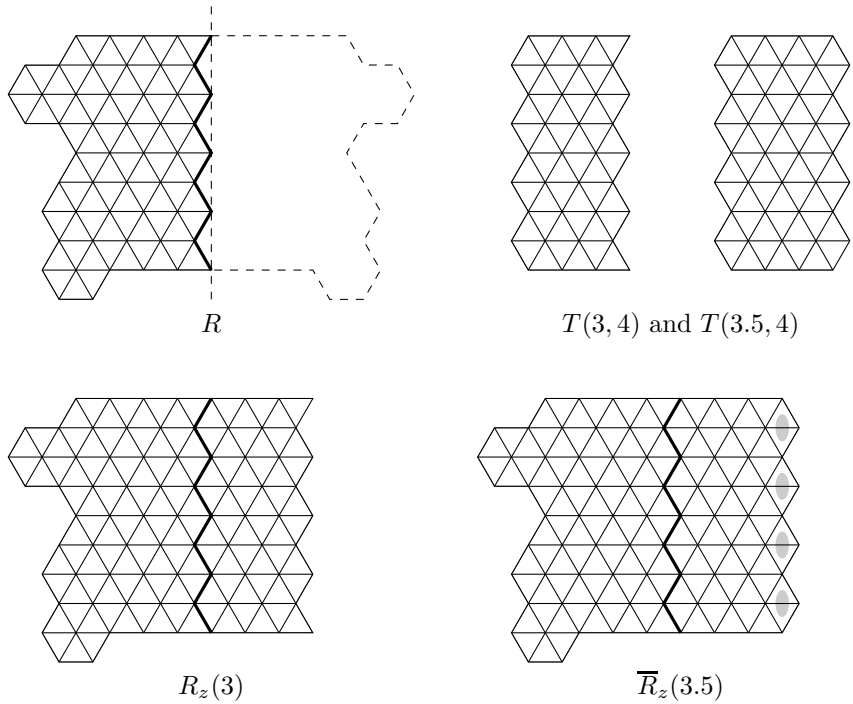

For $a \in \mathbb N/2$ and $h \in \mathbb Z$, Let $T(a,h)$ denote the region of triangles within the rectangle consisting of line segments of horizontal lattice lines which are $a+1/2$ units long, and vertical line segments which are $h\sqrt{3}$ units long,\footnote{Since adjacent parallel lattice lines are $\frac{\sqrt{3}}{2}$ units apart, this rectangle is $2h$ times the height of a triangle.} so that the top left corner of the rectangle is in the middle of the base of a triangle. We call this the \textbf{tube} of length $a$ and height $h$. Two tubes are depicted in the top right of Figure \ref{fig:pathRegion}.

Consider a set of tubes $\mathcal T = \{T(a,h): a \in \mathbb N\}$ which share a fixed western boundary which is denoted $z$. Let $R$ be a region with a fixed relative position to $z$ so that $R$ does not intersect any tube in $\mathcal T$ except along $z$. Then we say $R_z(0):=R$ and $R_z(a)$ is the union of $R$ and the tube $T(a,h) \in \mathcal T$. We call a region (or family of regions) which is constructed in this way \textbf{tubey}. Two examples are depicted in the bottom of Figure \ref{fig:pathRegion}.

It is natural to describe a tubey region in terms of $R$ and $z$. For example, $\{V_{a,b,n,(u_i)_1^n}: a \in \mathbb N\}$ is $R_{z}(a)$, and  $\{V^+_{a,b,n,(u_i)_1^n}: a \in \mathbb N\}$ is $R_{z}(a+1/2)$, where $R$ is $V_{0,b,n,(u_i)_1^n}$ and $z$ is the eastern boundary of $R$ extended by two edges. Tubey families thus generalize families of dented half-hexagons, and are characterized by ``stretching.'' When $a$ is sufficiently large that the tube extends east of $R$, the tubey regions $R_{z}(a)$ and $R_{z}(a+1/2)$ are congruent to complementary halves of a vertically symmetric region.

We will discuss tubey regions with an unusual weighting. We use the notation $\overline R_{z}(a)$ to denote the tubey region $R_{z}(a)$ where the vertical lozenges along\footnote{``Along'' here meaning sharing two edges with.} the eastern boundary of the tube are assigned weight $\frac 12$. Observe this notation is consistent with the notation for the weighted region $\overline V^+_{a,b,n,(u_i)_1^n}$ interpreted as a tubey region. This weighting scheme is natural for certain symmetrizable regions which arise when applying Ciucu's Factorization Theorem, which we discuss in Section \ref{background}.

The results of this paper yield simple product formulas for the ratios of the tiling functions of $V_{a,b,n,(u_i)_1^n}$ and $V_{0,b,n,(u_i)_1^n}$, as well as  $V^+_{a,b,n,(u_i)_1^n}$ and $V^+_{0,b,n,(u_i)_1^n}$. 
This is the second paper by this author and one of many in the field which identify simple relationships between related tiling functions which may not themselves be simple. Similar results have been given by Lai, Ciucu, Rohatgi, and Byun \cite{By19}, \cite{Ci19a}, \cite{Ci19b}, \cite{La19a}, \cite{La19b}, \cite{La19c}.
 In fact, in \cite{La21}, Lai independently studies dented hexagons and dented half-hexagons with some equivalent or more general results to this paper, and in  \cite{Fu} Fulmek gives alternative proofs of some of these. This paper gives alternative proofs of those results, and also gives a new type of result which expresses a close relationship between the tiling functions of certain weighted tubey regions which arise when applying Ciucu's Factorization Theorem and their unweighted counterparts. A similar relationship was observed in \cite{Ci13}, see Theorem \ref{relationgeneral} of this paper and the note that follows.

\section{Further Background}
\labelnote{background}

In this section we discuss some of the ideas and techniques from this area that will be useful in the sections ahead.

The partition of the region's triangles into those which are up-pointing and those which are down-pointing corresponds to a proper 2-coloring of the region's planar dual graph; thus, borrowing notions from matching theory on bipartite graphs, we say a region is \textbf{balanced} if it contains the same number of up-pointing and down-pointing triangles. An unbalanced region has no tilings.

For some regions, every possible tiling of that region includes a particular lozenge. For example, if a triangle in a region is adjacent to only one other triangle then all tilings of that region must include the lozenge covering those two triangles. Given a region $R$ with a forced lozenge $l$, if one constructs $R'$ by removing from $R$ the triangles covered by $l$, then the tilings of $R$ and $R'$ are in a natural bijection (tilings of $R$ restrict to tilings of $R'$ by the exclusion of $l$). Then if $\wt(l)$ denotes the weight of lozenge $l$,
$M(R) = \wt(l) M(R').$ Thus when we wish to study $R$ it will sometimes behoove us to study $R'$ instead, especially when $\wt(l)=1.$

The lozenge tilings of a region are in bijection with families of nonintersecting paths on a subgraph of the north-east directed square lattice. The particular subgraph depends on the region (but is essentially shaped like the region), and the paths start at the subgraph's sources and end at its sinks.

Given a lozenge tiling of a region, we mark the middle of the east edge of each up-pointing triangle with a dot; we decorate each lozenge which has dots on its boundary by drawing a line connecting these dots. An appropriate linear transformation of this image moves the dots to vertices on the square lattice, so that the lines decorating certain lozenges form north-east paths starting from the region's southwest boundaries and ending at the region's northeast boundaries. Figure \ref{fig:Lindstrom} depicts this transformation. The other possible tilings of the region correspond to the other nonintersecting families of lattice paths on the subgraph of the north-east directed square lattice which lies within the distorted region. Using this bijection to enumerate lozenge tilings is a standard technique, and is explained in greater detail in Sections 4 and 5 of \cite{Ci01}.

\begin{figure}
    \begin{minipage}[c]{\linewidth}
    \begin{multicols}{2}
    \centering
    \begin{tikzpicture}[x={(0:1cm)},y={(120:1cm)},z={(240:1cm)},scale=.45]
	\pgfmathtruncatemacro{\a}{5}
	\pgfmathtruncatemacro{\b}{4}
	\pgfmathtruncatemacro{\c}{4}
	\pgfmathtruncatemacro{\t}{1}
    
	\pgfmathtruncatemacro{\longdiag}{\a+\b+\c+\t+\t}
	
	\node[shape=circle,fill=black,scale=.35] (A) at (-2,1.5,0) {};
	\node[shape=circle,fill=black,scale=.35] (a) at (3,1.5,-1) {};
	\node[shape=circle,fill=black,scale=.35] (B) at (-1,.5,0) {};
	\node[shape=circle,fill=black,scale=.35] (b) at (3,.5,0) {};
	\node[shape=circle,fill=black,scale=.35] (C) at (0,.5,3) {};
	\node[shape=circle,fill=black,scale=.35] (c) at (4,.5,2) {};
	\node[shape=circle,fill=black,scale=.35] (D) at (1,.5,5) {};
	\node[shape=circle,fill=black,scale=.35] (d) at (5,.5,4) {};
	
	\draw[clip] (1,0,-4) --++
	(-4,0,0) --++ (0,0,1) --++(-1,0,0) --++(0,0,1) --++(0,-1,0) --++(1,0,0) --++ (0,-1,0) --++ (0,0,2) --++(0,-1,0) --++(0,0,1) --++(0,-1,0) --++(1,0,0) --++(0,0,-1) --++(3,0,0)
	--++ (0,1,0) --++(0,0,-1)
	--++ (0,1,0) --++(0,0,-1)
	--++ (0,1,0) --++(0,0,-1)
	--++ (0,1,0) -- cycle;
	
	%Horizontal
	\draw (0,0,4) --++ (1,0,0);
	\draw (0,0,3) --++ (4,0,0);
	\draw (0,1,3) --++ (3,0,0);
	\draw (4,1,3) --++ (1,0,0);
	\draw (-1,0,1) --++ (4,0,0);
	\draw (-1,0,0) --++ (4,0,0);
	\draw (-1,1,0) --++ (4,0,0);
	\draw (-2,2,0) --++ (3,0,0);
	\draw (2,2,0) --++ (2,0,0);
	\draw (0,3,0) --++ (4,0,0);

	%Vertical
	\draw (-1,1,0) --++(0,1,0) --++(0,0,-1);
	\draw (2,0,5) --++(0,1,0) --++(0,0,-1) --++(0,1,0) --++(0,0,-2) --++(0,2,0) --++(0,0,-2);
	\draw (2,0,4) --++(0,1,0);
	\draw (3,0,4) --++(0,2,0) --++(0,0,-2) --++(0,2,0) --++(0,0,-2);
	\draw (4,0,4) --++(0,2,0) --++(0,0,-2) --++(0,1,0) --++ (0,0,-1) --++ (0,1,0) --++(0,0,-1);
	\draw (0,0,-1) --++(0,0,-1) --++ (0,1,0);
	\draw (3,0,3) --++ (0,0,-1) --++(0,1,0);
	
	%lattice paths
	\draw[very thick, blue] (0,-.5,4) --++ (1,0,0) --++(0,0,-1) --++(3,0,0);
	\draw[very thick, blue] (0,.5,3) --++ (3,0,0) --++(0,0,-1) --++(1,0,0);
	\draw[very thick,blue] (-1,.5,0) --++(4,0,0);
	\draw[very thick,blue] (-2,1.5,0) --++(3,0,0) --++(0,0,-1) --++(2,0,0);
	
	\foreach \i in {0,...,\longdiag}{
	\foreach \j in {0,...,\longdiag}
        \node [shape=circle,fill=black,scale=.35] (ij) at (\i,\j+.5,\c+\t) {};
    }
	
    \end{tikzpicture}
    
    \vfill
    \null
     \columnbreak
    
    \begin{tikzpicture}[x={(0:1cm)},y={(135:1.41cm)},z={(270:1cm)},scale=.45]
	\pgfmathtruncatemacro{\a}{5}
	\pgfmathtruncatemacro{\b}{4}
	\pgfmathtruncatemacro{\c}{4}
	\pgfmathtruncatemacro{\t}{1}
    
	\pgfmathtruncatemacro{\longdiag}{\a+\b+\c+\t+\t}
	
		\node[shape=circle,fill=black,scale=.35] (A) at (-2,1.5,0) {};
	\node[shape=circle,fill=black,scale=.35] (a) at (3,1.5,-1) {};
	\node[shape=circle,fill=black,scale=.35] (B) at (-1,.5,0) {};
	\node[shape=circle,fill=black,scale=.35] (b) at (3,.5,0) {};
	\node[shape=circle,fill=black,scale=.35] (C) at (0,.5,3) {};
	\node[shape=circle,fill=black,scale=.35] (c) at (4,.5,2) {};
	\node[shape=circle,fill=black,scale=.35] (D) at (1,.5,5) {};
	\node[shape=circle,fill=black,scale=.35] (d) at (5,.5,4) {};
	
	\draw[clip] (1,0,-4) --++
	(-4,0,0) --++ (0,0,1) --++(-1,0,0) --++(0,0,1) --++(0,-1,0) --++(1,0,0) --++ (0,-1,0) --++ (0,0,2) --++(0,-1,0) --++(0,0,1) --++(0,-1,0) --++(1,0,0) --++(0,0,-1) --++(3,0,0)
	--++ (0,1,0) --++(0,0,-1)
	--++ (0,1,0) --++(0,0,-1)
	--++ (0,1,0) --++(0,0,-1)
	--++ (0,1,0) -- cycle;
	
	%Horizontal
	\draw (0,0,4) --++ (1,0,0);
	\draw (0,0,3) --++ (4,0,0);
	\draw (0,1,3) --++ (3,0,0);
	\draw (4,1,3) --++ (1,0,0);
	\draw (-1,0,1) --++ (4,0,0);
	\draw (-1,0,0) --++ (4,0,0);
	\draw (-1,1,0) --++ (4,0,0);
	\draw (-2,2,0) --++ (3,0,0);
	\draw (2,2,0) --++ (2,0,0);
	\draw (0,3,0) --++ (4,0,0);

	%Vertical
	\draw (-1,1,0) --++(0,1,0) --++(0,0,-1);
	\draw (2,0,5) --++(0,1,0) --++(0,0,-1) --++(0,1,0) --++(0,0,-2) --++(0,2,0) --++(0,0,-2);
	\draw (2,0,4) --++(0,1,0);
	\draw (3,0,4) --++(0,2,0) --++(0,0,-2) --++(0,2,0) --++(0,0,-2);
	\draw (4,0,4) --++(0,2,0) --++(0,0,-2) --++(0,1,0) --++ (0,0,-1) --++ (0,1,0) --++(0,0,-1);
	\draw (0,0,-1) --++(0,0,-1) --++ (0,1,0);
	\draw (3,0,3) --++ (0,0,-1) --++(0,1,0);
	
	%lattice paths
	\draw[very thick, blue] (0,-.5,4) --++ (1,0,0) --++(0,0,-1) --++(3,0,0);
	\draw[very thick, blue] (0,.5,3) --++ (3,0,0) --++(0,0,-1) --++(1,0,0);
	\draw[very thick,blue] (-1,.5,0) --++(4,0,0);
	\draw[very thick,blue] (-2,1.5,0) --++(3,0,0) --++(0,0,-1) --++(2,0,0);
	
	\foreach \i in {0,...,\longdiag}{
	\foreach \j in {0,...,\longdiag}
        \node [shape=circle,fill=black,scale=.35] (ij) at (\i,\j+.5,\c+\t) {};
    }
	
    \end{tikzpicture}
    
    \vfill
    
    \end{multicols}

    \vspace{-.5cm}
    
    \caption{Left: a lozenge tiling of a region is decorated. Right: that image is linearly transformed so that its decorations depict a nonintersecting family of paths on a particular subgraph of the north-east directed square lattice.}
    \labelnote{fig:Lindstrom}
    \end{minipage}
\end{figure}
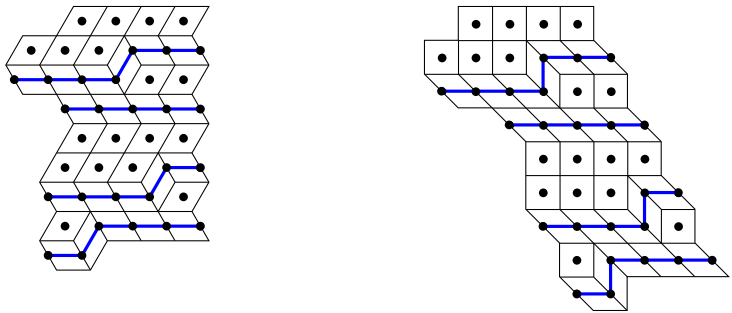
A signed count of families of nonintersecting lattice paths on an acyclic digraph from $n$ indexed sources to $n$ indexed sinks is given by the determinant of an $n \times n$ matrix, where the $(i,j)$th entry is the number of paths from the $i$th source to the $j$th sink. We henceforth refer to this as the \textbf{path matrix} of a region.
This result has been discovered independently at least three times. A version was published in 1959 by Karlin and McGregor ((B) from \cite{Ka}) who were studying the paths of multiple particles through certain discrete Markov processes, so that the particles never coincided in a state.
In 1973 Lindstr{\"o}m published the result, as described above, in \cite{Li}. The result again appeared in a paper in 1985  \cite{Ge85} by Gessel and Viennot, and again in a 1989 preprint \cite{Ge89}, in the context that the digraph is a portion of the north-east directed square lattice, so that the $(i,j)$th entry of the path matrix is a simple binomial. For Gessel and Viennot (as for us) the sign of each nonintersecting family can be assumed to be positive, so that the determinant of the matrix gives the total number of such families.

In one case where we apply this result, the path matrix is naturally the product of two other path matrices. In this case, we will apply the Cauchy-Binet formula (Theorem 4.15 in Section 4.6 of \cite{Br}).
This method has also been used by Ciucu. \cite{Ci21}

\begin{theorem}[The Cauchy Binet formula]\labelnote{CauchyBinet}
Let A be an $m\times n$ matrix and B an $n\times m$ matrix. For $S,T$ sets of integers, let $X_{S,T}$ denote the minor of matrix $X$ consisting of its rows indexed by $S$ and columns indexed by $T$. Then
$$\det(AB) = \sum_{S \in \binom{[n]}{m}} \det(A_{[m],S}) \det(B_{S,[m]}).$$
\end{theorem}

We will also apply the following results by Kuo, which are slight variations on Theorems 2.4 and 2.5 from \cite{Ku04}. These results were later given in much greater generality by Yan, Yeh, and Zhang \cite{Ya}.

\begin{lemma}\labelnote{kuo} Let $R$ be a simply connected region of the triangular lattice, and let $\alpha, \beta, \gamma, \delta$ be triangles in $R$ of the same orientation which touch the boundary of $R$ at a corner or edge, so that the places they touch the boundary appear in the cyclic order $\alpha, \beta, \gamma, \delta$. Then
\begin{align*}
    & \M(R - \alpha-\beta)\cdot \M(R - \gamma-\delta) \\
   =& \M(R -\alpha-\delta) \cdot \M(R - \beta-\gamma) - \M(R-\alpha-\gamma) \cdot \M(R-\beta-\delta) \numberthis \labelnote{eqn}
\end{align*}
and
\begin{align*}
    & \M(R - \alpha)\cdot \M(R-\beta - \gamma-\delta)+  M(R-\gamma)\cdot M(R-\alpha-\beta-\delta) \\
   =& \M(R -\beta) \cdot \M(R - \alpha-\gamma-\delta) + \M(R-\delta) \cdot \M(R-\alpha-\beta-\gamma). \numberthis \labelnote{eqn2}
\end{align*}
\end{lemma}

The following is a generalization of the Graph Splitting Lemma, which appears as Lemma 2.1 in a 2014 paper by Ciucu and Lai \cite{Ci14}, and is implicit in earlier work by Ciucu \cite{Ci97}. This more general result is also implicit in earlier work. The proof given here generalizes the proof from the 2014 paper.

\begin{lemma}
	Let $G=(V,E)$ be a connected balanced bipartite finite graph, with a proper coloring of vertices either black or white. Let $G_1 = (V_1,E_1) , G_2 = (V_2,E_2)$ be disjoint induced subgraphs of $G$ so that $V = V_1 \sqcup V_2$, and $E = E_1 \sqcup E_2 \sqcup E_3$. Suppose that all vertices in $V_1$ which are incident to an edge in $E_3$ are colored black, and let $n$ denote the number of black vertices in $V_1$ minus the number of white vertices in $V_1$. Then all perfect matchings of $G$ include $n$ edges from $E_3$; in particular if $n<0$ then $M(G)=0$, if $n > |E_3|$ then $M(G) = 0$, and if $n=0$ then $M(G)= M(G_1) \cdot M(G_2)$.
\end{lemma}

\begin{proof}
	Let $G, G_1, G_2, E_3,$ and $n$ be as defined above. We remark that all vertices in $G_2$ incident to an edge in $E_3$ must be white, and there must be $n$ more white vertices than black vertices in $G_2$.  Let $\mu \subset E$ be a perfect matching of $M$.

Suppose $n<0$. Each white vertex in $V_1$ must be matched to a black vertex in $V_1$. Since there are fewer black vertices than white vertices in $V_1$, no such matching $\mu$ exists.

Suppose instead $n \geq 0$. Then each white vertex in $V_1$ must be matched to a black vertex in $V_1$, leaving the other $n$ black vertices in $V_1$ which are matched to white vertices in $V_2$; thus exactly $n$ edges are in $\mu \cap E_3$. Therefore if $n > |E_3|$, no such matching $\mu$ exists.

In the case where $n=0$,  $\mu$ must induce a perfect matching on each of $G_1$ and $G_2$, and similarly any pair of perfect matchings $\mu_1$ on $G_1$ and $\mu_2$ on $G_2$ together form a matching on $G$, so the perfect matchings on $G$ and the pairs of perfect matchings on $G_1, G_2$ are in bijection.

\end{proof}

We apply this to a specific type of partition of regions on the triangular lattice: 

\begin{corollary}[Region Splitting Lemma]\labelnote{splitting}
Let $R$ be a balanced region of the triangular lattice, and $L$ a lattice line passing through $R$, so that $P$ and $Q$ are the subregions of $R$ on each side of the lattice line. Suppose the triangles in $P$ along $L$ are all up-pointing, and let $n$ denote the number of up-pointing triangles in $P$ minus the number of down-pointing triangles in $P$. Then all tilings of $R$ have exactly $n$ lozenges which cross $L$. 
In particular, if $n < 0$ then $R$ is untileable, if fewer than $n$ lozenges cross $L$ then $R$ is untileable, and if $n=0$ then $M(R) = M(P)\cdot M(Q).$
\end{corollary}

The following theorem is a special case of the Factorization Theorem by Ciucu \cite{Ci97}, expressed in the language and context of this paper.

\begin{theorem}[Factorization Theorem (Ciucu)]\labelnote{ciucu}
Let $R$ be a balanced vertically symmetric region with  symmetry axis $L$. Label the triangles within $R$ on $L$ from top to bottom, $a_1, b_1, ..., a_k, b_k$.
Let $R^-$ denote the subregion of $R$ consisting of triangles strictly west of $L$, along with triangles labeled $a_i$ which are down-pointing, and those labeled $b_i$ which are up-pointing. Let $R^+$ denote the subregion of $R$ complementary to $R^-$. Say the weight functions for $R^-$ and $R^+$ are the restrictions of the weight function of $R$ to those regions, but with the weights of any vertical lozenges on $L$ cut in half.
Then
$$\M(R) = 2^k \M(R^-) \M(R^+).$$
\end{theorem}

For example when Theorem \ref{ciucu} is applied to $H_{a,b,b,2n,(u_i)_1^n,(u_i)_1^n}$, $R^-$ is  $V_{a,b,n, (u_i)_1^n}$ and $R^+$ is congruent to $\overline V^+_{a,b+1,n, (u_i+1)_1^n}$. 

The author proved in \cite{Co} the following results for dented hexagons.

\begin{theorem}\labelnote{Condon}
    Let a balanced dented hexagon $H_{a,b,c,t,\vec u, \vec v}$ be given. Let $L_N$ be the Nth horizontal lattice line south of the northern side\footnote{$L_N$ is then $\frac{\sqrt3}{2}N$ units south of the northern side of $H_{a,b,c,t,\vec u, \vec v}$.} of  $H_{a,b,c,t,\vec u, \vec v}$, and $\mu_N$ be the number of dents north of $L_N$. Then $H$ has a tiling iff $\mu_N \leq N$ for all $N \in \mathbb N$.
\end{theorem}

\begin{theorem}\labelnote{paper1}
For $H_{a,b,c,m+n,(u_i)_1^m,(v_j)_1^n}$ a dented hexagon,
\begin{align*}
     {\M(H_{a,b,c,m+n,(u_i)_1^m, (v_j)_1^n})} &= {\M(H_{0,b,c,m+n,(u_i)_1^m, (v_j)_1^n})} \prod_{i=1}^m(u_i)_{\underline{u_i}} \prod_{j=1}^n(v_j)_{\underline{v_j}}\\
     & \times \frac{\p({a,b+n, c+m})}{\prod_{i=1}^m(a+u_i)_{\underline{u_i}}\prod_{j=1}^n(a+v_j)_{\underline{v_j}}}.
\end{align*}
\end{theorem}

This is equivalent to the following:

\begin{corollary}\labelnote{Condon3}
    For $H_{a,b,c,m+n,(u_i)_1^m,(v_j)_1^n}$ a dented hexagon, there exists some $k:=k(b,c,(u_i)_1^m,(v_j)_1^n)$ independent of $a$ so that
    $$ \M(H_{a,b,c,m+n,(u_i)_1^m, (v_j)_1^n}) = k \cdot \frac{\p({a,b+n, c+m})}{\prod_{i=1}^m(a+u_i)_{\underline{u_i}}\prod_{j=1}^n(a+v_j)_{\underline{v_j}}}.$$
\end{corollary}

\section{Main Results}
\labelnote{results}

In this section we describe the main results of this paper, and where proofs can be found.  The following results about tileability can be found in Section \ref{tileable}.

\begin{proposition}\labelnote{SymmetricTileability}
    The dented hexagon $H_{a,b,b,2n,(u_i)_1^n,(u_i)_1^n}$ has vertically symmetric tilings iff $a$ is even and $2i \leq u_i$ for all $i \in [n]$. 
\end{proposition}

\begin{corollary}\labelnote{Vexistence}
    The region $V_{a,b,n,(u_i)_1^n}$ is tileable iff $2i \leq u_i$ for all $i \in [n]$. 
\end{corollary}

\begin{corollary}\labelnote{easiness}
    Consider the dented half-hexagons $V_{a,b,n,(u_i)_1^n}$ and $V_{a',b',n,(u'_i)_1^n}$ with $u_i \leq u_i'$ for all $i \in [n]$. If $V_{a,b,n,(u_i)_1^n}$ is tileable, then so is $V_{a',b',n,(u'_i)_1^n}$.
\end{corollary}

Next, we give the tiling function for a dented half hexagon, with proof in Section \ref{mainproof}.
We use $\f_{b,n,(u_i)_1^n}(a)$ to denote the tiling function for the family of dented half-hexagons $\{V_{a,b,n,(u_i)_1^n} \ |\  a \in \mathbb N\}$. 

\begin{theorem}\labelnote{main}
    For $V_{a,b,n,(u_i)_1^n}$ a dented half-hexagon, $\f_{b,n,(u_i)_1^n}(a)$ is a polynomial in $a$, satisfying
    \begin{equation}\labelnote{eq:dentsfunction}
        \f_{b,n,(u_i)_1^n}(a) = f_{b,n,(u_i)_1^n}(0) \cdot \prod_{i=1}^n (u_i)_{\underline{u_i}} \cdot \dfrac{\p_{|}(a,b+n)}{\prod_{i=1}^n (2a+u_i)_{\underline{u_i}}}.
    \end{equation}
\end{theorem}

Lastly, in Section \ref{relationproof}, we prove the following results regarding a relationship between the tiling functions of symmetrizable regions and their complements in the symmetrization. We use $f^+_{b,n,(u_i)_1^n}(a)$ to denote the tiling function for the family $\{V^+_{a,b,n,(u_i)_1^n} \ |\  a \in \mathbb N\}$, and $\overline f^+_{b,n,(u_i)_1^n}(a)$ to denote the tiling function for the weighted family $\{\overline V^+_{a,b,n,(u_i)_1^n} \ |\  a \in \mathbb N\}$. We will show
\begin{corollary}\labelnote{mainplus}
	    For $V_{a,b,n,(u_i)_1^n}$ a dented half-hexagon, $f^+_{b,n,(u_i)_1^n}(a)$ is a polynomial in $a$, satisfying
    \begin{equation}
        \overline f^+_{b,n,(u_i)_1^n}(a) =  \overline  f^+_{b,n,(u_i)_1^n}(0) \cdot \prod_{i=1}^n (u_i)_{\underline{u_i}} \cdot \dfrac{\p_{-}(a,b+n)}{\prod_{i=1}^n (2a+u_i)_{\underline{u_i}}}.
    \end{equation}
\end{corollary}
As $f^+$ and $\overline f^+$ are polynomials in $a$, they can be evaluated for real values of $a$ which are not integers. Theorem \ref{relation} is to be understood in that context.

\begin{theorem}\labelnote{relation}
    For $V_{a,b,n,(u_i)_1^n}$ a dented half-hexagon,
    $$\overline{f}^+_{b,n,(u_i)_1^n}(a) = f^+_{b,n,(u_i)_1^n}(a-1/2).$$
\end{theorem}

Remarkably, Theorem \ref{relation} holds in much greater generality, as expressed in Theorem \ref{relationgeneral}.

\begin{theorem}\labelnote{relationgeneral}
    Given a family of tileable tubey regions $R_{z}(a)$, let $g(a)$ be the tiling function for $\{R_{z}(a+1/2): a \in \mathbb N\}$, and $\overline g(a)$ be the tiling function of $\{\overline R_{z}(a+1/2): a \in \mathbb N\}$. Then both $g(a)$ and $\overline g(a)$ are polynomials in $a$ and
    $$\overline g(a) = g(a-1/2).$$
\end{theorem}

A similar phenomenon was observed by Ciucu and Fischer, see Remark 1 at the end of Section 2 of \cite{Ci13}.

\section{Proof of Tileability Results}
\labelnote{tileable}

In this section we show that vertically symmetric regions are tileable iff they have tilings which are themselves vertically symmetric. In particular, tileability conditions for dented half-hexagons are equivalent to tileability conditions for dented hexagons. 

\begin{lemma}\labelnote{CiucuCorr}
    Let $R$ be a vertically symmetric region with symmetry axis $L$, whose subregion of triangles on $L$ is also a tileable region. Then $R$ has a tiling iff it has a vertically symmetric tiling.
\end{lemma}

\begin{proof}
    It suffices to show that if $R$ has a tiling it must also have a vertically symmetric tiling. We proceed using the notation of Theorem \ref{ciucu}.
    
    Since $R$ has a tiling, $\M(R)$ is nonzero, and so $\M(R^-)$ and $\M(R^+)$ must both be nonzero, by Theorem \ref{ciucu}.

Since the subregion of triangles on $L$ is tileable, it must be the union of interior-disjoint vertical lozenges on $L$: then each $a_i$ is up-pointing, each $b_i$ is down-pointing, and $R^-$ contains no triangles on $L$.
    
    Let $\mu$ be any tiling on $R^-$, and let $\mu'$ be its reflection across $L$ (interpreted as a tiling of the subregion of triangles strictly east of $L$). Let $\nu$ be the tiling of the subregion of triangles on $L$. Then  $\mu \cup \mu' \cup \nu$ is a tiling of $R$ which is symmetric across $L$.

\end{proof}

We are now ready to prove Proposition \ref{SymmetricTileability}.

\begin{proof}[Proof of Proposition \ref{SymmetricTileability}]
    Let $H:= H_{a,b,b,2n,(u_i)_1^n, (u_i)_1^n}$ be a dented hexagon. Then $H$ is symmetric across a vertical axis $L$, as displayed in the Figure \ref{fig:factorizeHex}.
    
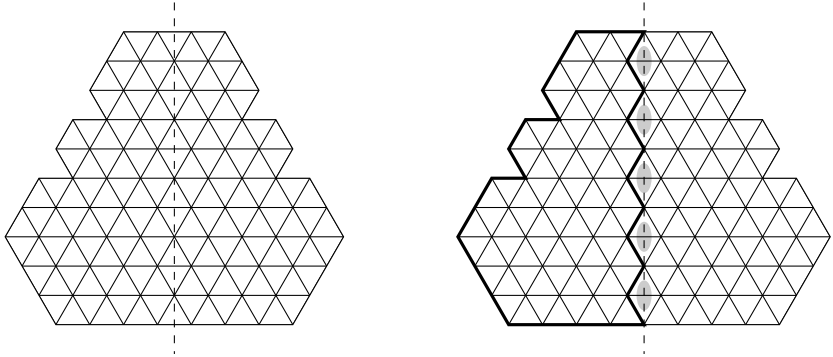
\begin{figure}
    \begin{minipage}[c]{\textwidth}
    \centering
    \begin{multicols}{2}

    \begin{tikzpicture}[x={(0:.45cm)},y={(120:.45cm)},z={(240:.45cm)}]
	\pgfmathtruncatemacro{\a}{3}
	\pgfmathtruncatemacro{\b}{3}
	\pgfmathtruncatemacro{\c}{3}
	\pgfmathtruncatemacro{\t}{4}
    
	\pgfmathtruncatemacro{\longdiag}{\a+\b+\c+\t+\t+1}
	
	\draw[dashed] (0,2,-2) -- (0,-4,4);
	
	\draw[clip] (\a+\t,0,\t) -- ++(0,2,0)
	-- ++ (-1,0,0) -- ++(0,0,-1)--++(0,1,0)
	-- ++ (-1,0,0) -- ++(0,0,-1)
	-- ++ (0,2,0)
	-- ++(-\a,0,0) -- ++(0,0,2)
	-- ++ (0,-1,0) -- ++(-1,0,0) -- ++ (0,0,1)
	-- ++ (0,-1,0) -- ++(-1,0,0) -- ++ (0,0,2) --
	++(0,-\b,0) -- ++(\a+\t,0,0) -- cycle;

	\foreach \i in {0,...,\longdiag}{
        \draw[-] (\i,0,\c+\t) -- ++(0,\longdiag,0);
        \draw[-] (\a+\t,\i,0) -- ++(0,0,\longdiag);
        \draw[-] (0,\b+\t,\i) -- ++(\longdiag,0,0);
    }
    \end{tikzpicture}

    \columnbreak
        
    \begin{tikzpicture}[x={(0:.45cm)},y={(120:.45cm)},z={(240:.45cm)}]
	\pgfmathtruncatemacro{\a}{4}
	\pgfmathtruncatemacro{\b}{3}
	\pgfmathtruncatemacro{\c}{3}
	\pgfmathtruncatemacro{\t}{4}

	\pgfmathtruncatemacro{\longdiag}{\a+\b+\c+\t+\t}
	
	\draw[very thick] (0,0,0) --++(0,0,-1) --++(0,1,0) --++(0,0,-1) --++(-2,0,0) -- ++(0,0,2)
	-- ++ (0,-1,0) -- ++(-1,0,0) -- ++ (0,0,1)
	-- ++ (0,-1,0) -- ++(-1,0,0) -- ++ (0,0,2) -- ++ (0,-3,0) --++ (4,0,0)
	--++(0,1,0)
	--++(0,0,-1) --++(0,1,0)
	--++(0,0,-1) --++(0,1,0)
	--++(0,0,-1) -- cycle;
	
	\draw[dashed] (.5,2,-2) -- (.5,-4,4);
	
	\draw[clip] (\a+\t,0,\t) -- ++(0,2,0)
	-- ++ (-1,0,0) -- ++(0,0,-1)--++(0,1,0)
	-- ++ (-1,0,0) -- ++(0,0,-1)
	-- ++ (0,2,0)
	-- ++(-\a,0,0) -- ++(0,0,2)
	-- ++ (0,-1,0) -- ++(-1,0,0) -- ++ (0,0,1)
	-- ++ (0,-1,0) -- ++(-1,0,0) -- ++ (0,0,2) --
	++(0,-\b,0) -- ++(\a+\t,0,0) -- cycle;

	\foreach \i in {0,...,\longdiag}{
        \draw[-] (\i,0,\c+\t) -- ++(0,\longdiag,0);
        \draw[-] (\a+\t,\i,0) -- ++(0,0,\longdiag);
        \draw[-] (0,\b+\t,\i) -- ++(\longdiag,0,0);
    }

		\foreach \i in {1,3,5,7,-1}{
	\fill[black, fill opacity=.2] (.5*\i,0,\i-1) ellipse (.1cm and .2cm);
	}

    \end{tikzpicture}
    
    \end{multicols}
    \vspace{-.5cm}
    \caption{Left: a vertically symmetric dented hexagon with $a$ odd. Right: a vertically symmetric dented hexagon with $a$ even, and the outline of $R^-$ drawn thickly, and the weighted lozenges of $R^+$ shaded. Both figures' symmetry axes are depicted as a dashed line.}
    \labelnote{fig:factorizeHex}
    \end{minipage}
\end{figure}

    The vertices along the northern side of $H$ are down-pointing. In the case that $a$ is odd the top triangle on $L$ is down-pointing, and is only adjacent to triangles strictly left or right of $L$: this triangle cannot be matched to either in a vertically symmetric tiling. Then $H$ has no vertically symmetric tilings.
    
    In the case that $a$ is even, $H$ satisfies the conditions in Lemma \ref{CiucuCorr} and thus has a symmetric tiling iff it has a tiling. By Theorem \ref{Condon}, this occurs if each $N$th lattice line south of $H$'s northern side has at most $N$ dents lying north of it. This is equivalent to the condition $2i \leq u_i$ for each $i \in [n]$.

\end{proof}

We are now ready to prove Corollaries \ref{Vexistence} and \ref{easiness}.

\begin{proof}[Proof of Corollary \ref{Vexistence}]

    It suffices to show that perfect matchings of $V_{a,b,n,(u_i)_1^n}$ are in bijection with vertically symmetric perfect matchings of $H_{a,b,b,2n,(u_i)_1^n,(u_i)_1^n}$.
    
    In any vertically symmetric perfect matching of a dented hexagon, triangles on the axis of symmetry must be covered by vertical lozenges. So the tiling of $H_{a,b,b,2n,(u_i)_1^n,(u_i)_1^n}$ is uniquely determined by the tiling of its subregion strictly left of the axis of symmetry; this region is $V_{a,b,n,(u_i)_1^n}$.
    
\end{proof}

\begin{proof}[Proof of Corollary \ref{easiness}]

    Since $V_{a,b,n,(u_i)_1^n}$ is tileable, $2i \leq u_i \leq u_i'$ for each $u_i'$.

\end{proof}

\section{Proof of Theorem \ref{main}}
\labelnote{mainproof}

Our Theorem \ref{main} has been proven independently by Lai in \cite{La21} and Fulmek in \cite{Fu}, in fact in more generality, but we give an alternate proof in this section for completeness.

\begin{lemma}\labelnote{polynomial}
    Where $V_{a,b,n,(u_i)_1^n}$ is a tileable dented half-hexagon, $f_{b,n,(u_i)_1^n}(a)$ is a polynomial in $a$.
    
\end{lemma}

\begin{proof}
    Per the transformation discussed briefly at the beginning of section \ref{background}, the tilings of $V_{a,b,n,(u_i)_1^n}$ are in bijection with nonintersecting lattice paths on a subgraph of the north-east directed square lattice. As depicted in Figure \ref{fig:DentedHexCoordinates1}, the sources of that digraph can be said to have coordinates $\{(i-1,b-i) : i \in [b]\}$ and $\{(0,2b+2n-u_i) : i \in [n]\}$, in which case the sinks have coordinates $\{(a+b+n-j, 2j-2): j \in [b+n]\}$. The entries in the path matrix are then of either form:
    \begin{equation}\binom{a+n+j-1}{2j-2-b+i} \quad \mbox{ or } \quad \binom{a-b-n-2+j+u_i}{2j-2-2b-2n+u_i}. \labelnote{explicit}\end{equation}
    So, treating $a$ as an indeterminate each entry in the path matrix is a polynomial on $a$ of degree $2j-2-b+i$ or $2j-2-2b-2n+u_i$, or is identically zero. The determinant of the path matrix is thus a signed sum over products of polynomials on $a$, and is therefore a polynomial on $a$.

\end{proof}

\begin{figure}
    \begin{minipage}[c]{\textwidth}
    \centering

	\begin{multicols}{2}

    \begin{tikzpicture}[x={(0:1cm)},y={(120:1cm)},z={(240:1cm)},scale=.45]
	\pgfmathtruncatemacro{\a}{6}
	\pgfmathtruncatemacro{\b}{5}
	\pgfmathtruncatemacro{\c}{5}
	\pgfmathtruncatemacro{\t}{6}
    
	\pgfmathtruncatemacro{\longdiag}{\a+\b+\c+\t}

	\begin{scope}

	\draw[clip] (0,0,0) --++(-3,0,0)
	 -- ++(0,0,2)-- ++ (0,-1,0) -- ++(-1,0,0)
	 -- ++(0,0,2)-- ++ (0,-1,0) -- ++(-1,0,0)
	 -- ++(0,0,2)-- ++ (0,-1,0) -- ++(-1,0,0)
	 -- ++(0,0,2)-- ++ (0,-5,0) --++ (6,0,0)
	--++ (0,1,0) --++(0,0,-1)
	--++ (0,1,0) --++(0,0,-1)
	--++ (0,1,0) --++(0,0,-1)
	--++ (0,1,0) --++(0,0,-1)
	--++ (0,1,0) --++(0,0,-1)
	--++ (0,1,0) --++(0,0,-1)
	--++ (0,1,0) --++(0,0,-1)
	--++ (0,1,0) --++(0,0,-1) -- cycle;
	
	\foreach \i in {0,...,\longdiag}{
        \draw[-] (0,\a,\i) -- ++(0,-\longdiag,0);
        \draw[-] (0,\a-\i,0) -- ++(0,0,\longdiag);
        \draw[-] (0,-\i,0) -- ++(-\longdiag,0,0);
   	 }
	\end{scope}

	\begin{scope}
		\clip (0,0,0) --++ (-4,0,0) --++(0,0,17) --++(12,0,0)  --++(0,0,-1) -- cycle;
		\foreach \i in {1,...,16}{
			\draw[red] (.5,.5-.5*\i,.5*\i) --++(-14,0,0);
		}
		\foreach \i in {1,...,11}{
			\draw[red] (-3.5 + \i, 0,-.5) --++(0,0,19);
		}
	\end{scope}	

	\begin{scope}[xshift=-10.2cm,yshift=-13.4cm, node distance=.45cm,]

	\node [shape=circle,fill=black,scale=.35] (O) at (0,0){};% node[below=2pt] {\tiny (0,0)};
	\node [below left of=O] {\tiny (0,0)};

	\node [shape=circle,fill=black,scale=.35] (B) at (4,0){};

	\node [shape=circle,fill=black,scale=.35] (C) at (4,4){};

	\node [shape=circle,fill=black,scale=.35] (C1) at (4,3){};
	\node [shape=circle,fill=black,scale=.35] (C2) at (4,2){};
	\node [shape=circle,fill=black,scale=.35] (C3) at (4,1){};

	\node [shape=circle,fill=black,scale=.35] (D3) at (7,7){};
	\node [shape=circle,fill=black,scale=.35] (D2) at (10,10){};
	\node [shape=circle,fill=black,scale=.35] (D1) at (13,13){};

	\foreach \i in {1,...,8}{
		\node [shape = circle, draw=black, fill=white, scale=.35] (S\i) at (9+\i, 2*\i-2){};
	}

	\end{scope}

    \end{tikzpicture}

%%%%%%%%%%%%%%%%%%%%%%%%%%%%%%%%%%%%%%%%%%%%%%%%%%%%%%%%%%%%%%%%

    \begin{tikzpicture}[x={(0:1cm)},y={(135:1.41cm)},z={(270:1cm)},scale=.45]
	\pgfmathtruncatemacro{\a}{6}
	\pgfmathtruncatemacro{\b}{5}
	\pgfmathtruncatemacro{\c}{5}
	\pgfmathtruncatemacro{\t}{6}
    
	\pgfmathtruncatemacro{\longdiag}{\a+\b+\c+\t}

	\begin{scope}

	\draw[clip] (0,0,0) --++(-3,0,0)
	 -- ++(0,0,2)-- ++ (0,-1,0) -- ++(-1,0,0)
	 -- ++(0,0,2)-- ++ (0,-1,0) -- ++(-1,0,0)
	 -- ++(0,0,2)-- ++ (0,-1,0) -- ++(-1,0,0)
	 -- ++(0,0,2)-- ++ (0,-5,0) --++ (6,0,0)
	--++ (0,1,0) --++(0,0,-1)
	--++ (0,1,0) --++(0,0,-1)
	--++ (0,1,0) --++(0,0,-1)
	--++ (0,1,0) --++(0,0,-1)
	--++ (0,1,0) --++(0,0,-1)
	--++ (0,1,0) --++(0,0,-1)
	--++ (0,1,0) --++(0,0,-1)
	--++ (0,1,0) --++(0,0,-1) -- cycle;
	
	\foreach \i in {0,...,\longdiag}{
        \draw[-] (0,\a,\i) -- ++(0,-\longdiag,0);
        \draw[-] (0,\a-\i,0) -- ++(0,0,\longdiag);
        \draw[-] (0,-\i,0) -- ++(-\longdiag,0,0);
   	 }
	\end{scope}

	\begin{scope}
		\clip (0,0,0) --++ (-4,0,0) --++(0,0,17) --++(12,0,0)  --++(0,0,-1) -- cycle;
		\foreach \i in {1,...,16}{
			\draw[red] (.5,.5-.5*\i,.5*\i) --++(-14,0,0);
		}
		\foreach \i in {1,...,11}{
			\draw[red] (-3.5 + \i, 0,-.5) --++(0,0,19);
		}
	\end{scope}	

	\begin{scope}[xshift=-2.5cm,yshift=-15.5cm, node distance=.45cm,]

	\node [shape=circle,fill=black,scale=.35] (O) at (0,0){};% node[below=2pt] {\tiny (0,0)};
	\node [below left of=O] {\tiny (0,0)};

	\node [shape=circle,fill=black,scale=.35] (B) at (4,0){};
	%\node [below of=B] {\tiny ($b$-1,0)};

	\node [shape=circle,fill=black,scale=.35] (C) at (4,4){};
	%\node [left of=C] {\tiny (0,$b$-1)};

	\node [shape=circle,fill=black,scale=.35] (C1) at (4,3){};
	\node [shape=circle,fill=black,scale=.35] (C2) at (4,2){};
	\node [shape=circle,fill=black,scale=.35] (C3) at (4,1){};

	\node [shape=circle,fill=black,scale=.35] (D3) at (7,7){};
	\node [shape=circle,fill=black,scale=.35] (D2) at (10,10){};
	\node [shape=circle,fill=black,scale=.35] (D1) at (13,13){};

	\foreach \i in {1,...,8}{
		\node [shape = circle, draw=black, fill=white, scale=.35] (S\i) at (9+\i, 2*\i-2){};
	}

	\draw [decorate, decoration=brace] (-2.1,-.1) --++ (4,4);
	\node (L1) at (-1.5,2) {$b-1$};

	\draw [decorate, decoration=brace] (2,4) --++ (11,11);
	\node (L2) at (6,9.5) {$b+2n$};

	\draw [decorate, decoration=brace] (1.9,-2) --++ (-4,0);
	\node (L3) at (-1,-3) {$b-1$};

	\draw [decorate, decoration=brace] (8,-2) --++(-6,0);
	\node (L3) at (4,-3) {$a+n$};
	\end{scope}

	\end{tikzpicture}

	\end{multicols}

    \caption{Left: A dented half-hexagon with the east-northeast square lattice superimposed. We are interested in paths on the directed subgraph of the lattice that lies within the region; sources for this graph are depicted as black circles, and sinks are shown as white circles. Right: We take a linear transformation of the left image so that lattice lines lie vertically and horizontally. We choose Cartesian coordinates for the points where lattice lines intersect, so that the intersection in the bottom left corner of the picture has coordinates $(0,0)$ and points in the dented hexagon have nonnegative coordinates. Sources lie along the southwest side of the region at coordinates $\{(i-1,b-i): i \in [b]\}$ and along dents at coordinates $\{(0,2b+2n-u_i: i \in [n]\}$. Sinks lie at coordinates $\{(a+b+n-j, 2j-2): j \in [b+n]\}$.}
    \label{fig:DentedHexCoordinates1}
    \end{minipage}
\end{figure}
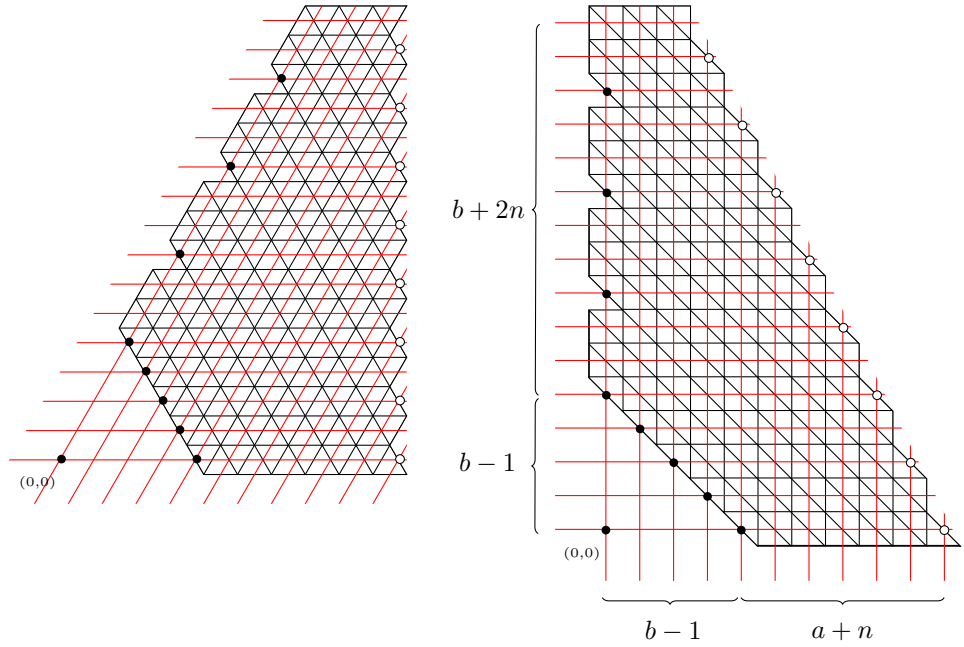

We denote the normalization\footnote{A normalization in the sense that it differs from the actual tiling function by a factor independent of $a$, but is simpler to express.} of the proposed tiling function for dented half-hexagons
$$F_{b,n,(u_i)_1^n}(a):= \dfrac{P_{|}(a,b+n)}{\prod_{i=1}^n(2a+u_i)_{\underline{u_i}}}.$$

We will proceed using the nonstandard notation $\aequiv$ to denote proportionality between rational functions which are polynomials in $a$. That is, for $g_1, g_2$ two rational functions (possibly of several variables) which are polynomials in $a$, we will write
$$g_1 \aequiv g_2$$
to mean there exists some (not identically zero) $c$ independent of $a$ so that
$$g_1 = c \cdot g_2.$$
It is easy to check $\aequiv$ is an equivalence relation of rational functions which are polynomials in $a$, satisfying the following properties where $p_1, p_2,$ and $q$ are polynomials in $a$, and $q$ is not identically zero.

\begin{align}
    p_1(a) \aequiv p_2(a)& \Rightarrow p_1(a) + p_2(a) \aequiv p_1(a) \aequiv p_2(a) \labelnote{eq:equiv1}\\
    p_1(a) \aequiv p_2(a) & \iff p_1(a)q(a) \aequiv p_2(a) q(a)\labelnote{eq:equiv2}
\end{align}

The method of our proof is to show that
\begin{equation}\labelnote{eq:main}
    F_{b,n,(u_i)_1^n}(a) \aequiv f_{b,n,(u_i)_1^n}(a).
\end{equation}
We will do this by induction, and we start by stating a lemma that will be important to the inductive step.
A dented half-hexagon with $u_1=2$ (respectively $u_n = b+2n$) has forced lozenges covering the entire region weakly north of that dent (respectively along its southwest side) which when removed leave dented half-hexagons with different parameters, as depicted in Figure \ref{fig:forced}.

\begin{figure}
    \begin{minipage}[c]{\textwidth}
    \centering
    \begin{multicols}{2}

    \begin{tikzpicture}[x={(0:.45cm)},y={(120:.45cm)},z={(240:.45cm)}]
	\pgfmathtruncatemacro{\a}{4}
	\pgfmathtruncatemacro{\b}{2}
	\pgfmathtruncatemacro{\c}{2}
	\pgfmathtruncatemacro{\t}{6}

	\pgfmathtruncatemacro{\longdiag}{\a+\b+\c+\t+\t}
	
	\fill[black!20] (0,1,-2)--++(-3,0,0)--++(0,0,1)--++(0,-1,0)--++(3,0,0)--++(0,1,0)--cycle;
	
	\draw[clip] (0,0,0) --++(0,0,-1) --++(0,1,0) --++(0,0,-1) --++
	(-3,0,0)%Change 'a'
	-- ++(0,0,1)
	-- ++ (0,-1,0) -- ++(-1,0,0) -- ++ (0,0,2)
	-- ++ (0,-1,0) -- ++(-1,0,0) -- ++ (0,0,2)
	-- ++ (0,-1,0) -- ++(-1,0,0) -- ++ (0,0,0) -- ++ (0,-2,0) --++ 
	(6,0,0)%Change 'a'
	--++(0,1,0)
	--++(0,0,-1) --++(0,1,0)
	--++(0,0,-1) --++(0,1,0)
	--++(0,0,-1) -- cycle;

	\foreach \i in {0,...,\longdiag}{
        \draw[-] (\i,0,\c+\t) -- ++(0,\longdiag,0);
        \draw[-] (\a+\t,\i,0) -- ++(0,0,\longdiag);
        \draw[-] (0,\b+\t,\i) -- ++(\longdiag,0,0);
    }

    \end{tikzpicture}

    \columnbreak
        
    \begin{tikzpicture}[x={(0:.45cm)},y={(120:.45cm)},z={(240:.45cm)}]
	\pgfmathtruncatemacro{\a}{4}
	\pgfmathtruncatemacro{\b}{2}
	\pgfmathtruncatemacro{\c}{2}
	\pgfmathtruncatemacro{\t}{6}

	\pgfmathtruncatemacro{\longdiag}{\a+\b+\c+\t+\t}
	
	%\fill[black!20] (-4,0,5) --++(0,-2,0)--++(6,0,0)--++(0,1,0)--++(0,0,-1)--cycle;
	
	\fill[black!20] (-4,0,5) --++(0,-2,0)--++(1,0,0)--++(0,2,0)--cycle;
	
	\draw[clip] (0,0,0) --++(0,0,-1) --++(0,1,0) --++(0,0,-1) --++
	(-3,0,0)%Change 'a'
	-- ++(0,0,1)
	-- ++ (0,-1,0) -- ++(-1,0,0) -- ++ (0,0,2)
	-- ++ (0,-1,0) -- ++(-1,0,0) -- ++ (0,0,2)
	-- ++ (0,-1,0) -- ++(-1,0,0) -- ++ (0,-2,0) --++ 
	(6,0,0)%Change 'a'
	--++(0,1,0)
	--++(0,0,-1) --++(0,1,0)
	--++(0,0,-1) --++(0,1,0)
	--++(0,0,-1) -- cycle;

	\foreach \i in {0,...,\longdiag}{
        \draw[-] (\i,0,\c+\t) -- ++(0,\longdiag,0);
        \draw[-] (\a+\t,\i,0) -- ++(0,0,\longdiag);
        \draw[-] (0,\b+\t,\i) -- ++(\longdiag,0,0);
    }

    \end{tikzpicture}
    
    \end{multicols}
    \vspace{-.5cm}
    \caption{Each image depicts the region $V_{a,b,n,(u_i)_1^n}$ with some of its forced lozenges shaded gray. Left: a dented half-hexagon with $u_1=2$ has forced lozenges weakly north of that dent; removing these leaves the region $V_{a+1,b,n-1,(u_i-2)_2^n}$. Right: a dented half-hexagon with $u_n=b+2n$ has forced lozenges along the southwestern side of the region; removing these leaves the region $V_{a,b+1,n-1,(u_i)_1^{n-1}}$.}
    \labelnote{fig:forced}
    \end{minipage}
\end{figure}
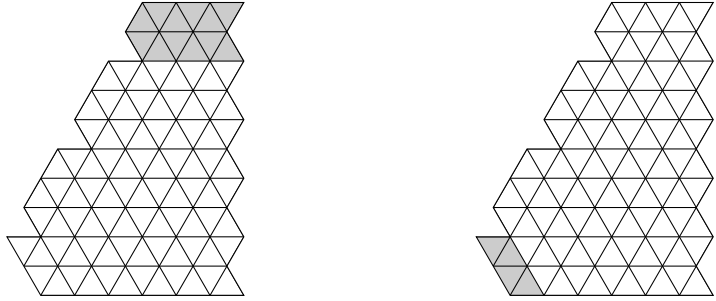

\begin{lemma}\labelnote{case} \ 

    \begin{enumerate}
        \item Where $V_{a,b,n,(u_i)_1^n}$ is a tileable dented half-hexagon, and $u_1=2$,
        $$F_{b,n,(u_i)_1^n}(a) \aequiv F_{b,n-1,(u_i-2)_2^n}(a+1),$$
        $$f_{b,n,(u_i)_1^n}(a) = f_{b,n-1,(u_i-2)_2^n}(a+1).$$
        \item Where $V_{a,b,n,(u_i)_1^{n}}$ is a tileable dented half-hexagon, and $u_n=b+2n$,
        $$F_{b,n,(u_i)_1^{n}}(a) \aequiv F_{b+1,n-1,(u_i)_1^{n-1}}(a),$$
        $$f_{b,n,(u_i)_1^{n} }(a) = f_{b+1,n-1,(u_i)_1^{n-1}}(a).$$
    \end{enumerate}
\end{lemma}

\begin{proof}
    Both statements regarding $f$ follow from the removal of forced lozenges as depicted in Figure \ref{fig:forced}.
    Both statements regarding $F$ are straightforward to check from the definition of $F$.

\end{proof}

We now show that Theorem \ref{main} holds for dented half-hexagons with at most one dent. This will serve as a base case when we show relation \ref{eq:main} holds in general. We let $\vec \emptyset$ denote the empty vector. 

\begin{lemma}\labelnote{basecase}
    For $V_{a, b, n, \vec u}$ a tileable dented half-hexagon with $\vec u = \vec \emptyset$ and $n=0$, or $\vec u = (u)$ and $n=1$, relation \ref{eq:main} holds.
\end{lemma}

\begin{proof}
    We first address some specific cases.
    \begin{itemize}
        \item The region has no dents: in this case,
        $$f_{b, 0, \vec \emptyset}(a) = M(V_{a,b,0, \vec \emptyset}) = P_|(a,b) = F_{b,0,\vec \emptyset}(a).$$
        \item $u=2$; then there are forced lozenges at the top of the region, and
        $$f_{b,1,(2)}(a) = f_{b,0,\vec \emptyset}(a) \aequiv F_{b,0, \vec \emptyset}(a) \aequiv F_{b,1,(2)}(a)$$ by reference to the case with no dents and Lemma \ref{case}.1.
        \item $u=b+2$; then there are lozenges along the southwest boundary of the region, and
        $$f_{b,1,(b+2)}(a) = f_{b+1,0,\vec \emptyset}(a) \aequiv F_{b+1,0, \vec \emptyset} \aequiv F_{b,1,(b+2)}$$
        by reference to the case with no dents and Lemma \ref{case}.2.
        \item $b=0$; then the region is only tileable if $u=2$; in this case the region has one tiling, and $F_{0,1,(u)}(a)=1$ also.
        \item $b=1$; then either $u=2$ or $u=3=b+2$, both cases of which are addressed earlier.
    \end{itemize}
   With these cases addressed, we proceed to discuss regions with one dent, $b \geq 2$, and $2 < u < b+2$. We proceed by induction on $u$, using the cases addressed above as base cases.
   
 Assume that relation \ref{eq:main} holds for dented half-hexagons $V_{a,b,1,(v)}$ with $v<u$, and consider a tileable family of dented half-hexagons $\{V_{a,b,1,(u)}: a \in \mathbb Z\}.$
   
   Regard $V_{a,b,1,(u)}$ as a subregion of the (unbalanced) region $R := V_{a,b,1,\vec \emptyset}$.  With each region $R$ say: \alf\ denotes the $u$th triangle from the top lying along the northwest side; \bet\ denotes the 2nd triangle from the top lying along the northwest side; \gam\ denotes the southmost down-pointing triangle along the eastern side; \del\ denotes the southmost triangle along the northwest side. These positions are depicted within Figure \ref{fig:kuo1}.
   
\begin{figure}
    \begin{minipage}[c]{\textwidth}
    \centering

    \begin{tikzpicture}[x={(0:.45cm)},y={(120:.45cm)},z={(240:.45cm)}]
	\pgfmathtruncatemacro{\a}{4}
	\pgfmathtruncatemacro{\b}{2}
	\pgfmathtruncatemacro{\c}{2}
	\pgfmathtruncatemacro{\t}{6}

	\pgfmathtruncatemacro{\longdiag}{\a+\b+\c+\t+\t}
	
	\draw[very thick] (0,0,0) --++(0,0,-1) --++(0,1,0) --++(0,0,-1) --++
	(-4,0,0)%Change 'a'
	-- ++(0,0,3)
	-- ++ (0,-1,0) -- ++(-1,0,0)
	-- ++ (0,0,1) -- ++ (0,0,0) -- ++ (0,-3,0) --++ 
	(5,0,0)%Change 'a'
	--++(0,1,0)
	--++(0,0,-1) --++(0,1,0)
	--++(0,0,-1) -- cycle;
	
	\draw[clip] (0,0,0) --++(0,0,-1) --++(0,1,0) --++(0,0,-1) --++
	(-4,0,0)%Change 'a'
	-- ++(0,0,4)
	%-- ++ (0,-1,0) -- ++(-1,0,0)
	-- ++ (0,0,1) -- ++ (0,0,0) -- ++ (0,-3,0) --++ 
	(5,0,0)%Change 'a'
	--++(0,1,0)
	--++(0,0,-1) --++(0,1,0)
	--++(0,0,-1) -- cycle;

	\foreach \i in {0,...,\longdiag}{
        \draw[-] (\i,0,\c+\t) -- ++(0,\longdiag,0);
        \draw[-] (\a+\t,\i,0) -- ++(0,0,\longdiag);
        \draw[-] (0,\b+\t,\i) -- ++(\longdiag,0,0);
    }
	
	\node (A) at (-4,.66,1.33) {\alf};
    \node (B) at (-4,.66,-.66) {\bet};
	\node (C) at (-.66,-2.33,1) {\gam};
	\node (D) at (-4,.66,2.33) {\del};
	
    \end{tikzpicture}

    \caption{$R$ defined with respect to the region $V_{4,3,1,(4)}$ with \alf, \bet, \gam, \del \  marked. The outline of $V_{4,3,1,(4)}$ is drawn thickly. Note $R$ is the original region with its dent ``filled in.''}
       \labelnote{fig:kuo1}
    \end{minipage}
\end{figure}
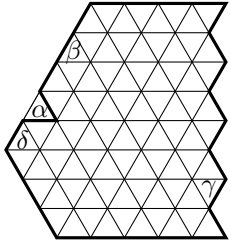

   We will apply Kuo's condensation method (Lemma \ref{kuo}, equation \ref{eqn2}) to $R$, \alf, \bet, \gam, \del. We use subscript to denote triangles that are removed from $R$, so that $R_\alf = V_{a,b,1,(u)}$ and the condensation formula reads
   \begin{equation}\labelnote{eq:kuo1}
       M(R_\alf)M(R_{\bet,\gam,\del}) + M(R_\gam)M(R_{\alf,\bet,\del})= M(R_\bet)M(R_{\alf,\gam,\del}) + M(R_\del)M(R_{\alf,\bet,\gam}).
   \end{equation}
   
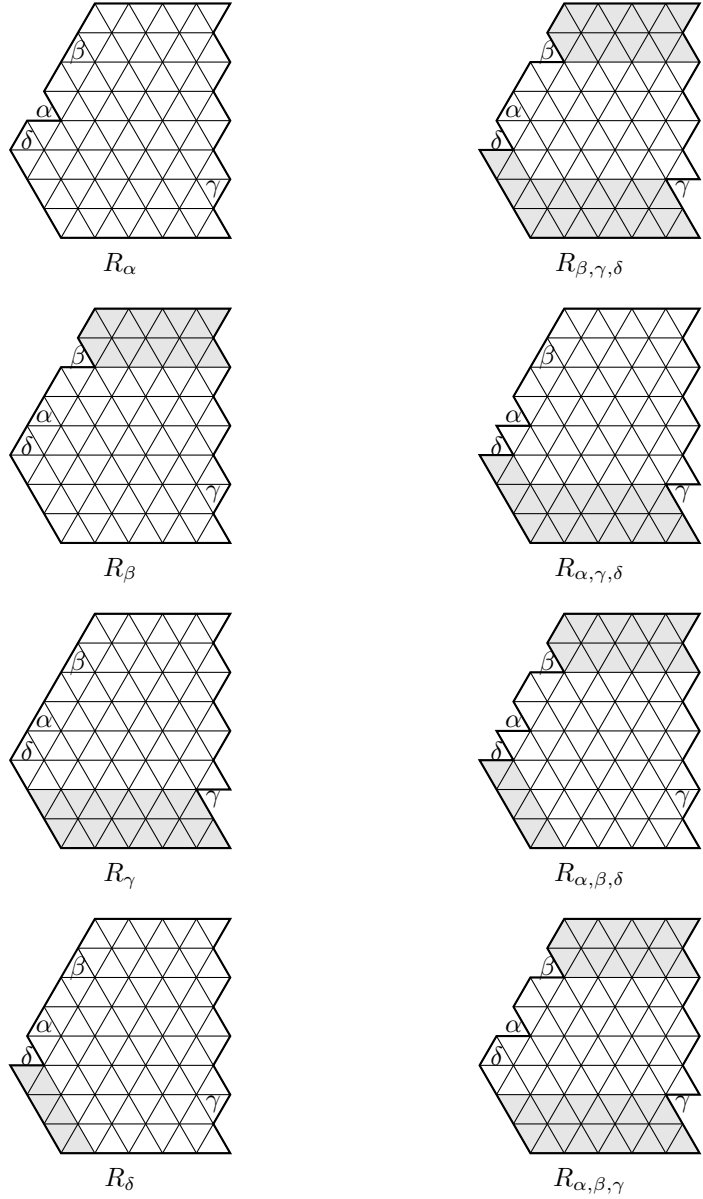
\begin{figure}
    \begin{minipage}[c]{\textwidth}
    \centering
    
    %%First Row
    \begin{multicols}{2}

    \begin{tikzpicture}[x={(0:.45cm)},y={(120:.45cm)},z={(240:.45cm)}]
	\pgfmathtruncatemacro{\a}{4}
	\pgfmathtruncatemacro{\b}{2}
	\pgfmathtruncatemacro{\c}{2}
	\pgfmathtruncatemacro{\t}{6}
	
	\node (A) at (-4,.66,1.33) {\alf};
    \node (B) at (-4,.66,-.66) {\bet};
	\node (C) at (-.66,-2.33,1) {\gam};
	\node (D) at (-4,.66,2.33) {\del};
	
	\pgfmathtruncatemacro{\longdiag}{\a+\b+\c+\t+\t}
	
	\draw[thick] (0,0,0) --++(0,0,-1) --++(0,1,0) --++(0,0,-1) --++
	(-4,0,0)%Change 'a'
	-- ++(0,0,3)
	-- ++ (0,-1,0) -- ++(-1,0,0)
	-- ++ (0,0,1) -- ++ (0,0,0) -- ++ (0,-3,0) --++ 
	(5,0,0)%Change 'a'
	--++(0,1,0)
	--++(0,0,-1) --++(0,1,0)
	--++(0,0,-1) -- cycle;
	
	\draw[clip] (0,0,0) --++(0,0,-1) --++(0,1,0) --++(0,0,-1) --++
	(-4,0,0)%Change 'a'
	-- ++(0,0,3)
	-- ++ (0,-1,0) -- ++(-1,0,0)
	-- ++ (0,0,1) -- ++ (0,0,0) -- ++ (0,-3,0) --++ 
	(5,0,0)%Change 'a'
	--++(0,1,0)
	--++(0,0,-1) --++(0,1,0)
	--++(0,0,-1) -- cycle;

	\foreach \i in {0,...,\longdiag}{
        \draw[-] (\i,0,\c+\t) -- ++(0,\longdiag,0);
        \draw[-] (\a+\t,\i,0) -- ++(0,0,\longdiag);
        \draw[-] (0,\b+\t,\i) -- ++(\longdiag,0,0);
    }
	
    \end{tikzpicture}
    
    $R_\alf$
    
    \columnbreak

    \begin{tikzpicture}[x={(0:.45cm)},y={(120:.45cm)},z={(240:.45cm)}]
	\pgfmathtruncatemacro{\a}{4}
	\pgfmathtruncatemacro{\b}{2}
	\pgfmathtruncatemacro{\c}{2}
	\pgfmathtruncatemacro{\t}{6}
	
	\pgfmathtruncatemacro{\longdiag}{\a+\b+\c+\t+\t}
	
	\node (A) at (-4,.66,1.33) {\alf};
    \node (B) at (-4,.66,-.66) {\bet};
	\node (C) at (-.66,-2.33,1) {\gam};
	\node (D) at (-4,.66,2.33) {\del};
	
	%Remove beta
	\fill[black!10] (-5,0,-2) --++(0,-1,0)--++(4,0,0)--++(0,1,0)--++(0,0,-1)--++(-4,0,0)-- cycle;
	
	%Remove gamma
	\fill[black!10] (0,-1,2) --++(-5,0,0)--++(0,-2,0)--++(5,0,0)-- cycle;
	
	%\node (A) at (-5,0,2){A};
	\fill[black!10] (-5,0,2) --++(0,-3,0)--++(1,0,0)--++(0,3,0)--cycle;

	\draw[thick] (0,0,0) --++(0,0,-1) --++(0,1,0) --++(0,0,-1) --++
	(-4,0,0)%Change 'a'
	-- ++(0,0,1)
	-- ++ (0,-1,0) -- ++(-1,0,0)
	-- ++(0,0,2)
	-- ++ (0,-1,0) -- ++(-1,0,0)
	-- ++ (0,0,0) -- ++ (0,0,0) -- ++ (0,-3,0) --++ 
	(5,0,0)%Change 'a'
	--++(0,2,0)
	--++(1,0,0) --++(0,1,0)
	--++(0,0,-1) -- cycle;
	
	\draw[clip] (0,0,0) --++(0,0,-1) --++(0,1,0) --++(0,0,-1) --++
	(-4,0,0)%Change 'a'
	-- ++(0,0,1)
	-- ++ (0,-1,0) -- ++(-1,0,0)
	-- ++(0,0,2)
	-- ++ (0,-1,0) -- ++(-1,0,0)
	-- ++ (0,0,0) -- ++ (0,0,0) -- ++ (0,-3,0) --++ 
	(5,0,0)%Change 'a'
	--++(0,2,0)
	--++(1,0,0) --++(0,1,0)
	--++(0,0,-1) -- cycle;

	\foreach \i in {0,...,\longdiag}{
        \draw[-] (\i,0,\c+\t) -- ++(0,\longdiag,0);
        \draw[-] (\a+\t,\i,0) -- ++(0,0,\longdiag);
        \draw[-] (0,\b+\t,\i) -- ++(\longdiag,0,0);
    }
	
	%\node (A) at (-4,.66,1.33) {\alf};
    %\node (B) at (-4,.66,-.66) {\bet};
	%\node (C) at (-.66,-2.33,1) {\gam};
	%\node (D) at (-4,.66,2.33) {\del};
	
    \end{tikzpicture}
    
    $R_{\bet,\gam,\del}$
    
    \end{multicols}
    
    %%Second Row
    \begin{multicols}{2}

    \begin{tikzpicture}[x={(0:.45cm)},y={(120:.45cm)},z={(240:.45cm)}]
	\pgfmathtruncatemacro{\a}{4}
	\pgfmathtruncatemacro{\b}{2}
	\pgfmathtruncatemacro{\c}{2}
	\pgfmathtruncatemacro{\t}{6}
	
	\pgfmathtruncatemacro{\longdiag}{\a+\b+\c+\t+\t}
	
	\node (A) at (-4,.66,1.33) {\alf};
    \node (B) at (-4,.66,-.66) {\bet};
	\node (C) at (-.66,-2.33,1) {\gam};
	\node (D) at (-4,.66,2.33) {\del};
	
	%Remove beta
	\fill[black!10] (-5,0,-2) --++(0,-1,0)--++(4,0,0)--++(0,1,0)--++(0,0,-1)--++(-4,0,0)-- cycle;
	
	\draw[thick] (0,0,0) --++(0,0,-1) --++(0,1,0) --++(0,0,-1) --++
	(-4,0,0)%Change 'a'
	-- ++(0,0,1)
	-- ++ (0,-1,0) -- ++(-1,0,0)
	-- ++(0,0,3)
	-- ++ (0,0,0) -- ++ (0,0,0) -- ++ (0,-3,0) --++ 
	(5,0,0)%Change 'a'
	--++(0,1,0)
	--++(0,0,-1) --++(0,1,0)
	--++(0,0,-1) -- cycle;
	
	\draw[clip]  (0,0,0) --++(0,0,-1) --++(0,1,0) --++(0,0,-1) --++
	(-4,0,0)%Change 'a'
	-- ++(0,0,1)
	-- ++ (0,-1,0) -- ++(-1,0,0)
	-- ++(0,0,3)
	-- ++ (0,0,0) -- ++ (0,0,0) -- ++ (0,-3,0) --++ 
	(5,0,0)%Change 'a'
	--++(0,1,0)
	--++(0,0,-1) --++(0,1,0)
	--++(0,0,-1) -- cycle;

	\foreach \i in {0,...,\longdiag}{
        \draw[-] (\i,0,\c+\t) -- ++(0,\longdiag,0);
        \draw[-] (\a+\t,\i,0) -- ++(0,0,\longdiag);
        \draw[-] (0,\b+\t,\i) -- ++(\longdiag,0,0);
    }

    \end{tikzpicture}
    
    $R_\bet$
    
    \columnbreak

    \begin{tikzpicture}[x={(0:.45cm)},y={(120:.45cm)},z={(240:.45cm)}]
	\pgfmathtruncatemacro{\a}{4}
	\pgfmathtruncatemacro{\b}{2}
	\pgfmathtruncatemacro{\c}{2}
	\pgfmathtruncatemacro{\t}{6}
	
	\pgfmathtruncatemacro{\longdiag}{\a+\b+\c+\t+\t}

	\node (A) at (-4,.66,1.33) {\alf};
    \node (B) at (-4,.66,-.66) {\bet};
	\node (C) at (-.66,-2.33,1) {\gam};
	\node (D) at (-4,.66,2.33) {\del};
	
	%Remove gamma
	\fill[black!10] (0,-1,2) --++(-5,0,0)--++(0,-2,0)--++(5,0,0)-- cycle;
	
	%\node (A) at (-5,0,2){A};
	\fill[black!10] (-5,0,2) --++(0,-3,0)--++(1,0,0)--++(0,3,0)--cycle;

	\draw[thick] (0,0,0) --++(0,0,-1) --++(0,1,0) --++(0,0,-1) --++
	(-4,0,0)%Change 'a'
	-- ++(0,0,3)
	-- ++ (0,-1,0) -- ++(-1,0,0)
	-- ++(0,0,0)
	-- ++ (0,-1,0) -- ++(-1,0,0)
	-- ++ (0,0,0) -- ++ (0,0,0) -- ++ (0,-3,0) --++ 
	(5,0,0)%Change 'a'
	--++(0,2,0)
	--++(1,0,0) --++(0,1,0)
	--++(0,0,-1) -- cycle;
	
	\draw[clip] (0,0,0) --++(0,0,-1) --++(0,1,0) --++(0,0,-1) --++
	(-4,0,0)%Change 'a'
	-- ++(0,0,3)
	-- ++ (0,-1,0) -- ++(-1,0,0)
	-- ++(0,0,0)
	-- ++ (0,-1,0) -- ++(-1,0,0)
	-- ++ (0,0,0) -- ++ (0,0,0) -- ++ (0,-3,0) --++ 
	(5,0,0)%Change 'a'
	--++(0,2,0)
	--++(1,0,0) --++(0,1,0)
	--++(0,0,-1) -- cycle;

	\foreach \i in {0,...,\longdiag}{
        \draw[-] (\i,0,\c+\t) -- ++(0,\longdiag,0);
        \draw[-] (\a+\t,\i,0) -- ++(0,0,\longdiag);
        \draw[-] (0,\b+\t,\i) -- ++(\longdiag,0,0);
    }

    \end{tikzpicture}
    
    $R_{\alf,\gam,\del}$
    
    \end{multicols}

    %%Third Row
    \begin{multicols}{2}

    \begin{tikzpicture}[x={(0:.45cm)},y={(120:.45cm)},z={(240:.45cm)}]
	\pgfmathtruncatemacro{\a}{4}
	\pgfmathtruncatemacro{\b}{2}
	\pgfmathtruncatemacro{\c}{2}
	\pgfmathtruncatemacro{\t}{6}
	
	\pgfmathtruncatemacro{\longdiag}{\a+\b+\c+\t+\t}

	\node (A) at (-4,.66,1.33) {\alf};
    \node (B) at (-4,.66,-.66) {\bet};
	\node (C) at (-.66,-2.33,1) {\gam};
	\node (D) at (-4,.66,2.33) {\del};
	
	%Remove gamma
	\fill[black!10] (0,-1,2) --++(-5,0,0)--++(0,-2,0)--++(5,0,0)-- cycle;

	\draw[thick] (0,0,0) --++(0,0,-1) --++(0,1,0) --++(0,0,-1) --++
	(-4,0,0)%Change 'a'
	-- ++(0,0,2)
	%-- ++ (0,-1,0) -- ++(-1,0,0)
	-- ++(0,0,3)
	%-- ++ (0,-1,0) -- ++(-1,0,0)
	-- ++ (0,0,0) -- ++ (0,0,0) -- ++ (0,-3,0) --++ 
	(5,0,0)%Change 'a'
	--++(0,2,0)
	--++(1,0,0) --++(0,1,0)
	--++(0,0,-1) -- cycle;
	
	\draw[clip] (0,0,0) --++(0,0,-1) --++(0,1,0) --++(0,0,-1) --++
	(-4,0,0)%Change 'a'
	-- ++(0,0,2)
	%-- ++ (0,-1,0) -- ++(-1,0,0)
	-- ++(0,0,3)
	%-- ++ (0,-1,0) -- ++(-1,0,0)
	-- ++ (0,0,0) -- ++ (0,0,0) -- ++ (0,-3,0) --++ 
	(5,0,0)%Change 'a'
	--++(0,2,0)
	--++(1,0,0) --++(0,1,0)
	--++(0,0,-1) -- cycle;

	\foreach \i in {0,...,\longdiag}{
        \draw[-] (\i,0,\c+\t) -- ++(0,\longdiag,0);
        \draw[-] (\a+\t,\i,0) -- ++(0,0,\longdiag);
        \draw[-] (0,\b+\t,\i) -- ++(\longdiag,0,0);
    }

    \end{tikzpicture}
    
    $R_\gam$
    
    \columnbreak

    \begin{tikzpicture}[x={(0:.45cm)},y={(120:.45cm)},z={(240:.45cm)}]
	\pgfmathtruncatemacro{\a}{4}
	\pgfmathtruncatemacro{\b}{2}
	\pgfmathtruncatemacro{\c}{2}
	\pgfmathtruncatemacro{\t}{6}
	
	\pgfmathtruncatemacro{\longdiag}{\a+\b+\c+\t+\t}

	\node (A) at (-4,.66,1.33) {\alf};
    \node (B) at (-4,.66,-.66) {\bet};
	\node (C) at (-.66,-2.33,1) {\gam};
	\node (D) at (-4,.66,2.33) {\del};
	
	%Remove beta
	\fill[black!10] (-5,0,-2) --++(0,-1,0)--++(4,0,0)--++(0,1,0)--++(0,0,-1)--++(-4,0,0)-- cycle;
	
	%Remove delta
	\fill[black!10] (-5,0,2) --++(0,-3,0)--++(1,0,0)--++(0,3,0)--cycle;

	\draw[thick] (0,0,0) --++(0,0,-1) --++(0,1,0) --++(0,0,-1) --++
	(-4,0,0)%Change 'a'
	-- ++(0,0,1)
	-- ++ (0,-1,0) -- ++(-1,0,0)
	-- ++(0,0,1)
	-- ++ (0,-1,0) -- ++(-1,0,0)
	-- ++ (0,-1,0) -- ++(-1,0,0)
	-- ++ (0,0,0) -- ++ (0,0,0) -- ++ (0,-3,0) --++ 
	(5,0,0)%Change 'a'
	--++(0,1,0)
	--++(0,0,-1) --++(0,1,0)
	--++(0,0,-1) -- cycle;
	
	\draw[clip] (0,0,0) --++(0,0,-1) --++(0,1,0) --++(0,0,-1) --++
	(-4,0,0)%Change 'a'
	-- ++(0,0,1)
	-- ++ (0,-1,0) -- ++(-1,0,0)
	-- ++(0,0,1)
	-- ++ (0,-1,0) -- ++(-1,0,0)
	-- ++ (0,-1,0) -- ++(-1,0,0)
	-- ++ (0,0,0) -- ++ (0,0,0) -- ++ (0,-3,0) --++ 
	(5,0,0)%Change 'a'
	--++(0,1,0)
	--++(0,0,-1) --++(0,1,0)
	--++(0,0,-1) -- cycle;

	\foreach \i in {0,...,\longdiag}{
        \draw[-] (\i,0,\c+\t) -- ++(0,\longdiag,0);
        \draw[-] (\a+\t,\i,0) -- ++(0,0,\longdiag);
        \draw[-] (0,\b+\t,\i) -- ++(\longdiag,0,0);
    }

    \end{tikzpicture}
    
    $R_{\alf,\bet,\del}$
    
    \end{multicols}

    %%Fourth Row
    \begin{multicols}{2}

    \begin{tikzpicture}[x={(0:.45cm)},y={(120:.45cm)},z={(240:.45cm)}]
	\pgfmathtruncatemacro{\a}{4}
	\pgfmathtruncatemacro{\b}{2}
	\pgfmathtruncatemacro{\c}{2}
	\pgfmathtruncatemacro{\t}{6}
	
	\pgfmathtruncatemacro{\longdiag}{\a+\b+\c+\t+\t}
	
	\node (A) at (-4,.66,1.33) {\alf};
    \node (B) at (-4,.66,-.66) {\bet};
	\node (C) at (-.66,-2.33,1) {\gam};
	\node (D) at (-4,.66,2.33) {\del};
	
	%Remove delta
	\fill[black!10] (-5,0,2) --++(0,-3,0)--++(1,0,0)--++(0,3,0)--cycle;

	\draw[thick] (0,0,0) --++(0,0,-1) --++(0,1,0) --++(0,0,-1) --++
	(-4,0,0)%Change 'a'
	-- ++(0,0,2)
	%-- ++ (0,-1,0) -- ++(-1,0,0)
	-- ++(0,0,2)
	-- ++ (0,-1,0) -- ++(-1,0,0)
	-- ++ (0,0,0) -- ++ (0,0,0) -- ++ (0,-3,0) --++ 
	(5,0,0)%Change 'a'
	--++(0,1,0)
	--++(0,0,-1) --++(0,1,0)
	--++(0,0,-1) -- cycle;
	
	\draw[clip] (0,0,0) --++(0,0,-1) --++(0,1,0) --++(0,0,-1) --++
	(-4,0,0)%Change 'a'
	-- ++(0,0,1)
	-- ++ (0,-1,0) -- ++(-1,0,0)
	-- ++(0,0,2)
	-- ++ (0,-1,0) -- ++(-1,0,0)
	-- ++ (0,0,0) -- ++ (0,0,0) -- ++ (0,-3,0) --++ 
	(5,0,0)%Change 'a'
	--++(0,1,0)
	--++(0,0,-1) --++(0,1,0)
	--++(0,0,-1) -- cycle;

	\foreach \i in {0,...,\longdiag}{
        \draw[-] (\i,0,\c+\t) -- ++(0,\longdiag,0);
        \draw[-] (\a+\t,\i,0) -- ++(0,0,\longdiag);
        \draw[-] (0,\b+\t,\i) -- ++(\longdiag,0,0);
    }

    \end{tikzpicture}
    
    $R_\del$
    
    \columnbreak

    \begin{tikzpicture}[x={(0:.45cm)},y={(120:.45cm)},z={(240:.45cm)}]
	\pgfmathtruncatemacro{\a}{4}
	\pgfmathtruncatemacro{\b}{2}
	\pgfmathtruncatemacro{\c}{2}
	\pgfmathtruncatemacro{\t}{6}
	
	\pgfmathtruncatemacro{\longdiag}{\a+\b+\c+\t+\t}
	
	\node (A) at (-4,.66,1.33) {\alf};
    \node (B) at (-4,.66,-.66) {\bet};
	\node (C) at (-.66,-2.33,1) {\gam};
	\node (D) at (-4,.66,2.33) {\del};
	
	%Remove beta
	\fill[black!10] (-5,0,-2) --++(0,-1,0)--++(4,0,0)--++(0,1,0)--++(0,0,-1)--++(-4,0,0)-- cycle;
	
	%Remove gamma
	\fill[black!10] (0,-1,2) --++(-5,0,0)--++(0,-2,0)--++(5,0,0)-- cycle;

	\draw[thick] (0,0,0) --++(0,0,-1) --++(0,1,0) --++(0,0,-1) --++
	(-4,0,0)%Change 'a'
	-- ++(0,0,1)
	-- ++ (0,-1,0) -- ++(-1,0,0)
	-- ++(0,0,1)
	-- ++ (0,-1,0) -- ++(-1,0,0)
	-- ++ (0,0,1) -- ++ (0,0,0) -- ++ (0,-3,0) --++ 
	(5,0,0)%Change 'a'
	--++(0,2,0)
	--++(1,0,0) --++(0,1,0)
	--++(0,0,-1) -- cycle;
	
	\draw[clip] (0,0,0) --++(0,0,-1) --++(0,1,0) --++(0,0,-1) --++
	(-4,0,0)%Change 'a'
	-- ++(0,0,1)
	-- ++ (0,-1,0) -- ++(-1,0,0)
	-- ++(0,0,1)
	-- ++ (0,-1,0) -- ++(-1,0,0)
	-- ++ (0,0,1) -- ++ (0,0,0) -- ++ (0,-3,0) --++ 
	(5,0,0)%Change 'a'
	--++(0,2,0)
	--++(1,0,0) --++(0,1,0)
	--++(0,0,-1) -- cycle;

	\foreach \i in {0,...,\longdiag}{
        \draw[-] (\i,0,\c+\t) -- ++(0,\longdiag,0);
        \draw[-] (\a+\t,\i,0) -- ++(0,0,\longdiag);
        \draw[-] (0,\b+\t,\i) -- ++(\longdiag,0,0);
    }

    \end{tikzpicture}
    
    $R_{\alf,\bet,\gam}$
    
    \end{multicols}

    \caption{Each of the regions defined by removing an odd subset of $\{\alpha, \beta, \gamma, \delta\}$ from $R$. Each region with $\beta, \gamma, \delta$ removed has forced lozenges along the top, bottom, and southwest respectively.}
       \labelnote{fig:kuoregions1}
    \end{minipage}
\end{figure}
   
   It is worth noting that $R_{\alpha,\beta,\gamma}$ may be untileable, and $R_{\gamma}$ and $R_{\alpha, \beta, \delta}$ are unbalanced. Since each of the other regions among these is tileable and has forced lozenges that leave half-hexagons or dented half-hexagons once removed (as depicted in Figure \ref{fig:kuoregions1}), we can rewrite equation \ref{eq:kuo1} in context.
   $$f_{b,1,(u)}(a) P_|(a+1,b-1) = P_|(a+1,b) f_{b-1,1,(u)}(a) + P_|(a,b+1) f_{b-2,1,(u-2)}(a+1)$$
   
   By the inductive hypothesis $f_{b-1,1,(u)}(a) \aequiv F_{b-1,1,(u)}(a)$ and\\ $f_{b-2,1,(u-2)}(a) \aequiv F_{b-2,1,(u-2)}(a)$.
   By relations $\ref{eq:equiv1}$ and $\ref{eq:equiv2}$, in order to show $f_{b,1,(u)}(a) \aequiv F_{b,1,(u)}(a)$ it suffices to show\footnote{In the case that $R_{\alpha, \beta, \gamma}$ is untileable, the last term in relation \ref{eq:trio1} is identically zero. In this case it is neither possible nor necessary to show it is $\aequiv$-equivalent to the other terms, so we proceed assuming it is not identically zero.} that
   \begin{equation}\labelnote{eq:trio1}
       F_{b,1,(u)}(a) P_|(a+1,b-1) \aequiv P_|(a+1,b) F_{b-1,1,(u)}(a) \aequiv P_|(a,b+1) F_{b-2,1,(u-2)}(a+1).
   \end{equation}
   
   We check that the first and last products are $\aequiv$-equivalent by expanding $F$.
   \begin{align*}
       F_{b,1,(u)}(a) P_|(a+1,b-1) &\aequiv P_|(a,b+1) F_{b-2,1,(u-2)}(a+1)\\[10pt]
       \dfrac{P_|(a,b+1)}{(2a+u)_{b+2-u}} P_|(a+1,b-1) &\aequiv P_|(a,b+1) \dfrac{P_|(a+1,b-1)}{(2a+u)_{b+2-u}}%\\[10pt]
   \end{align*}
   
   It remains to show that $P_|(a+1,b) F_{b-1,1,(u)}(a)$ is $\aequiv$-equivalent to these. We will make liberal use of the following identity, which is straightforward to check.
   \begin{equation}
       \dfrac{P_|(a,b)}{P_|(a,b-1)} \aequiv (a+b-1)(2a+b)_{b-2}
   \end{equation}
   We apply this identity twice in turn.
   \begin{align}
       \dfrac{P_|(a,b+1)}{P_|(a,b)(2a+b+1)} \aequiv &  (a+b)(2a+b+2)_{b-2} \labelnote{eq:techlem1}\\[10pt]
       \dfrac{P_|(a+1,b)}{P_|(a+1,b-1)} \aequiv & (a+b)(2a+b+2)_{b-2} \labelnote{eq:techlem2}
   \end{align}
   Now we rewrite the first and second product from \ref{eq:trio1} explicitly.
   \begin{equation}
       \dfrac{P_|(a,b+1)}{(2a+u)_{b+2-u}} P_|(a+1,b-1) = P_|(a+1)\dfrac{P_|(a,b)}{(2a+u)_{b+1-u}}
   \end{equation}
   We multiply both sides by $\frac{(2a+u)_{b+1-u}}{P_|(a,b)P_|(a+1,b-1)}$.
   \begin{equation}
       \dfrac{P_|(a,b+1)}{P_|(a,b)(2a+b+1)} \aequiv \dfrac{P_|(a+1,b)}{P_|(a+1,b-1)}
   \end{equation}
   We apply relations \ref{eq:techlem1} and \ref{eq:techlem2} to get
   \begin{equation}
       \dfrac{(a+b)(2a+b+1)_{b-1}}{(2a+b+1)} \aequiv (a+b)(2a+b+2)_{b-2}
   \end{equation}
   which holds since both sides are equal. Thus, $f_{b,1,(u)}(a) \aequiv F_{b,1,(u)}(a)$ whenever $2 \leq u \leq b+2n$.

\end{proof}

Our full proof of Theorem \ref{main} is quite similar to the previous proof of the special case.

\begin{proof}[Proof of Theorem \ref{main}]
    We will first show relation \eqref{eq:main} holds in general, by inducting on the number of dents, with regions with one or no dents as a base case which is covered by Lemma \ref{basecase}.
    
    For our inductive hypothesis, assume that relation \eqref{eq:main} holds for dented half-hexagons with fewer than $n$ dents, and consider a tileable family of dented half-hexagons $\{V_{a,b,n,(u_i)_1^n}: a \in \mathbb Z\}$. 
    
    In the case $u_1=2$, there are forced lozenges at the top of the region, and
    $$f_{b,n,(u_i)_1^n}(a) = f_{b,n-1,(u_i-2)_2^n}(a+1) \aequiv F_{b,n-1,(u_i-2)_2^n}(a+1) \aequiv F_{b,n,(u_i)_1^n}(a)$$
    by the inductive hypothesis and Lemma \ref{case}.1.
    
    In the case $u_n=b+2n$, there are forced lozenges along the southwest boundary of the region, and
    $$f_{b,n,(u_i)_1^{n}}(a) = f_{b+1,n-1,(u_i)_1^{n-1}}(a) \aequiv F_{b+1,n-1,(u_i)_1^{n-1}}(a) \aequiv F_{b,n,(u_i)_1^{n}}(a) $$
    by the inductive hypothesis and Lemma \ref{case}.2.
    
    In the case that $2 < u_1 < u_n < b+2n$ we proceed using Kuo's condensation method. Regard each $V_{a,b,n,(u_i)_1^n}$ as a subregion of $R:=V_{a,b,n,(u_i)_2^{n-1}}$, that is, the original region without its first and last dent. Within each $R$, say: {\alf} denotes the second triangle from the top along the northwest side; {\bet} denotes the $u_1$st triangle from the top along the northwest side; {\gam} denotes the $u_n$th triangle from the top along the northwest side; and {\del} denotes the bottom triangle along the northwest side.
    These positions are depicted within Figure \ref{fig:kuo2}.
   
\begin{figure}
    \begin{minipage}[c]{\textwidth}
    \centering

    \begin{tikzpicture}[x={(0:.45cm)},y={(120:.45cm)},z={(240:.45cm)}]
	\pgfmathtruncatemacro{\a}{4}
	\pgfmathtruncatemacro{\b}{3}
	\pgfmathtruncatemacro{\c}{3}
	\pgfmathtruncatemacro{\t}{6}

	\pgfmathtruncatemacro{\longdiag}{\a+\b+\c+\t+\t}
	
	\node (A) at (-4.33,.33,-1){\alf};
	\node (B) at (-4.33,.33,1){\bet};
	\node (G) at (-4.33,.33,4){\gam};
	\node (D) at (-4.33,.33,6){\del};
	
	\draw[very thick] (0,0,0) --++(0,0,-1) --++(0,1,0) --++(0,0,-1) --++
	(-4,0,0)%Change 'a'
	-- ++(0,0,3)
	-- ++ (0,-1,0) -- ++(-1,0,0)
	-- ++ (0,0,0)
	-- ++ (0,-1,0) -- ++(-1,0,0)
	-- ++ (0,0,1)
	-- ++ (0,-1,0) -- ++(-1,0,0) 
	--++ (0,0,2) -- ++ (0,-3,0) --++ 
	(7,0,0)%Change 'a'
	--++(0,1,0)
	--++(0,0,-1) --++(0,1,0)
	--++(0,0,-1) --++(0,1,0)
	--++(0,0,-1) --++(0,1,0)
	--++(0,0,-1) -- cycle;
	
	\draw[clip] (0,0,0) --++(0,0,-1) --++(0,1,0) --++(0,0,-1) --++
	(-4,0,0)%Change 'a'
	-- ++(0,0,4)
	-- ++ (0,-1,0) -- ++(-1,0,0)
	-- ++ (0,0,2)
	%-- ++ (0,-1,0) -- ++(-1,0,0)
	-- ++ (0,0,2)
	%-- ++ (0,-1,0) -- ++(-1,0,0) 
	--++ (0,0,0) -- ++ (0,-3,0) --++ 
	(7,0,0)%Change 'a'
	--++(0,1,0)
	--++(0,0,-1) --++(0,1,0)
	--++(0,0,-1) --++(0,1,0)
	--++(0,0,-1) --++(0,1,0)
	--++(0,0,-1) -- cycle;

	\foreach \i in {0,...,\longdiag}{
        \draw[-] (\i-1,0,\c+\t) -- ++(0,\longdiag,0);
        \draw[-] (\a+\t,\i,0) -- ++(0,0,\longdiag);
        \draw[-] (0,\b+\t,\i) -- ++(\longdiag,0,0);
    }

    \end{tikzpicture}

    \caption{$R$ defined with respect to the region $V_{4,3,3,(4,5,7)}$ with that region's outline drawn thickly and \alf, \bet, \gam, \del\  marked. Note, $R$ is the original region with its first and last dent ``filled in.''}% It is easy to see $R_\alf$ is $V_{a,b,1,(u)}$.}
       \labelnote{fig:kuo2}
    \end{minipage}
\end{figure}
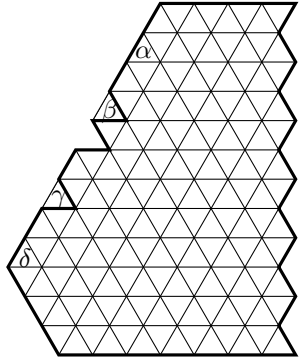
   
      We will apply Kuo's condensation method (Lemma \ref{kuo}, equation \ref{eqn}) to $R$, \alf, \bet, \gam, \del. We use subscript to denote triangles that are removed from $R$, so that $R_{\bet,\gam} = V_{a,b,n, (u_i)_1^n}$ and the condensation formula reads
   \begin{equation}\labelnote{eq:kuo2}
       M(R_{\bet,\gam})M(R_{\alf,\del}) = M(R_{\alf,\gam})M(R_{\bet,\del}) - M(R_{\alf,\bet})M(R_{\gam,\del})
   \end{equation}
   
\begin{figure}
    \begin{minipage}[c]{\textwidth}
    \centering
    
    %%Row 1
    \begin{multicols}{2}
    \begin{tikzpicture}[x={(0:.45cm)},y={(120:.45cm)},z={(240:.45cm)}]
	\pgfmathtruncatemacro{\a}{4}
	\pgfmathtruncatemacro{\b}{3}
	\pgfmathtruncatemacro{\c}{3}
	\pgfmathtruncatemacro{\t}{6}
	
	\pgfmathtruncatemacro{\longdiag}{\a+\b+\c+\t+\t}
	
	\node (A) at (-4.33,.33,-1){\alf};
	\node (B) at (-4.33,.33,1){\bet};
	\node (G) at (-4.33,.33,4){\gam};
	\node (D) at (-4.33,.33,6){\del};
	
	%Remove Alpha
	\fill[black!10] (-5,0,-2) --++(0,-1,0)--++(4,0,0)--++(0,1,0)--++(0,0,-1)--++(-4,0,0)-- cycle;
	%Remove Delta
	%\fill[black!10](-5,0,6)--++(0,-3,0)--++(1,0,0)--++(0,3,0)--cycle;
	
	\draw[very thick] (0,0,0) --++(0,0,-1) --++(0,1,0) --++(0,0,-1) --++
	(-4,0,0)%Change 'a'
	-- ++(0,0,1)
	-- ++ (0,-1,0) -- ++(-1,0,0)
	-- ++ (0,0,1)
	-- ++ (0,-1,0) -- ++(-1,0,0)
	-- ++ (0,0,0)
	-- ++ (0,-1,0) -- ++(-1,0,0) 
	--++ (0,0,4) -- ++ (0,-3,0) --++ 
	(7,0,0)%Change 'a'
	--++(0,1,0)
	--++(0,0,-1) --++(0,1,0)
	--++(0,0,-1) --++(0,1,0)
	--++(0,0,-1) --++(0,1,0)
	--++(0,0,-1) -- cycle;
	
	\draw[clip] (0,0,0) --++(0,0,-1) --++(0,1,0) --++(0,0,-1) --++
	(-4,0,0)%Change 'a'
	-- ++(0,0,1)
	-- ++ (0,-1,0) -- ++(-1,0,0)
	-- ++ (0,0,1)
	-- ++ (0,-1,0) -- ++(-1,0,0)
	-- ++ (0,0,0)
	-- ++ (0,-1,0) -- ++(-1,0,0) 
	--++ (0,0,4) -- ++ (0,-3,0) --++ 
	(7,0,0)%Change 'a'
	--++(0,1,0)
	--++(0,0,-1) --++(0,1,0)
	--++(0,0,-1) --++(0,1,0)
	--++(0,0,-1) --++(0,1,0)
	--++(0,0,-1) -- cycle;

	\foreach \i in {0,...,\longdiag}{
        \draw[-] (\i-1,0,\c+\t) -- ++(0,\longdiag,0);
        \draw[-] (\a+\t,\i,0) -- ++(0,0,\longdiag);
        \draw[-] (0,\b+\t,\i) -- ++(\longdiag,0,0);
    }

    \end{tikzpicture}

    $R_{\alf,\bet}$
    
    \columnbreak

    \begin{tikzpicture}[x={(0:.45cm)},y={(120:.45cm)},z={(240:.45cm)}]
	\pgfmathtruncatemacro{\a}{4}
	\pgfmathtruncatemacro{\b}{3}
	\pgfmathtruncatemacro{\c}{3}
	\pgfmathtruncatemacro{\t}{6}
	
	\pgfmathtruncatemacro{\longdiag}{\a+\b+\c+\t+\t}
	
	\node (A) at (-4.33,.33,-1){\alf};
	\node (B) at (-4.33,.33,1){\bet};
	\node (G) at (-4.33,.33,4){\gam};
	\node (D) at (-4.33,.33,6){\del};
	
	%Remove Alpha
	%\fill[black!10] (-5,0,-2) --++(0,-1,0)--++(4,0,0)--++(0,1,0)--++(0,0,-1)--++(-4,0,0)-- cycle;
	%Remove Delta
	\fill[black!10](-5,0,6)--++(0,-3,0)--++(1,0,0)--++(0,3,0)--cycle;
	
	\draw[very thick] (0,0,0) --++(0,0,-1) --++(0,1,0) --++(0,0,-1) --++
	(-4,0,0)%Change 'a'
	-- ++(0,0,4)
	-- ++ (0,-1,0) -- ++(-1,0,0)
	-- ++ (0,0,1)
	-- ++ (0,-1,0) -- ++(-1,0,0)
	-- ++ (0,0,1)
	-- ++ (0,-1,0) -- ++(-1,0,0) 
	--++ (0,0,0) -- ++ (0,-3,0) --++ 
	(7,0,0)%Change 'a'
	--++(0,1,0)
	--++(0,0,-1) --++(0,1,0)
	--++(0,0,-1) --++(0,1,0)
	--++(0,0,-1) --++(0,1,0)
	--++(0,0,-1) -- cycle;
	
	\draw[clip] (0,0,0) --++(0,0,-1) --++(0,1,0) --++(0,0,-1) --++
	(-4,0,0)%Change 'a'
	-- ++(0,0,4)
	-- ++ (0,-1,0) -- ++(-1,0,0)
	-- ++ (0,0,1)
	-- ++ (0,-1,0) -- ++(-1,0,0)
	-- ++ (0,0,1)
	-- ++ (0,-1,0) -- ++(-1,0,0) 
	--++ (0,0,0) -- ++ (0,-3,0) --++ 
	(7,0,0)%Change 'a'
	--++(0,1,0)
	--++(0,0,-1) --++(0,1,0)
	--++(0,0,-1) --++(0,1,0)
	--++(0,0,-1) --++(0,1,0)
	--++(0,0,-1) -- cycle;

	\foreach \i in {0,...,\longdiag}{
        \draw[-] (\i-1,0,\c+\t) -- ++(0,\longdiag,0);
        \draw[-] (\a+\t,\i,0) -- ++(0,0,\longdiag);
        \draw[-] (0,\b+\t,\i) -- ++(\longdiag,0,0);
    }

    \end{tikzpicture}
    
    $R_{\gam,\del}$
    \end{multicols}
    
    %%Row 2
    \begin{multicols}{2}
    \begin{tikzpicture}[x={(0:.45cm)},y={(120:.45cm)},z={(240:.45cm)}]
	\pgfmathtruncatemacro{\a}{4}
	\pgfmathtruncatemacro{\b}{3}
	\pgfmathtruncatemacro{\c}{3}
	\pgfmathtruncatemacro{\t}{6}
	
	\pgfmathtruncatemacro{\longdiag}{\a+\b+\c+\t+\t}
	
	\node (A) at (-4.33,.33,-1){\alf};
	\node (B) at (-4.33,.33,1){\bet};
	\node (G) at (-4.33,.33,4){\gam};
	\node (D) at (-4.33,.33,6){\del};
	
	%Remove Alpha
	\fill[black!10] (-5,0,-2) --++(0,-1,0)--++(4,0,0)--++(0,1,0)--++(0,0,-1)--++(-4,0,0)-- cycle;
	%Remove Delta
	%\fill[black!10](-5,0,6)--++(0,-3,0)--++(1,0,0)--++(0,3,0)--cycle;
	
	\draw[very thick] (0,0,0) --++(0,0,-1) --++(0,1,0) --++(0,0,-1) --++
	(-4,0,0)%Change 'a'
	-- ++(0,0,1)
	-- ++ (0,-1,0) -- ++(-1,0,0)
	-- ++ (0,0,2)
	-- ++ (0,-1,0) -- ++(-1,0,0)
	-- ++ (0,0,1)
	-- ++ (0,-1,0) -- ++(-1,0,0) 
	--++ (0,0,2) -- ++ (0,-3,0) --++ 
	(7,0,0)%Change 'a'
	--++(0,1,0)
	--++(0,0,-1) --++(0,1,0)
	--++(0,0,-1) --++(0,1,0)
	--++(0,0,-1) --++(0,1,0)
	--++(0,0,-1) -- cycle;
	
	\draw[clip] (0,0,0) --++(0,0,-1) --++(0,1,0) --++(0,0,-1) --++
	(-4,0,0)%Change 'a'
	-- ++(0,0,1)
	-- ++ (0,-1,0) -- ++(-1,0,0)
	-- ++ (0,0,2)
	-- ++ (0,-1,0) -- ++(-1,0,0)
	-- ++ (0,0,1)
	-- ++ (0,-1,0) -- ++(-1,0,0) 
	--++ (0,0,2) -- ++ (0,-3,0) --++ 
	(7,0,0)%Change 'a'
	--++(0,1,0)
	--++(0,0,-1) --++(0,1,0)
	--++(0,0,-1) --++(0,1,0)
	--++(0,0,-1) --++(0,1,0)
	--++(0,0,-1) -- cycle;

	\foreach \i in {0,...,\longdiag}{
        \draw[-] (\i-1,0,\c+\t) -- ++(0,\longdiag,0);
        \draw[-] (\a+\t,\i,0) -- ++(0,0,\longdiag);
        \draw[-] (0,\b+\t,\i) -- ++(\longdiag,0,0);
    }

    \end{tikzpicture}

    $R_{\alf,\gam}$
    
    \columnbreak

    \begin{tikzpicture}[x={(0:.45cm)},y={(120:.45cm)},z={(240:.45cm)}]
	\pgfmathtruncatemacro{\a}{4}
	\pgfmathtruncatemacro{\b}{3}
	\pgfmathtruncatemacro{\c}{3}
	\pgfmathtruncatemacro{\t}{6}
	
	\pgfmathtruncatemacro{\longdiag}{\a+\b+\c+\t+\t}
	
	\node (A) at (-4.33,.33,-1){\alf};
	\node (B) at (-4.33,.33,1){\bet};
	\node (G) at (-4.33,.33,4){\gam};
	\node (D) at (-4.33,.33,6){\del};
	
	%Remove Alpha
	%\fill[black!10] (-5,0,-2) --++(0,-1,0)--++(4,0,0)--++(0,1,0)--++(0,0,-1)--++(-4,0,0)-- cycle;
	%Remove Delta
	\fill[black!10](-5,0,6)--++(0,-3,0)--++(1,0,0)--++(0,3,0)--cycle;
	
	\draw[very thick] (0,0,0) --++(0,0,-1) --++(0,1,0) --++(0,0,-1) --++
	(-4,0,0)%Change 'a'
	-- ++(0,0,3)
	-- ++ (0,-1,0) -- ++(-1,0,0)
	-- ++ (0,0,0)
	-- ++ (0,-1,0) -- ++(-1,0,0)
	-- ++ (0,0,3)
	-- ++ (0,-1,0) -- ++(-1,0,0) 
	--++ (0,0,0) -- ++ (0,-3,0) --++ 
	(7,0,0)%Change 'a'
	--++(0,1,0)
	--++(0,0,-1) --++(0,1,0)
	--++(0,0,-1) --++(0,1,0)
	--++(0,0,-1) --++(0,1,0)
	--++(0,0,-1) -- cycle;
	
	\draw[clip] (0,0,0) --++(0,0,-1) --++(0,1,0) --++(0,0,-1) --++
	(-4,0,0)%Change 'a'
	-- ++(0,0,3)
	-- ++ (0,-1,0) -- ++(-1,0,0)
	-- ++ (0,0,0)
	-- ++ (0,-1,0) -- ++(-1,0,0)
	-- ++ (0,0,3)
	-- ++ (0,-1,0) -- ++(-1,0,0) 
	--++ (0,0,0) -- ++ (0,-3,0) --++ 
	(7,0,0)%Change 'a'
	--++(0,1,0)
	--++(0,0,-1) --++(0,1,0)
	--++(0,0,-1) --++(0,1,0)
	--++(0,0,-1) --++(0,1,0)
	--++(0,0,-1) -- cycle;

	\foreach \i in {0,...,\longdiag}{
        \draw[-] (\i-1,0,\c+\t) -- ++(0,\longdiag,0);
        \draw[-] (\a+\t,\i,0) -- ++(0,0,\longdiag);
        \draw[-] (0,\b+\t,\i) -- ++(\longdiag,0,0);
    }

    \end{tikzpicture}
    
    $R_{\bet,\del}$
    \end{multicols}
    
    %%Row 3
    \begin{multicols}{2}
    \begin{tikzpicture}[x={(0:.45cm)},y={(120:.45cm)},z={(240:.45cm)}]
	\pgfmathtruncatemacro{\a}{4}
	\pgfmathtruncatemacro{\b}{3}
	\pgfmathtruncatemacro{\c}{3}
	\pgfmathtruncatemacro{\t}{6}
	
	\pgfmathtruncatemacro{\longdiag}{\a+\b+\c+\t+\t}
	
	\node (A) at (-4.33,.33,-1){\alf};
	\node (B) at (-4.33,.33,1){\bet};
	\node (G) at (-4.33,.33,4){\gam};
	\node (D) at (-4.33,.33,6){\del};
	
	%Remove Alpha
	\fill[black!10] (-5,0,-2) --++(0,-1,0)--++(4,0,0)--++(0,1,0)--++(0,0,-1)--++(-4,0,0)-- cycle;
	%Remove Delta
	\fill[black!10](-5,0,6)--++(0,-3,0)--++(1,0,0)--++(0,3,0)--cycle;
	
	\draw[very thick] (0,0,0) --++(0,0,-1) --++(0,1,0) --++(0,0,-1) --++
	(-4,0,0)%Change 'a'
	-- ++(0,0,1)
	-- ++ (0,-1,0) -- ++(-1,0,0)
	-- ++ (0,0,2)
	-- ++ (0,-1,0) -- ++(-1,0,0)
	-- ++ (0,0,3)
	-- ++ (0,-1,0) -- ++(-1,0,0) 
	--++ (0,0,0) -- ++ (0,-3,0) --++ 
	(7,0,0)%Change 'a'
	--++(0,1,0)
	--++(0,0,-1) --++(0,1,0)
	--++(0,0,-1) --++(0,1,0)
	--++(0,0,-1) --++(0,1,0)
	--++(0,0,-1) -- cycle;
	
	\draw[clip] (0,0,0) --++(0,0,-1) --++(0,1,0) --++(0,0,-1) --++
	(-4,0,0)%Change 'a'
	-- ++(0,0,1)
	-- ++ (0,-1,0) -- ++(-1,0,0)
	-- ++ (0,0,2)
	-- ++ (0,-1,0) -- ++(-1,0,0)
	-- ++ (0,0,3)
	-- ++ (0,-1,0) -- ++(-1,0,0) 
	--++ (0,0,0) -- ++ (0,-3,0) --++ 
	(7,0,0)%Change 'a'
	--++(0,1,0)
	--++(0,0,-1) --++(0,1,0)
	--++(0,0,-1) --++(0,1,0)
	--++(0,0,-1) --++(0,1,0)
	--++(0,0,-1) -- cycle;

	\foreach \i in {0,...,\longdiag}{
        \draw[-] (\i-1,0,\c+\t) -- ++(0,\longdiag,0);
        \draw[-] (\a+\t,\i,0) -- ++(0,0,\longdiag);
        \draw[-] (0,\b+\t,\i) -- ++(\longdiag,0,0);
    }

    \end{tikzpicture}

    $R_{\alf,\del}$
    
    \columnbreak

    \begin{tikzpicture}[x={(0:.45cm)},y={(120:.45cm)},z={(240:.45cm)}]
	\pgfmathtruncatemacro{\a}{4}
	\pgfmathtruncatemacro{\b}{3}
	\pgfmathtruncatemacro{\c}{3}
	\pgfmathtruncatemacro{\t}{6}
	
	\pgfmathtruncatemacro{\longdiag}{\a+\b+\c+\t+\t}
	
	\node (A) at (-4.33,.33,-1){\alf};
	\node (B) at (-4.33,.33,1){\bet};
	\node (G) at (-4.33,.33,4){\gam};
	\node (D) at (-4.33,.33,6){\del};
	
	%Remove Alpha
	%\fill[black!10] (-5,0,-2) --++(0,-1,0)--++(4,0,0)--++(0,1,0)--++(0,0,-1)--++(-4,0,0)-- cycle;
	%Remove Delta
	%\fill[black!10](-5,0,6)--++(0,-3,0)--++(1,0,0)--++(0,3,0)--cycle;
	
	\draw[very thick] (0,0,0) --++(0,0,-1) --++(0,1,0) --++(0,0,-1) --++
	(-4,0,0)%Change 'a'
	-- ++(0,0,3)
	-- ++ (0,-1,0) -- ++(-1,0,0)
	-- ++ (0,0,0)
	-- ++ (0,-1,0) -- ++(-1,0,0)
	-- ++ (0,0,1)
	-- ++ (0,-1,0) -- ++(-1,0,0) 
	--++ (0,0,2) -- ++ (0,-3,0) --++ 
	(7,0,0)%Change 'a'
	--++(0,1,0)
	--++(0,0,-1) --++(0,1,0)
	--++(0,0,-1) --++(0,1,0)
	--++(0,0,-1) --++(0,1,0)
	--++(0,0,-1) -- cycle;
	
	\draw[clip] (0,0,0) --++(0,0,-1) --++(0,1,0) --++(0,0,-1) --++
	(-4,0,0)%Change 'a'
	-- ++(0,0,3)
	-- ++ (0,-1,0) -- ++(-1,0,0)
	-- ++ (0,0,0)
	-- ++ (0,-1,0) -- ++(-1,0,0)
	-- ++ (0,0,1)
	-- ++ (0,-1,0) -- ++(-1,0,0) 
	--++ (0,0,2) -- ++ (0,-3,0) --++ 
	(7,0,0)%Change 'a'
	--++(0,1,0)
	--++(0,0,-1) --++(0,1,0)
	--++(0,0,-1) --++(0,1,0)
	--++(0,0,-1) --++(0,1,0)
	--++(0,0,-1) -- cycle;

	\foreach \i in {0,...,\longdiag}{
        \draw[-] (\i-1,0,\c+\t) -- ++(0,\longdiag,0);
        \draw[-] (\a+\t,\i,0) -- ++(0,0,\longdiag);
        \draw[-] (0,\b+\t,\i) -- ++(\longdiag,0,0);
    }

    \end{tikzpicture}
    
    $R_{\bet,\gam}$
    \end{multicols}

    \caption{Each of the regions defined by removing an two of $\{\alpha, \beta, \gamma, \delta\}$ from $R$. Each region with $\alpha, \delta$ removed has forced lozenges along its top and southwest respectively.}
       \labelnote{fig:kuoregions2}
    \end{minipage}
\end{figure}
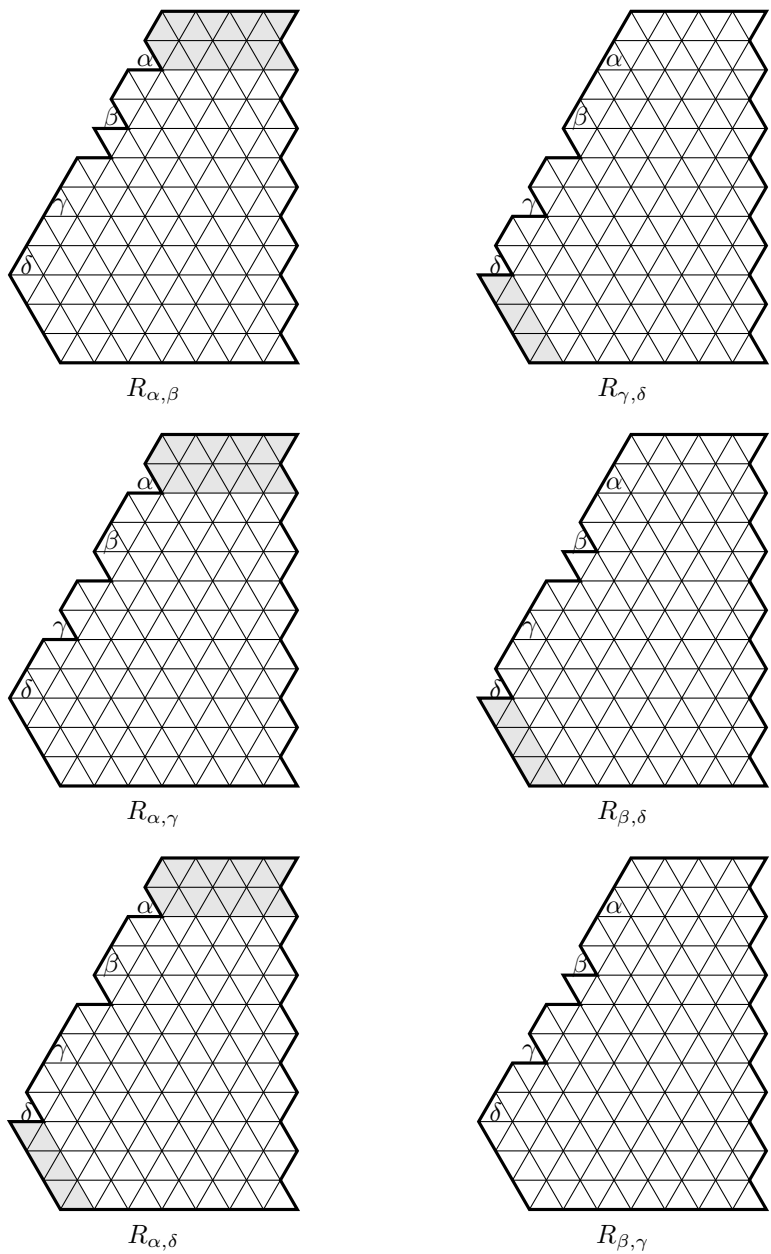
   
   Since each of these regions has forced lozenges that leave dented half-hexagons once removed (as depicted in Figure \ref{fig:kuoregions2}), we can rewrite equation \ref{eq:kuo2} in context.
   
   \begin{align}
    f_{b,n,(u_i)_1^n}(a) \cdot f_{b+1,n-2,(u_i-2)_2^{n-2}}(a+1) & =\\
    f_{b,n-1,(u_i-2)_2^n}(a+1)\cdot f_{b+1,n-1,(u_i)_1^{n-1}}(a) & -\\
    f_{b,n-1,(u_i-2)_1^{n-1}}(a+1) \cdot f_{b+1,n-1,(u_i)_2^n}(a)   & \labelnote{eq:line0}
   \end{align}
   
   It is worth noting that while the first term in line \ref{eq:line0} may refer to an untileable region, and thus may be identically zero, the rest of the terms are not identically zero. By the inductive hypothesis, we can rewrite this as follows.
   
    \begin{align*}
    f_{b,n,(u_i)_1^n}(a) \cdot F_{b+1,n-2,(u_i-2)_2^{n-2}}(a+1) & \aequiv\\
    F_{b,n-1,(u_i-2)_2^n}(a+1)\cdot F_{b+1,n-1,(u_i)_1^{n-1}}(a) & -\\
    F_{b,n-1,(u_i-2)_1^{n-1}}(a+1) \cdot F_{b+1,n-1,(u_i)_2^n}(a)   & 
   \end{align*}
   
   By relations $\ref{eq:equiv1}$ and $\ref{eq:equiv2}$, in order to show $f_{b,n,(u_i)_1^n}(a) \aequiv F_{b,n,(u_i)_1^n}(a)$ it suffices to show\footnote{As in the previous proof, in the case that the third term is identically zero it is unnecessary to show it is $\aequiv$-equivalent to the other terms, so we proceed assuming it is not identically zero. } that
   
    \begin{align}
    F_{b,n,(u_i)_1^n}(a) \cdot F_{b+1,n-2,(u_i-2)_2^{n-1}}(a+1) & \aequiv \\ 
    F_{b,n-1,(u_i-2)_2^n}(a+1)\cdot F_{b+1,n-1,(u_i)_1^{n-1}}(a) & \aequiv \\ 
    F_{b,n-1,(u_i-2)_1^{n-1}}(a+1) \cdot F_{b+1,n-1,(u_i)_2^n}(a)   &. 
   \end{align}
   
   We write these terms out explicitly.
   
    \begin{align}
    \dfrac{P_|(a,b+n)}{\prod_{i=1}^n(2a+u_i)_{\underline{u_i}}}\cdot \dfrac{P_|(a+1,b+n-1)}{\prod_{i=2}^{n-1} (2a+u_i)_{\underline{u_i}}}
    & \aequiv \\[10pt]
    \dfrac{P_|(a+1,b+n-1)}{\prod_{i=2}^n (2a+u_i)_{\underline{u_i}}}
    \cdot 
    \dfrac{P_|(a,b+n)}{\prod_{i=1}^{n-1} (2a+u_i)_{\underline{u_i}}}& \aequiv \\[10pt]
    \dfrac{P_|(a+1,b+n-1)}{\prod_{i=1}^{n-1}(2a+u_i)_{\underline{u_i}+1}}
    \cdot 
    \dfrac{P_|(a,b+n)}{\prod_{i=2}^n(2a+u_i)_{\underline{u_i}-1}} & 
   \end{align}
   
   Since the numerators of each expression are identical, it suffices to check the denominators are $\aequiv$-equivalent. We write down this condition.
   
    \begin{align}
    {\prod_{i=1}^n(2a+u_i)_{\underline{u_i}}}\cdot {\prod_{i=2}^{n-1} (2a+u_i)_{\underline{u_i}}}
    & \aequiv \labelnote{eq:line1}\\[10pt]
    {\prod_{i=2}^n (2a+u_i)_{\underline{u_i}}}
    \cdot 
    {\prod_{i=1}^{n-1} (2a+u_i)_{\underline{u_i}}}& \aequiv \labelnote{eq:line2}\\[10pt]
    {\prod_{i=1}^{n-1}(2a+u_i)_{\underline{u_i}+1}}
    \cdot 
    {\prod_{i=2}^n(2a+u_i)_{\underline{u_i}-1}} & \labelnote{eq:line3}
   \end{align}
   The products in lines \ref{eq:line1} and \ref{eq:line2} are equal. The product in line \ref{eq:line3} can be rewritten
   \begin{align}
    \prod_{i=1}^{n-1}\left((2a+b+n+i)(2a+u_i)_{\underline{u_i}}\right)
    \cdot 
    \prod_{i=2}^n\dfrac{(2a+u_i)_{\underline{u_i}-1}}{(2a+b+n+i-1)} &=\\[10pt]
    {\prod_{i=1}^{n-1}(2a+u_i)_{\underline{u_i}+1}}
    \cdot 
    {\prod_{i=2}^n(2a+u_i)_{\underline{u_i}-1}} & \labelnote{eq:line4}
   \end{align}
   to show it is also equal to the product in lines \ref{eq:line1}, \ref{eq:line2}. This completes our proof that $F_{b,n,(u_i)_1^n}(a) \aequiv f_{b,n,(u_i)_1^n}(a)$ for any tileable dented half-hexagon $V_{a,b,n,(u_i)_1^n}.$
   
   To complete our proof of Theorem \ref{main}, consider any dented half-hexagon $V_{a,b,n,(u_i)_1^n}$. If this region is untileable then equation \ref{eq:dentsfunction} holds trivially. If the region is tileable, then there exists some $c$ independent of $a$ so that
   $M(V_{a,b,n,(u_i)_1^n}) = c \cdot F_{b,n,(u_i)_1^n}(a)$. Then
   \begin{align*}
   \dfrac{f_{b,n,(u_i)_1^n}(a)}{f_{b,n,(u_i)_1^n}(0)} &= \dfrac{c \cdot F_{b,n,(u_i)_1^n}(a)}{c \cdot F_{b,n,(u_i)_1^n}(0)}
   = \dfrac{ F_{b,n,(u_i)_1^n}(a)}{ F_{b,n,(u_i)_1^n}(0)}
   = \prod_{i=1}^n (u_i)_{\underline{u_i}} \cdot \dfrac{P_|(a,b+n)}{\prod_{i=1}^n (2a+u_i)_{\underline{u_i}}}
   \end{align*}
   and equation \ref{eq:dentsfunction} follows by multiplying across by ${f_{b,n,(u_i)_1^n}(0)}$.

\end{proof}

\section{Proofs of Theorems \ref{relation} and \ref{relationgeneral}}
\labelnote{relationproof}

Recall, we denote the tiling functions $$f^+_{b,n,(u_i)_1^n}(a) := \M\left(V^+_{a,b,n,(u_i)_1^n}\right), \quad \overline f^+_{b,n,(u_i)_1^n}(a) := \M\left(\overline V^+_{a,b,n,(u_i)_1^n}\right).$$
In this section we will prove Corollary \ref{mainplus}, Theorem \ref{relation}, and Theorem \ref{relationgeneral} in sequence, as each proof depends on the previous result.

\begin{proof}[Proof of Corollary \ref{mainplus}]
	First we will use Theorem \ref{ciucu} to show
$$\overline f^+_{b,n,(u_i)_1^n}(a) \aequiv  \dfrac{P_{-}(a,b+n)}{\prod_{i=1}^n (2a+u_i)_{\underline{u_i}}}.$$
    Applying Theorem \ref{ciucu} to $H_{2a,b,b,2n,(u_i)_1^n,(u_i)_1^n}$, $R^-$ is $V_{a,b,n,(u_i)_1^n}$ whereas $R^+$ is congruent to $\overline V^+_{a,b,n,(u_i)_1^n}$. See Figure \ref{fig:partition2}, Left. Then 
    \begin{equation}\labelnote{factors}
        \M(H_{a,b,b,2n,(u_i)_1^n,(u_i)_1^n}) = 2^{b+n} \cdot f_{b,n,(u_i)_1^n}(a) \cdot \overline f^+_{b,n,(u_i)_1^n}(a).
    \end{equation}

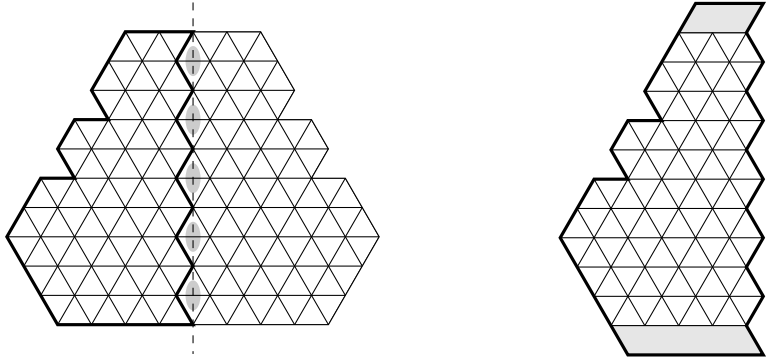
\begin{figure}
    \begin{minipage}[c]{\textwidth}
    \centering
    \begin{multicols}{2}

    \begin{tikzpicture}[x={(0:.45cm)},y={(120:.45cm)},z={(240:.45cm)}]
	\pgfmathtruncatemacro{\a}{4}
	\pgfmathtruncatemacro{\b}{3}
	\pgfmathtruncatemacro{\c}{3}
	\pgfmathtruncatemacro{\t}{4}

	\pgfmathtruncatemacro{\longdiag}{\a+\b+\c+\t+\t}
	
	\draw[very thick] (0,0,0) --++(0,0,-1) --++(0,1,0) --++(0,0,-1) --++(-2,0,0) -- ++(0,0,2)
	-- ++ (0,-1,0) -- ++(-1,0,0) -- ++ (0,0,1)
	-- ++ (0,-1,0) -- ++(-1,0,0) -- ++ (0,0,2) -- ++ (0,-3,0) --++ (4,0,0)
	--++(0,1,0)
	--++(0,0,-1) --++(0,1,0)
	--++(0,0,-1) --++(0,1,0)
	--++(0,0,-1) -- cycle;
	
	\draw[dashed] (.5,2,-2) -- (.5,-4,4);
	
	\draw[clip] (\a+\t,0,\t) -- ++(0,2,0)
	-- ++ (-1,0,0) -- ++(0,0,-1)--++(0,1,0)
	-- ++ (-1,0,0) -- ++(0,0,-1)
	-- ++ (0,2,0)
	-- ++(-\a,0,0) -- ++(0,0,2)
	-- ++ (0,-1,0) -- ++(-1,0,0) -- ++ (0,0,1)
	-- ++ (0,-1,0) -- ++(-1,0,0) -- ++ (0,0,2) --
	++(0,-\b,0) -- ++(\a+\t,0,0) -- cycle;

	\foreach \i in {0,...,\longdiag}{
        \draw[-] (\i,0,\c+\t) -- ++(0,\longdiag,0);
        \draw[-] (\a+\t,\i,0) -- ++(0,0,\longdiag);
        \draw[-] (0,\b+\t,\i) -- ++(\longdiag,0,0);
    }

		\foreach \i in {1,3,5,7,-1}{
	\fill[black, fill opacity=.2] (.5*\i,0,\i-1) ellipse (.1cm and .2cm);
	}
    \end{tikzpicture}

    \columnbreak
        
    \begin{tikzpicture}[x={(0:.45cm)},y={(120:.45cm)},z={(240:.45cm)},xscale=-1]
	\pgfmathtruncatemacro{\a}{4}
	\pgfmathtruncatemacro{\b}{3}
	\pgfmathtruncatemacro{\c}{3}
	\pgfmathtruncatemacro{\t}{4}
	
	\pgfmathtruncatemacro{\longdiag}{\a+\b+\c+\t+\t}
	
	\fill[black!10] (0,2,-2)--++(2,0,0)--++(0,-1,0)--++(-2,0,0)--cycle;
	\fill[black!10](0,-4,4)--++(4,0,0)--++(0,0,-1)--++(-4,0,0)--cycle;
	
	\draw[very thick] (0,0,0) --++(0,0,-1) --++(0,1,0) --++(0,0,-1) --++ (0,1,0)--++(2,0,0) -- ++(0,-3,0)
	-- ++ (0,0,1) -- ++(1,0,0) -- ++ (0,-1,0)
	-- ++ (0,0,1) -- ++(1,0,0) -- ++ (0,-2,0) -- ++ (0,0,4) --++ (-4,0,0) --++(0,0,-1)
	--++(0,1,0)
	--++(0,0,-1) --++(0,1,0)
	--++(0,0,-1) --++(0,1,0)
	--++(0,0,-1) -- cycle;
	
	%\draw[dashed] (.5,2,-2) -- (.5,-4,4);
	
\begin{scope}[xscale=-1]
	\draw[clip] (-1,1,-2) --++(-2,0,0) -- ++(0,0,2)
	-- ++ (0,-1,0) -- ++(-1,0,0) -- ++ (0,0,1)
	-- ++ (0,-1,0) -- ++(-1,0,0) -- ++ (0,0,2) -- ++ (0,-3,0) --++ (4,0,0)
	--++(0,0,-1) --++(0,1,0)
	--++(0,0,-1) --++(0,1,0)
	--++(0,0,-1) --++(0,1,0)--++(0,0,-1) --++(0,1,0)--++(0,0,-1) --++(0,1,0)-- cycle;

	\foreach \i in {0,...,\longdiag}{
        \draw[-] (\i,0,\c+\t) -- ++(0,\longdiag,0);
        \draw[-] (\a+\t,\i,0) -- ++(0,0,\longdiag+1);
        \draw[-] (0,\b+\t,\i) -- ++(\longdiag,0,0);
    }
\end{scope}
	
    \end{tikzpicture}
    
    \end{multicols}
    \vspace{-.5cm}
    \caption{Left: In the language of Theorem \ref{ciucu} applied to $H_{2a,b,b,2n,(u_i)_1^n,(u_i)_1^n}$, the region $R^-$ is $V_{a,b,n,(u_i)_1^n}$ and $R^+$ is $\overline V^+_{a,b,n,(u_i)_1^n}$. Right: The region $V^+_{a,b,n,(u_i)_1^n}$ is congruent to the region which remains after removing forced lozenges from $ V_{a,b+1,n,(u_i+1)_1^n}$.}
    \labelnote{fig:partition2}
    \end{minipage}
\end{figure}
  
It follows from Theorems \ref{paper1} and \ref{main} that
    \begin{equation}\labelnote{eq:weightedmain}
         \dfrac{P(2a,b+n,b+n)}{\left(\prod_{i=1}^n (2a+u_i)_{\underline{u_i}}\right)^2} \aequiv \dfrac{P_|(a,b+n)}{\prod_{i=1}^n (2a+u_i)_{\underline{u_i}}} \cdot \overline f^+_{b,n,(u_i)_1^n}(a).
    \end{equation}
    Then
    \begin{equation}
    \overline
        f^+_{b,n,(u_i)_1^n}(a) \aequiv \dfrac{P(2a,b+n,b+n)/P_|(a,b+n)}{\prod_{i=1}^n (2a+u_i)_{\underline{u_i}}} = \dfrac{P_-(a,b+n)}{\prod_{i=1}^n (2a+u_i)_{\underline{u_i}}}.
    \end{equation}
        Then
        $$\overline f^+_{b,n,(u_i)_1^n}(a) \aequiv \dfrac{P_-(a,b+n)}{\prod_{i=1}^n (2a+u_i)_{\underline{u_i}}}.$$
	The rest of the proof follows as in the proof of Theorem \ref{main}, by calculating
	$$ \dfrac{\overline f^+_{b,n,(u_i)_1^n}(a)}{\overline f^+_{b,n,(u_i)_1^n}(0)}.$$
\end{proof}

We will now prove a technical lemma to assist in the proof of Theorem \ref{relation}.

\begin{lemma} \label{weakrelation}
	For $V_{a,b,n,(u_i)_1^n}$ a tileable dented half-hexagon,
    $$\overline{f}^+_{b,n,(u_i)_1^n}(a) \aequiv f^+_{b,n,(u_i)_1^n}(a-1/2).$$
\end{lemma}

\begin{proof}
As Figure \ref{fig:partition2}, Right, demonstrates, $V^+_{a,b,n,(u_i)_1^n}$ is congruent to $V_{a,b+1,n,(u_i+1)_1^n}$ with forced lozenges removed, and therefore has the same tiling function as $V_{a,b+1,n,(u_i+1)_1^n}$.
	It follows from this fact and Theorem \ref{main} that
	$$f^+_{b,n,(u_i)_1^n}(a-1/2) = f_{b+1,n,(u_i+1)_1^n}(a-1/2) \aequiv \dfrac{P_|(a-1/2,b+n+1)}{\prod_{i=1}^n(2a+u_i)_{\underline{u_i}}}.$$
	Since
$$\overline f^+_{b,n,(u_i)_1^n}(a) \aequiv \dfrac{P_-(a,b+n)}{\prod_{i=1}^n (2a+u_i)_{\underline{u_i}}}$$	
	it suffices to show
	$P_|(a-1/2,b+1) \aequiv P_-(a,b) $. It is straightforward to check
	$$P_-(a,b) \aequiv (2a+1)_b \prod_{1 \leq i < j \leq b} (2a+i+j).$$
	We can rewrite
	\begin{align*}
	P_|(a-1/2,b+1) &\aequiv \dfrac{(a+1/2)_b}{(2a)_b} \prod_{1 \leq i \leq j \leq b} (2a+i+j-2)\\
				&= \dfrac{(a+1/2)_b}{(2a)_b} \prod_{1 \leq i \leq j \leq b} \dfrac{(2a+i+j-2)}{(2a+i+j)}(2a+i+j).
	\end{align*}
	It is straightforward to check that $$\prod_{1 \leq i \leq j \leq b} \dfrac{x+i+j-1}{x+i+j} = \prod_{k=1}^b \dfrac{x+k}{x+2k}$$
	and therefore
	$$\prod_{1 \leq i \leq j \leq b} \dfrac{2a+i+j-2}{2a+i+j} = \prod_{1 \leq i \leq j \leq b} \dfrac{(2a+i+j-2)(2a+i+j-1)}{(2a+i+j-1)(2a+i+j)} = \dfrac{(2a)_b(2a+1)_b}{(2a+1)_{2b}}.$$ 
	Then
	\begin{align*}
	P_|(a-1/2,b+1) &\aequiv \dfrac{(a+1/2)_b}{(2a)_b} \cdot \dfrac{(2a)_b(2a+1)_b}{(2a+1)_{2b}} \prod_{1 \leq i \leq j \leq b}(2a+i+j)\\
				&=  \dfrac{(a+1/2)_b}{(2a+b+1)_b} \prod_{1 \leq i < j \leq b}(2a+i+j) \prod_{1 \leq k \leq b} (2a+2k)\\
				&= \dfrac{1}{(2a+b+1)_b} \prod_{1 \leq i < j \leq b}(2a+i+j) \prod_{1 \leq k \leq b} (2a+2k)(a-1/2+k)\\
				&\aequiv \dfrac{1}{(2a+b+1)_b} \prod_{1 \leq i < j \leq b}(2a+i+j) \prod_{1 \leq k \leq b} (2a+2k)(2a+k)\\
				&= \dfrac{(2a+1)_{2b}}{(2a+b+1)_b} \prod_{1 \leq i < j \leq b}(2a+i+j)\\
				&= (2a+1)_b  \prod_{1 \leq i < j \leq b}(2a+i+j)
	\end{align*}
	which is $\aequiv$-equivalent to $P_{-}(a,b)$. 

\end{proof}

We can now prove Theorem \ref{relation} directly.

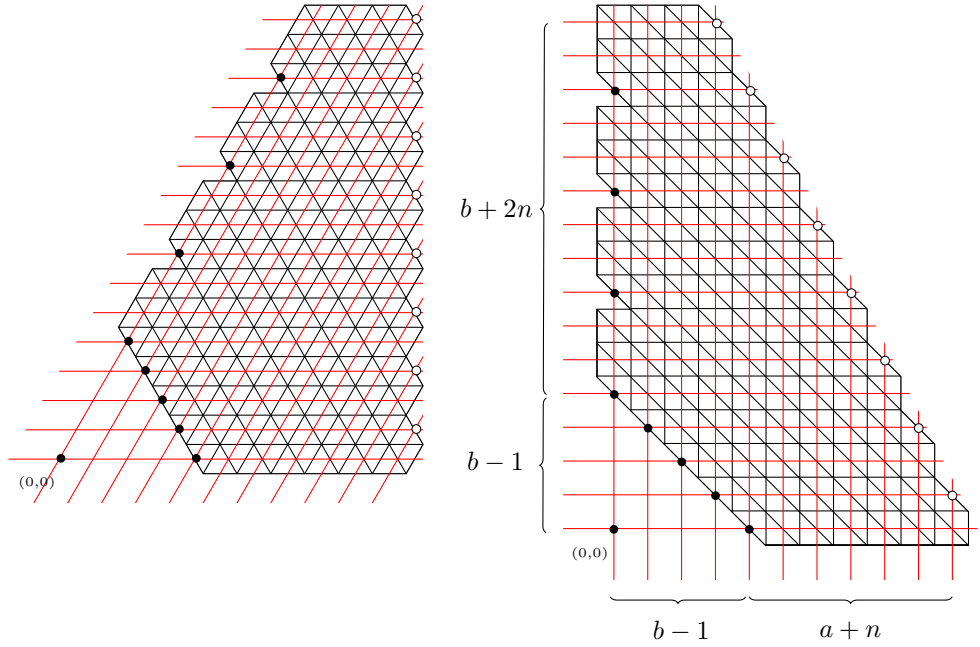
\begin{figure}
    \begin{minipage}[c]{\textwidth}
    \centering

	\begin{multicols}{2}

        %[x={(0:1cm)},y={(120:1cm)},z={(240:1cm)},scale=.45]
    \begin{tikzpicture}[x={(0:1cm)},y={(120:1cm)},z={(240:1cm)},scale=.45]
	\pgfmathtruncatemacro{\a}{6}
	\pgfmathtruncatemacro{\b}{5}
	\pgfmathtruncatemacro{\c}{5}
	\pgfmathtruncatemacro{\t}{6}
    
	\pgfmathtruncatemacro{\longdiag}{\a+\b+\c+\t}
	
	%\node (A) at (0,0,0){A};	

	\begin{scope}

	\draw[clip] (0,0,0) --++(-3,0,0)
	 -- ++(0,0,2)-- ++ (0,-1,0) -- ++(-1,0,0)
	 -- ++(0,0,2)-- ++ (0,-1,0) -- ++(-1,0,0)
	 -- ++(0,0,2)-- ++ (0,-1,0) -- ++(-1,0,0)
	 -- ++(0,0,2)-- ++ (0,-5,0) --++ (6,0,0)
	 --++(0,0,-1)
	--++ (0,1,0) --++(0,0,-1)
	--++ (0,1,0) --++(0,0,-1)
	--++ (0,1,0) --++(0,0,-1)
	--++ (0,1,0) --++(0,0,-1)
	--++ (0,1,0) --++(0,0,-1)
	--++ (0,1,0) --++(0,0,-1)
	--++ (0,1,0) --++(0,0,-1) 
	--++ (0,1,0)-- cycle;
	
	\foreach \i in {0,...,\longdiag}{
        \draw[-] (0,\a,\i) -- ++(0,-\longdiag,0);
        \draw[-] (0,\a-\i,0) -- ++(0,0,\longdiag);
        \draw[-] (0,-\i,0) -- ++(-\longdiag,0,0);
   	 }
	\end{scope}

	\begin{scope}
		\clip (.5,0,0) --++ (-4.5,0,0) --++(0,0,17) --++(12.5,0,0)  --++(0,0,-1) -- cycle;
		\foreach \i in {1,...,16}{
			\draw[red] (1,.5-.5*\i,.5*\i) --++(-14,0,0);
		}
		\foreach \i in {1,...,11}{
			\draw[red] (-3.5 + \i, 0,-.5) --++(0,0,19);
		}
	\end{scope}	

	\begin{scope}[xshift=-10.2cm,yshift=-13.4cm, node distance=.45cm,]

	\node [shape=circle,fill=black,scale=.35] (O) at (0,0){};% node[below=2pt] {\tiny (0,0)};
	\node [below left of=O] {\tiny (0,0)};

	\node [shape=circle,fill=black,scale=.35] (B) at (4,0){};
	%\node [below of=B] {\tiny ($b$-1,0)};

	\node [shape=circle,fill=black,scale=.35] (C) at (4,4){};
	%\node [left of=C] {\tiny (0,$b$-1)};

	\node [shape=circle,fill=black,scale=.35] (C1) at (4,3){};
	\node [shape=circle,fill=black,scale=.35] (C2) at (4,2){};
	\node [shape=circle,fill=black,scale=.35] (C3) at (4,1){};

	\node [shape=circle,fill=black,scale=.35] (D3) at (7,7){};
	\node [shape=circle,fill=black,scale=.35] (D2) at (10,10){};
	\node [shape=circle,fill=black,scale=.35] (D1) at (13,13){};

	\foreach \i in {1,...,8}{
		\node [shape = circle, draw=black, fill=white, scale=.35] (S\i) at (9+\i, 2*\i-2,-1){};
	}

	%\draw [decorate, decoration=brace] (-2,-.1) --++ (4,4);
	%\node (L1) at (-1.5,2) {$b-1$};

	%\draw [decorate, decoration=brace] (2,4) --++ (11,11);
	%\node (L2) at (6.5,10) {$b+2n$};

	%\draw [decorate, decoration=brace] (1.9,-2) --++ (-4,0);
	%\node (L3) at (-1,-3) {$b-1$};

	%\draw [decorate, decoration=brace] (8,-2) --++(-6,0);
	%\node (L3) at (4,-3) {$a+n$};
	\end{scope}

    \end{tikzpicture}

%%%%%%%%%%%%%%%%%%%%%%%%%%%%%%%%%%%%%%%%%%%%%%%%%%%%%%%%%%%%%%%%

        %[x={(0:1cm)},y={(120:1cm)},z={(240:1cm)},scale=.45]
    \begin{tikzpicture}[x={(0:1cm)},y={(135:1.41cm)},z={(270:1cm)},scale=.45]
	\pgfmathtruncatemacro{\a}{6}
	\pgfmathtruncatemacro{\b}{5}
	\pgfmathtruncatemacro{\c}{5}
	\pgfmathtruncatemacro{\t}{6}
    
	\pgfmathtruncatemacro{\longdiag}{\a+\b+\c+\t}
	
	%\node (A) at (0,0,0){A};	

	\begin{scope}

	\draw[clip] (0,0,0) --++(-3,0,0)
	 -- ++(0,0,2)-- ++ (0,-1,0) -- ++(-1,0,0)
	 -- ++(0,0,2)-- ++ (0,-1,0) -- ++(-1,0,0)
	 -- ++(0,0,2)-- ++ (0,-1,0) -- ++(-1,0,0)
	 -- ++(0,0,2)-- ++ (0,-5,0) --++ (6,0,0)
	 --++(0,0,-1)
	--++ (0,1,0) --++(0,0,-1)
	--++ (0,1,0) --++(0,0,-1)
	--++ (0,1,0) --++(0,0,-1)
	--++ (0,1,0) --++(0,0,-1)
	--++ (0,1,0) --++(0,0,-1)
	--++ (0,1,0) --++(0,0,-1)
	--++ (0,1,0) --++(0,0,-1) 
	--++ (0,1,0)-- cycle;
	
	\foreach \i in {0,...,\longdiag}{
        \draw[-] (0,\a,\i) -- ++(0,-\longdiag,0);
        \draw[-] (0,\a-\i,0) -- ++(0,0,\longdiag);
        \draw[-] (0,-\i,0) -- ++(-\longdiag,0,0);
   	 }
	\end{scope}

	\begin{scope}
		\clip (.5,0,0) --++ (-4.5,0,0) --++(0,0,17) --++(12.5,0,0)  --++(0,0,-1) -- cycle;
		\foreach \i in {1,...,16}{
			\draw[red] (1,.5-.5*\i,.5*\i) --++(-14,0,0);
		}
		\foreach \i in {1,...,11}{
			\draw[red] (-3.5 + \i, 0,-.5) --++(0,0,19);
		}
	\end{scope}	

	\begin{scope}[xshift=-2.5cm,yshift=-15.5cm, node distance=.45cm,]

	\node [shape=circle,fill=black,scale=.35] (O) at (0,0){};% node[below=2pt] {\tiny (0,0)};
	\node [below left of=O] {\tiny (0,0)};

	\node [shape=circle,fill=black,scale=.35] (B) at (4,0){};
	%\node [below of=B] {\tiny ($b$-1,0)};

	\node [shape=circle,fill=black,scale=.35] (C) at (4,4){};
	%\node [left of=C] {\tiny (0,$b$-1)};

	\node [shape=circle,fill=black,scale=.35] (C1) at (4,3){};
	\node [shape=circle,fill=black,scale=.35] (C2) at (4,2){};
	\node [shape=circle,fill=black,scale=.35] (C3) at (4,1){};

	\node [shape=circle,fill=black,scale=.35] (D3) at (7,7){};
	\node [shape=circle,fill=black,scale=.35] (D2) at (10,10){};
	\node [shape=circle,fill=black,scale=.35] (D1) at (13,13){};

	\foreach \i in {1,...,8}{
		\node [shape = circle, draw=black, fill=white, scale=.35] (S\i) at (9+\i, 2*\i-2,-1){};
	}

	\draw [decorate, decoration=brace] (-2.1,-.1) --++ (4,4);
	\node (L1) at (-1.5,2) {$b-1$};

	\draw [decorate, decoration=brace] (2,4) --++ (11,11);
	\node (L2) at (6,9.5) {$b+2n$};

	\draw [decorate, decoration=brace] (1.9,-2) --++ (-4,0);
	\node (L3) at (-1,-3) {$b-1$};

	\draw [decorate, decoration=brace] (8,-2) --++(-6,0);
	\node (L3) at (4,-3) {$a+n$};
	\end{scope}

	\end{tikzpicture}

	\end{multicols}

    \caption{Left: The region $V_{a,b,n,\vec u} = V_{3,5,3,(3,6,9)}$ with the east-northeast square lattice superimposed. Right: As with Figure \ref{fig:DentedHexCoordinates1} we take a linear transformation of the figure to give Cartesian coordinates for the lattice's vertices.
 Sources lie along the southwest side of the region at coordinates $\{(i-1,b-i): i \in [b]\}$ and along dents at coordinates $\{(0,2b+2n-u_i: i \in [n]\}$. Sinks lie at coordinates $\{(a+b+n-j, 2j-1): j \in [b+n]\}$.}
    \label{fig:DentedHexCoordinates}
    \end{minipage}
\end{figure}

\begin{proof}[Proof of Theorem \ref{relation}]
Recall we wish to show that
\begin{equation}\labelnote{eq:relation}
   \overline f^+_{b,n,(u_i)_1^n}(a)= f^+_{b,n,(u_i)_1^n}(a- 1/2).
\end{equation}

Lemma \ref{weakrelation} shows that $\overline f^+_{b,n,(u_i)_1^n}(a)\aequiv f^+_{b,n,(u_i)_1^n}(a- 1/2).$
We will show that the leading coefficients of $\overline f^+_{b,n,(u_i)_1^n}$ and $f^+_{b,n,(u_i)_1^n}$ are the same, and thus the polynomials are equal.

Since individual regions' tiling numbers are given by determinants of corresponding path matrices, the tiling functions $f^+_{b,n,(u_i)_1^n}(a)$  , $ 
\overline f^+_{b,n,(u_i)_1^n}(a)$ are given by determinants of appropriate path matrices whose entries are functions of $a$. Let $M$ denote the path matrix for $V^+_{a,b,n,(u_i)_1^n}$, whose $(i,j)$th entry counts paths from the $i$th source to the $j$th sink. This entry is then of the form $$\binom{a+k_{1,i,j}+k_{2,i,j}}{k_{2,i,j}}.$$
The value of $k_{1,i,j}$ can be computed by subtracting the $x$-coordinate of the the $i$th source from the $x$-coordinate of the $j$th sink, using the coordinate scheme from Figure \ref{fig:DentedHexCoordinates}. Similarly, $k_{2,i,j}$ can be computed by subtracting the $y$-coordinate of the $i$th source from the $y$-coordinate of the $j$th sink using that coordinate scheme. Thus, when there exists an east-northeast lattice path from the $i$th source to the $j$th sink, $a+k_{1,i,j}$ counts the east-directed steps on that path, and $k_{2,i,j}$ counts the northeast-directed steps.

Let $\overline M$ be the analogous path matrix for $\overline V^+_{a,b,n,(u_i)_1^n}$. Then the $(i,j)$th entry of this matrix is $$\binom{a+k_{1,i,j}+k_{2,i,j}-1}{k_{2,i,j}} + \frac 12 \binom{a+k_{1,i,j}+k_{2,i,j}-1}{k_{2,i,j}-1}.$$ This calculation separates the lattice paths according to their last step; for each lattice path in a nonintersecting family whose last step is northeast, the tiling induced by the family includes a lozenge of weight $1/2$, and the calculation above gives a weighted count of lattice paths accordingly. It is straightforward to check that when the $(i,j)$th entry from each matrix is expanded as a polynomial in $a$, the two polynomials have the same leading coefficient. It suffices to show the leading coefficients of the determinants of these matrices are also the same.

We will show for $\pi$ any permutation of $[b+n]$, that $\prod_{i=1}^{b+n}M[i, \pi(i)]$ and $\prod_{i=1}^{b+n}\overline M[i, \pi(i)]$ are either both 0, or both polynomials in $a$ of degree $$b+n + \binom{b+n}{2} - \sum_{i=1}^n \underline{u_i}.$$ 
The $(i,j)$th entry of $M$ is nonzero if $k_{2,i,j}$ is nonnegative. For any permutation $\pi$ such that each $k_{2,i,\pi(i)}$ is nonnegative, the degree of the polynomial $\prod_i M[i, \pi(i)]$ is the sum of entries $\sum_i k_{2,i,\pi(i)}$; this is the sum of the $y$-coordinates of the sinks, minus the sum of the $y$-coordinates of the sources. The sum of sink $y$-coordinates is
$$\sum_{j=1}^{b+n}(2j-1) =b+n+ 2 \sum_{j=0}^{b+n-1} j = b+n+2 \binom{b+n}{2}.$$
The sum of source $y$-coordinates is
\begin{align}
	\sum_{i=0}^{b-1}i + \sum_{i=1}^n \left(2b+2n-u_i\right)&= \sum_{i=0}^{b-1}i + \sum_{i=1}^n(b+n-i)+ \sum_{i=1}^n \left(b+n+i-u_i\right) \\
&= \binom{b}{2} + bn + n^2 - \binom {n+1}2 + \sum_{i=1}^n \underline {u_i}\\
&=\binom{b+n}{2}  + \sum_{i=1}^n \underline {u_i}.
\end{align}
The difference is then $$b+n+\binom{b+n}{2} - \sum_{i=1}^n \underline {u_i}.$$

Since this is also the degree of $f^+_{b,n,(u_i)_1^n}(a)$ and $\overline f^+_{b,n,(u_i)_1^n}(a)$, both $f^+_{b,n,(u_i)_1^n}(a)$ and $\overline f^+_{b,n,(u_i)_1^n}(a)$ must have the same leading coefficient. 

\end{proof}

We will consider how the path matrices of tubey regions relate to those of dented half-hexagons, to show that Theorem \ref{relationgeneral} follows from Theorem \ref{relation}. Toward this end we will develop some vocabulary for tubey regions. We say $R$ is the \textbf{core} of $R_{z}(a)$. We say the \textbf{height} of a tubey region is half the number of edges in $z$. Recall that the construction of a tubey region is not necessarily unique, and note that both the core and height of a region depend on its specific construction.

We include the following proof of a special case of Theorem \ref{relationgeneral}, where the number of southwest boundary edges of a simply connected tubey region is the same as its height, because we feel the proof is illustrative. This is the case where all lattice paths of lozenges in tilings of the region cross $z$ and end within the tube.

\begin{proof}[Proof of Theorem \ref{relationgeneral}, special case]
Let $R_{z}(0)$ be a simply connected tileable tubey region of height $h$, and let $g(a)$ be the tiling function for $\{R_{z}(a+1/2): a \in \mathbb N\}$ and $\overline g(a)$ be the tiling function of $\{\overline R_{z}(a+1/2): a \in \mathbb N\}$. Assume for the purposes of this special case that the number of southwest boundary edges of $R$ is $h$.

For  $a \geq h$, the southeast-directed lattice line which passes through the top of $z$ also passes through the southern boundary of the tube (rather than its eastern boundary), as depicted in Figure \ref{fig:decomposition}. We call the subregion of $R_z(a+1/2)$ west of that lattice line $W$ and the subregion east of the line $E(a)$, since $W$ is the same for each $a$ but $E(a)$ grows with $a$; we let $\overline E(a)$ denote the analogous weighted subregion of $\overline R_{z}(a+1/2)$.

\begin{figure}
    \begin{minipage}[c]{\textwidth}
    \centering

    \begin{tikzpicture}[x={(0:.45cm)},y={(120:.45cm)},z={(240:.45cm)}]
	\pgfmathtruncatemacro{\a}{10}
	\pgfmathtruncatemacro{\b}{5}
	\pgfmathtruncatemacro{\c}{4}

	\pgfmathtruncatemacro{\width}{\a+\c}

	\draw[very thick] (0,0,0)
	--++ (-\a,0,0)
	--++ (0,0,2)--++(-2,0,0)--++(0,0,1)
	--++(0,-2,0)--++(1,0,0)	--++(0,0,1)
	--++(0,-1,0)--++(-1,0,0)--++(0,0,1)--++(0,-2,0)
	--++(2+\a,0,0)
	--++(0,1,0)--++(0,0,-1)
	--++(0,1,0)--++(0,0,-1)
	--++(0,1,0)--++(0,0,-1)
	--++(0,1,0)--++(0,0,-1)
	--++(0,1,0)--++(0,0,-1);
	
	\fill[fill=gray!20] (-9,0,0)
	--++ (-1,0,0)
	--++ (0,0,2)--++(-2,0,0)--++(0,0,1)
	--++(0,-2,0)--++(1,0,0)	--++(0,0,1)
	--++(0,-1,0)--++(-1,0,0)--++(0,0,1)--++(0,-2,0)
	--++(2+1,0,0)
	--++(0,1,0)--++(0,0,-1)
	--++(0,1,0)--++(0,0,-1)
	--++(0,1,0)--++(0,0,-1)
	--++(0,1,0)--++(0,0,-1)
	--++(0,1,0)--++(0,0,-1);

	\draw[clip] (0,0,0)
	--++ (-\a,0,0)
	--++ (0,0,2)--++(-2,0,0)--++(0,0,1)
	--++(0,-2,0)--++(1,0,0)	--++(0,0,1)
	--++(0,-1,0)--++(-1,0,0)--++(0,0,1)--++(0,-2,0)
	--++(2+\a,0,0)
	--++(0,1,0)--++(0,0,-1)
	--++(0,1,0)--++(0,0,-1)
	--++(0,1,0)--++(0,0,-1)
	--++(0,1,0)--++(0,0,-1)
	--++(0,1,0)--++(0,0,-1);
	
	\foreach \X in {-\b,...,\width}{
	\draw[-] (-\X,0,0)--++(0,0,2*\b);
	\draw[-] (-\X-\b,0,0)--++(0,-2*\b,0);
	}
	\foreach \Y in {0,...,\b}{
	\draw[-] (0,-\Y-1,\Y)--++(-\width,0,0); \draw [-](0,-\Y,\Y)--++(-\width,0,0);
	}
	
	\draw[very thick, red] (-\a+1,0,0) --++(0,-2*\b,0);
	
	%\draw[very thick, blue] (-\b,0,0) --++(0,-2*\b,0);

    \end{tikzpicture}

    \caption{A tubey region whose core is shaded grey, separated into subregions $W$ (west of the line) and $E$ (east of the line) by the southeast-directed lattice line descending from the top of $z$.}
    \labelnote{fig:decomposition}
    \end{minipage}
\end{figure}
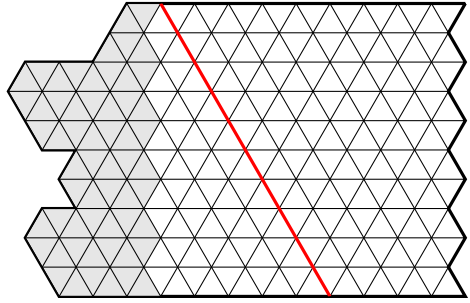

Let $N$ denote the path matrix with polynomial entries in $a$ for lattice paths starting from the southwest boundary edges of $R_{z}(a+1/2)$ and ending at the northeast boundary edges; we take the indexing of rows and columns to be increasing from sources and sinks at the bottom of the region to the top. We let $\overline N$ denote the analogous path matrix for $\overline R_{z}(a+1/2)$. We will show that these path matrices  have natural decompositions under the Cauchy-Binet formula into a linear combination of tiling functions of dented half-hexagons.

Let $N^W$ be the path matrix for east-northeast paths on $W$, and $N^{E}$ the path matrix with polynomial entries in $a$ for east-northeast paths on $E(a)$, using the same bottom-to-top indexing scheme as with $N$. Since $R_{z}(a+1/2)$ is tileable with height $h$, $N$ is an $h\times h$ matrix whereas $N^W$ is an $h\times 2h$ matrix and $N^E$ is a $2h\times h$ matrix. Furthermore,
\begin{equation}
    N = N^W \cdot N^E.
\end{equation}
By Theorem \ref{CauchyBinet} (Cauchy-Binet),
\begin{align}\labelnote{eq:CB}
    \det(N) &= \sum_{v \in \binom{[2h]}{h}} \det(N^W_{[h],v}) \cdot \det(N^E_{v,[h]})\\[10pt]
    \mbox{similarly,} \quad \det(\overline N) &= \sum_{v \in \binom{[2h]}{h}} \det(N^W_{[h],v}) \cdot \det(N^{\overline E}_{v,[h]}). \labelnote{eq:CB2}
\end{align}

In fact, $N^E_{v,[h]}$ is the path matrix for an upside-down copy of the dented half-hexagon $V^+_{a-h,0,h,\vec v}$ where $\vec v$ denotes the vector of elements of $[2h]-v$ in increasing order. This is demonstrated in Figure \ref{fig:flip}. Then

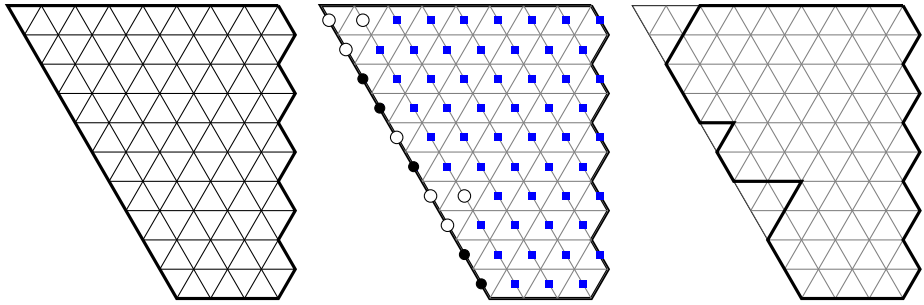
\begin{figure}
    \begin{minipage}[c]{\textwidth}
    \centering

    \begin{multicols}{3}
    
    %Left
    
    \begin{tikzpicture}[x={(0:.45cm)},y={(120:.45cm)},z={(240:.45cm)}]
	\pgfmathtruncatemacro{\a}{9}
	\pgfmathtruncatemacro{\b}{5}
	\pgfmathtruncatemacro{\c}{4}

	\pgfmathtruncatemacro{\width}{\a+\c}

	\draw[very thick] (0,0,0)
	--++ (-\a+1,0,0)
	--++ (0,-2*\b,0)
	--++(\b-2,0,0) --++(0,0,-1)
	--++(0,1,0)--++(0,0,-1)
	--++(0,1,0)--++(0,0,-1)
	--++(0,1,0)--++(0,0,-1)
	--++(0,1,0)--++(0,0,-1)
	--++(0,1,0);

	\clip  (0,0,0)
	--++ (-\a+1,0,0)
	--++ (0,-2*\b,0)
	--++(\b-2,0,0) --++(0,0,-1)
	--++(0,1,0)--++(0,0,-1)
	--++(0,1,0)--++(0,0,-1)
	--++(0,1,0)--++(0,0,-1)
	--++(0,1,0)--++(0,0,-1)
	--++(0,1,0);
	
	\foreach \X in {-\b,...,\width}{
	\draw[-] (-\X,0,0)--++(0,0,2*\b);
	\draw[-] (-\X-\b,0,0)--++(0,-2*\b,0);
	}
	\foreach \Y in {0,...,\b}{
	\draw[-] (0,-\Y-1,\Y)--++(-\width,0,0);
	 \draw [-](0,-\Y,\Y)--++(-\width,0,0);
	}

    \end{tikzpicture}
    
    \columnbreak
    
    %Middle
    
    \begin{tikzpicture}[x={(0:.45cm)},y={(120:.45cm)},z={(240:.45cm)}]
	\pgfmathtruncatemacro{\a}{9}
	\pgfmathtruncatemacro{\b}{5}
	\pgfmathtruncatemacro{\c}{4}

	\pgfmathtruncatemacro{\width}{\a+\c}

	\draw[very thick] (0,0,0)
	--++ (-\a+1,0,0)
	--++ (0,-2*\b,0)
	--++(\b-2,0,0) --++(0,0,-1)
	--++(0,1,0)--++(0,0,-1)
	--++(0,1,0)--++(0,0,-1)
	--++(0,1,0)--++(0,0,-1)
	--++(0,1,0)--++(0,0,-1)
	--++(0,1,0);

	\begin{scope}
	\clip  (0,0,0)
	--++ (-\a+1,0,0)
	--++ (0,-2*\b,0)
	--++(\b-2,0,0) --++(0,0,-1)
	--++(0,1,0)--++(0,0,-1)
	--++(0,1,0)--++(0,0,-1)
	--++(0,1,0)--++(0,0,-1)
	--++(0,1,0)--++(0,0,-1)
	--++(0,1,0);
	
	\foreach \X in {-\b,...,\width}{
	\draw[gray] (-\X,0,0)--++(0,0,2*\b);
	\draw[gray] (-\X-\b,0,0)--++(0,-2*\b,0);
	}
	\foreach \Y in {0,...,\b}{
	\draw[gray] (0,-\Y-1,\Y)--++(-\width,0,0); \draw [gray](0,-\Y,\Y)--++(-\width,0,0);
	}
	
	\foreach \y in {0,...,9}{
	\foreach \x in {0,...,\a}{
	\node[shape=rectangle,fill=blue,scale=.4] (A\x) at (-\x-1,-\y-.5,0){};
	}
	}
	\end{scope}
	
	\foreach \y in {0,...,4}{
	\node[shape=rectangle,fill=blue,scale=.4] (A) at (-1,-\y-1.5,\y-1){};
	}
	
	%The starting vertices
	\foreach \y in {3,2,5,8,9}{
	\node[shape=circle,draw=black,fill=black,scale=.4] (A) at (-8,-\y-.5,0){};
	}
	
	%The inaccessible vertices
	\node[shape=circle,draw=black,fill=white,scale=.5] (A) at (-8,0-.5,0){};
	\node[shape=circle,draw=black,fill=white,scale=.5] (A) at (-7,0-.5,0){};
	\node[shape=circle,draw=black,fill=white,scale=.5] (A) at (-8,0-1.5,0){};
	\node[shape=circle,draw=black,fill=white,scale=.5] (A) at (-8,0-4.5,0){};
	\node[shape=circle,draw=black,fill=white,scale=.5] (A) at (-8,0-6.5,0){};
	\node[shape=circle,draw=black,fill=white,scale=.5] (A) at (-8,0-7.5,0){};
	\node[shape=circle,draw=black,fill=white,scale=.5] (A) at (-7,0-6.5,0){};
	\end{tikzpicture}

    \columnbreak
    
    %Right
    
    \begin{tikzpicture}[x={(0:.45cm)},y={(120:.45cm)},z={(240:.45cm)}]
	\pgfmathtruncatemacro{\a}{9}
	\pgfmathtruncatemacro{\b}{5}
	\pgfmathtruncatemacro{\c}{4}

	\pgfmathtruncatemacro{\width}{\a+\c}
	
	\begin{scope}
	
	\draw[clip]  (0,0,0)
	--++ (-\a+1,0,0)
	--++ (0,-2*\b,0)
	--++(\b-2,0,0) --++(0,0,-1)
	--++(0,1,0)--++(0,0,-1)
	--++(0,1,0)--++(0,0,-1)
	--++(0,1,0)--++(0,0,-1)
	--++(0,1,0)--++(0,0,-1)
	--++(0,1,0);
	
	\foreach \X in {-\b,...,\width}{
	\draw[gray] (-\X,0,0)--++(0,0,2*\b);
	\draw[gray] (-\X-\b,0,0)--++(0,-2*\b,0);
	}
	\foreach \Y in {0,...,\b}{
	\draw[gray] (0,-\Y-1,\Y)--++(-\width,0,0); \draw [gray](0,-\Y,\Y)--++(-\width,0,0);
	}
	
	\end{scope}
	
	\draw[very thick] (0,0,0)--++(-6,0,0)--++(0,0,2)--++(0,-2,0)--++(1,0,0)--++(0,0,1)--++(0,-1,0) --++(2,0,0)--++(0,0,2)--++(0,-2,0)--++(3,0,0)
	 --++(0,0,-1)
	--++(0,1,0)--++(0,0,-1)
	--++(0,1,0)--++(0,0,-1)
	--++(0,1,0)--++(0,0,-1)
	--++(0,1,0)--++(0,0,-1)
	--++(0,1,0);
	
	\end{tikzpicture}
    
    \end{multicols}
    \caption{Left: $E(a)$ corresponding to $R_{z}(a)$ from Figure \ref{fig:decomposition}; note it is defined entirely by $a$ and $h$. The entries of $N^E_{v,[h]}$ count east-northeast lattice paths starting from points indexed by $v$ on the west boundary and ending at the eastern boundary. Center: For $v=\{1,2,5,7,8\}$ the vertices on the square lattice are depicted and color coded: sources are drawn as black circles; vertices accessible from these by east and northeast steps are depicted as blue squares; vertices inaccessible from the sources are drawn as white circles. Right: The region whose tilings are enumerated by families of nonintersecting lattice paths on the black and blue vertices is an upside-down copy of $V^+_{a-h,0,h,\vec v}$ with some forced lozenges removed, where the entries of $\vec v$ are $[2h]-v$ in increasing order. The outline of this region is drawn thickly over the whole region; the large triangular omissions come from forced lozenges where dents are grouped together.}
    \labelnote{fig:flip}
    \end{minipage}
\end{figure}

\begin{align*}
g(a) &= \sum_{v \in \binom{[2h]}{h}} \det(N^W_{[h],v}) \cdot f^+_{0,h,\vec v}(a-h)\\
\mbox{similarly,} \quad \overline g(a) &= \sum_{v \in \binom{[2h]}{h}} \det(N^W_{[h],v}) \cdot \overline f^+_{0,h,\vec v}(a-h),\\
\end{align*}
so that this case of Theorem \ref{relationgeneral} follows from Theorem \ref{relation}.

\end{proof}

As the determinants of path matrices do not reliably count the tilings of regions which are not simply connected, the more general case does not follow directly from an argument about path matrices. However, we may reinterpret equations \ref{eq:CB} and \ref{eq:CB2} in a way that does not refer to path matrices to get a more general proof.

\begin{proof}[Proof of Theorem \ref{relationgeneral}]

We continue using the language of the previous proof. It follows from Lemma \ref{splitting} (Region Splitting Lemma), that all tilings of $R_{z}(a+1/2)$, which has height $h$, must include exactly $h$ lozenges which cross the lattice line between $W$ and $E(a)$. Where $v$ indexes the possible positions of this set of lozenges, indexing along the border from bottom to top, we thus have
\begin{equation}\labelnote{eq:CBnoCB}
    M(R_{z}(a+1/2)) = \sum_{v \in \binom{2h}{h}} M(W_v)M(E_v(a))
\end{equation}
where $W_v$ is the subregion of $W$ omitting those triangles covered by lozenges in positions indexed by $v$, and $E_v(a)$ is the analogous subregion of $E(a)$. {In fact, $M(E_v(a))= \det(N^E_{v,[h]})$, and under the conditions of the previous proof $M(W_v) = \det(N^W_{[h],v})$ so that equation \ref{eq:CBnoCB} generalizes \ref{eq:CB}.} As before, $M(E_v(a))$ is the tiling function for a dented half-hexagon, and the rest of the proof follows as in the previous special case.

\end{proof}

We remark that Theorem \ref{relationgeneral} then holds for families of tubey regions where the core may be quite strange in shape and weighting scheme, though the tube should be unweighted except for the vertical lozenges along its eastern boundary which have weight 1/2.

\section{Final Remarks}

While writing this paper we solved some problems, not yet mentioned, with uninteresting answers. Seeking to generalize Theorem \ref{relationgeneral}, one might consider tubey regions with vertical lozenges along the eastern boundary having weight $q$, denoting these $\overline R_{z}(a;q)$ so that $\overline R_{z}(a) = \overline R_{z}(a;1/2)$. Defining $g(a;q)$ to be the tiling function for $\{\overline R_{z}(a+1/2;q): a \in \mathbb N\}$, one might then ask whether there is any relation
$$g(a+d;1) = g(a;q),$$
for $R$ general and $d$ independent of $a$, generalizing Theorem \ref{relationgeneral}. In fact, no such relation exists except for $q \in \{0, 1/2,1\}$, even when $R$ is restricted to vertically symmetric dented half-hexagons.

Moving forward we are interested in regions that can be constructed in a similar manner to tubey regions, in the sense that one can add (possibly weighted) layers to a region. 

\subsection*{Acknowledgements}

The author would like to express many thanks to M. Ciucu for his advice on this problem and this paper.


\begin{bibdiv}
\begin{biblist}

\bib{An78}{article}{
title={Plane partitions (1): The MacMahon conjecture},
author={Andrews, G.},
journal={Adv. in Math. Suppl. Stud.},
volume = {1},
pages={131-150},
date={1978}
}
\bib{Br}{book}{
title={A Comprehensive Introduction to Linear Algebra},
author={Broida, J. G.},
author={Williamson, S. G.},
publisher={Addison-Wesley Publishing Company},
address={Redwood City, California},
date={1989}
}

\bib{By19}{article}{
title={A short proof of two shuffling theorems for tilings and a weighted generalization},
author={Byun, S. H.},
status={accepted},
eprint={arXiv:1906.04533 [math.CO]},
journal = {Discrete Mathematics},
volume = {345},
number = {3},
date={2022}
}

\bib{Ci97}{article}{
title={Enumeration of Perfect Matchings in Graphs with Reflective Symmetry},
author={Ciucu, M.},
journal={J. Combin. Theory Ser. A},
date={1997},
Volume={77},
number={1},
pages={67-97}
}

\bib{Ci01}{article}{
title={Enumeration of Lozenge Tilings of Hexagons with a Central Triangular Hole},
author={Ciucu, M.},
author={Eisenk{\"o}bl, T.},
author={Krattenthaler, C.},
author={Zare, D.},
journal={J. Combin. Theory Ser. A},
date={2001},
volume={95},
pages={ 251-334}
}

\bib{Ci13}{article}{
author={Ciucu, M.},
author={Fischer, I.},
title={Proof of two conjectures of Ciucu and Krattenthaler on the enumeration of lozenge tilings of hexagons with cut off corners},
journal={J. Combin. Theory Ser. A}
volume = {133},
date={2013},
pages={228–250}
}

\bib{Ci14}{article}{
title={Proof of Blum's Conjecture on Hexagonal Dungeons},
author={Ciucu, M.},
author={Lai, T.},
journal={J. Combin. Theory Ser. A},
date={2014},
volume={125},
number={1},
pages={}
}

\bib{Ci96}{article}{
title={Lozenge tilings of hexagons with arbitrary dents},
author={Ciucu, M.},
author={Fischer, I.},
journal={Adv. Appl. Math.},
volume = {73 C},
date={2016},
pages={1-22}
}

\bib{Ci19a}{article}{
title={Tilings of hexagons with a removed triad of bowties},
author={Ciucu, M.},
author={Lai, T.},
author={Rohatgi, R.},
date={2021},
journal={J. Combin. Theory Ser. A},
Volume={178},
pages={105359}
}

\bib{Ci19b}{article}{
title={Lozenge tilings of doubly-intruded hexagons},
author={Ciucu, M.},
author={Lai, T.},
date={2019},
journal={J. Combin. Theory Ser. A},
Volume={167},
pages={294-339}
}

\bib{Ci21}{article}{
title={A New Solution for the Two Dimensional Dimer Problem},
author={Ciucu, M.},
status={submitted},
eprint={ arXiv:2102.07229v1 [math.CO]},
date={2021}
}

\bib{Co}{article}{
title={Lozenge Tiling Function Ratios for Hexagons with Dents on Two Sides},
author={Condon, D.},
date={2019},
journal={Electron. J. Comb.},
Volume={27},
number={3},
pages={3.60}
}

\bib{Ei}{article}{
title={Rhombus Tilings of a hexagon with three fixed border tiles},
author={Eisenk{\"o}lbl, T.},
date={1999},
journal={J. Combin. Theory Ser. A},
Volume={88},
pages={368-378}
}

\bib{Fu}{article}{
title = {Generating Functions of Lozenge Tilings for Hexagonal Regions via
Nonintersecting Lattice Paths},
author={Fulmek, M.},
journal = {Enumer. Combin. Appl.},
volume = {1},
number = {3},
date = {2021},
pages = {S2R24}
}

\bib{Ge85}{article}{
title={Binomial determinants, paths, and hook length formulae. },
author={Gessel, I.},
author={Viennot, X.},
journal={Adv. in Math.},
date={1985},
volume={58},
number={3},
pages={300-321}
}

\bib{Ge89}{article}{
title={Determinants, paths, and plane partitions },
author={Gessel, I.},
author={Viennot, X.},
date={1989},
status={preprint}
}

\bib{Gi}{article}{
title={Inverting the Kasteleyn matrix for holey hexagons},
author={Gilmore, T.},
status={submitted},
eprint={ arXiv:1701.07092v2 [math.CO]},
date={2017}
}

\bib{Go83}{article}{
title={A proof of the Bender-Knuth conjecture},
author={Gordon, B.},
journal={Pacific J. Math.},
date={1983},
volume={108},
pages={99-113}
}

\bib{Ka}{article}{
title={Coincidence probabilities},
author={Karlin, S.},
author={McGregor, J.},
journal={Pacific J. Math.},
Volume={9},
date={1959},
pages={1141-1164}
}

\bib{Ku04}{article}{
title={Applications of Graphical Condensation for Enumerating Matchings and Tilings},
author={Kuo, E.},
journal={Theoret. Comput. Sci.},
date={2004},
Volume={319},
number={1-3},
pages={29-57}}
%https://arxiv.org/abs/math/0304090

\bib{La19a}{article}{
title={A shuffling theorem for reflectively symmetric tilings},
author={Lai, T.},
journal = {Discrete Mathematics},
date={2021},
volume={344},
number={7},
pages={112390}
}

\bib{La19b}{article}{
title={A shuffling theorem for centrally symmetric tilings},
author={Lai, T.},
status={submitted},
eprint={arXiv:1906.03759 [math.CO]},
date={2019}
}

\bib{La19c}{article}{
title={A shuffling theorem for lozenge tilings of doubly-dented hexagons},
author={Lai, T.},
author={Rohatgi, R.},
status={submitted},
eprint={arXiv:1905.08311 [math.CO]},
date={2019}
}

\bib{La21}{article}{
title={Ratio of Tiling Generating Functions of semi-Hexagons and Quartered Hexagons
with Dents},
author = {Lai, T.},
journal = {Enumer. Combin. Appl.},
volume = {2},
number = {1},
pages={S2R5}
}

\bib{Li}{article}{
title={On  the  vector  representations  of  induced  matroids},
author={Lindstr{\"o}m, B.},
journal={Bull. Lond. Math. Soc.},
volume={5},
date={1973},
pages={85-90}
}

\bib{MD}{book}{
title = {Symmetric Functions and Hall Polynomials},
publisher = {Oxford Univ. Press, London/New York},
author = {Macdonald, I.},
date = {1979},
pages={52}
}

\bib{Ma}{book}{
title={Combinatory Analysis},
volume={2},
author={MacMahon, P. A.},
publisher={Cambridge University Press},
date={1916},
pages={12},
reprint={
title={Combinatory Analysis},
volume={1-2},
author={MacMahon, P. A.},
publisher={Chelsea, New York},
date={1960}
}
}

\bib{Pr84}{article}{
author = {Proctor, R.},
title = {Bruhat lattices, plane partition generating functions, and miniscule representations},
journal = {European J. Combin.},
volume =  {5},
date = {1984},
pages = {331-350}
}

\bib{Pr88}{article}{
author = {Proctor, R.},
title = {Odd symplectic groups},
journal = {Inventiones Mathematicae},
volume =  {92},
date = {1988},
pages = {307-332}
}

\bib{TeFi}{article}{
title={Dimer problem in statistical mechanics - an exact result},
journal={Philos. Mag. (8)},
volume={6},
author={Temperley, H. N. V.},
author={Fisher, M. E.},
date={1961},
pages={1061-1063}
}

\bib{Ya}{article}{
title={Graphical condensation of plane graphs: A combinatorial approach},
journal={Theoret. Comput. Sci.},
volume={349},
author={W.G. Yan},
author={Y.-N. Yeh},
author={F.J.Zhang},
date={2005},
pages={452-461}
}

\end{biblist}
\end{bibdiv}
\end{document}